\documentclass[10pt]{amsart}
\usepackage{amsmath}
\usepackage{amsxtra}
\usepackage{amscd}
\usepackage{amsthm}
\usepackage{amsfonts}
\usepackage{amssymb}
\usepackage{eucal}
\usepackage[all]{xy}
\usepackage{graphicx}
\usepackage[usenames]{color}

\newtheorem{cor}[subsubsection]{Corollary}
\newtheorem{lem}[subsubsection]{Lemma}
\newtheorem{prop}[subsubsection]{Proposition}

\newtheorem{thm}[subsubsection]{Theorem}

\newtheorem{defn}[subsubsection]{Definition}
\newtheorem{quest}[subsubsection]{Question}

\theoremstyle{remark}
\newtheorem{rem}[subsubsection]{Remark}
\newtheorem{example}[subsubsection]{Example}

\theoremstyle{definition}

\theoremstyle{remark}

\newcommand{\thmref}[1]{Theorem~\ref{#1}}

\newcommand{\secref}[1]{Sect.~\ref{#1}}
\newcommand{\lemref}[1]{Lemma~\ref{#1}}
\newcommand{\propref}[1]{Proposition~\ref{#1}}
\newcommand{\corref}[1]{Corollary~\ref{#1}}

\newcommand{\remref}[1]{Remark~\ref{#1}}

\numberwithin{equation}{section}

\newcommand{\nc}{\newcommand}
\nc{\renc}{\renewcommand}
\nc{\ssec}{\subsection}
\nc{\sssec}{\subsubsection}
\nc{\on}{\operatorname}

\newcommand{\iso}{\buildrel{\sim}\over{\longrightarrow}}
\newcommand{\mono}{\hookrightarrow}
\newcommand{\epi}{\twoheadrightarrow}

\nc{\ips}{{\iota_P^{(S)}}}
\nc{\ipms}{{\iota_{P^-}^{(S)}}}
\nc{\sfpps}{{\sfp_P^{(S)}}}
\nc{\sfppms}{{\sfp_{P^-}^{(S)}}}

\nc\ol{\overline}
\nc\wt{\widetilde}
\nc\tboxtimes{\wt{\boxtimes}}
\nc\tstar{\wt{\star}}
\nc{\alp}{\alpha}

\nc{\ZZ}{{\mathbb Z}}
\nc{\NN}{{\mathbb N}}
\nc{\OO}{{\mathbb O}}
\renc{\SS}{{\mathbb S}}
\nc{\DD}{{\mathbb D}}
\nc{\GG}{{\mathbb G}}

\nc{\Fq}{{\mathbb F}_q}
\nc{\Fqb}{\ol{{\mathbb F}_q}}
\nc{\Ql}{\ol{{\mathbb Q}_\ell}}
\nc{\id}{\text{id}}
\nc\X{\mathcal X}

\nc{\red}{\on{red}}
\nc{\Ho}{\on{Ho}}
\nc{\Hom}{\on{Hom}}
\nc{\Mor}{\on{Mor}}
\nc{\coef}{\on{coeff}}
\nc{\Lie}{\on{Lie}}
\nc{\Loc}{\on{Loc}}
\nc{\Pic}{\on{Pic}}
\nc{\Bun}{\on{Bun}}
\nc{\IC}{\on{IC}}
\nc{\Aut}{\on{Aut}}
\nc{\rk}{\on{rk}}
\nc{\Sh}{\on{Sh}}
\nc{\Perv}{\on{Perv}}
\nc{\pos}{{\on{pos}}}
\nc{\Conv}{\on{Conv}}
\nc{\Sph}{\on{Sph}}
\nc{\Sym}{\on{Sym}}
\nc{\BunBb}{\overline{\Bun}_B}
\nc{\BunNb}{\overline{\Bun}_N}
\nc{\BunTb}{\overline{\Bun}_T}
\nc{\BunBbm}{\overline{\Bun}_{B^-}}
\nc{\BunBbel}{\overline{\Bun}_{B,el}}
\nc{\BunBbmel}{\overline{\Bun}_{B^-,el}}
\nc{\Buno}{\overset{o}{\Bun}}
\nc{\BunPb}{{\overline{\Bun}_P}}
\nc{\BunBM}{\Bun_{B(M)}}
\nc{\BunBMb}{\overline{\Bun}_{B(M)}}
\nc{\BunPbw}{{\widetilde{\Bun}_P}}
\nc{\BunBP}{\widetilde{\Bun}_{B,P}}
\nc{\GUb}{\overline{G/U}}
\nc{\GUPb}{\overline{G/U(P)}}

\nc{\PP}{\underline{P}'}

\nc{\Hhom}{\underline{\on{Hom}}}
\nc\syminfty{\on{Sym}^{\infty}}
\nc\lal{\ol{\lambda}}
\nc\xl{\ol{x}}
\nc\thl{\ol{\theta}}
\nc\nul{\ol{\nu}}
\nc\mul{\ol{\mu}}
\nc\Sum\Sigma
\nc{\oX}{\overset{o}{X}{}}
\nc{\hl}{\overset{\leftarrow}h{}}
\nc{\hr}{\overset{\rightarrow}h{}}
\nc{\M}{{\mathcal M}}
\nc{\N}{{\mathcal N}}
\nc{\F}{{\mathcal F}}
\nc{\D}{{\mathcal D}}
\nc{\Q}{{\mathcal Q}}
\nc{\Y}{{\mathcal Y}}
\nc{\G}{{\mathcal G}}
\nc{\E}{{\mathcal E}}
\nc{\CalC}{{\mathcal C}}
\nc\Dh{\widehat{\D}}

\nc{\C}{{\mathcal C}}
\nc{\K}{{\mathcal K}}
\renewcommand{\H}{{\mathcal H}}

\nc{\T}{{\mathcal T}}
\nc{\V}{{\mathcal V}}
\renc{\P}{{\mathcal P}}
\nc{\A}{{\mathcal A}}
\nc{\B}{{\mathcal B}}
\nc{\U}{{\mathcal U}}

\nc{\Gr}{{\on{Gr}}}

\nc{\frn}{{\check{\mathfrak u}(P)}}

\nc{\fC}{\mathfrak C}
\nc{\p}{\mathfrak p}
\nc{\q}{\mathfrak q}
\nc\f{{\mathfrak f}}

\nc{\qo}{{\mathfrak q}}
\nc{\po}{{\mathfrak p}}
\nc{\s}{{\mathfrak s}}
\nc\w{\text{w}}

\newcommand{\mmu}{{\mu}}

\nc\mathi\iota
\nc\Spec{\on{Spec}}
\nc\Proj{\on{Proj}}
\nc\Mod{\on{Mod}}
\nc{\tw}{\widetilde{\mathfrak t}}
\nc{\pw}{\widetilde{\mathfrak p}}
\nc{\qw}{\widetilde{\mathfrak q}}
\nc{\jw}{\widetilde j}

\nc{\grb}{\overline{\Gr}}
\nc{\I}{\mathcal I}

\nc{\lambdach}{{\check\lambda}}
\nc{\Lambdach}{{\check\Lambda}{}}
\nc{\much}{{\check\mu}}
\nc{\omegach}{{\check\omega}}
\nc{\nuch}{{\check\nu}}
\nc{\etach}{{\check\eta}}
\nc{\alphach}{{\check\alpha}}
\nc{\oblvtach}{{\check\oblvta}}
\nc{\rhoch}{{\check\rho}}
\nc{\ch}{{\check h}}

\nc{\Hb}{\overline{\H}}

\nc{\BA}{{\mathbb{A}}}
\nc{\BC}{{\mathbb{C}}}
\nc{\BG}{{\mathbb{G}}}
\nc{\BM}{{\mathbb{M}}}
\nc{\BO}{{\mathbb{O}}}
\nc{\BD}{{\mathbb{D}}}
\nc{\BBD}{{\mathbf{D}}}
\nc{\BN}{{\mathbb{N}}}
\nc{\BP}{{\mathbb{P}}}
\nc{\BQ}{{\mathbb{Q}}}
\nc{\BR}{{\mathbb{R}}}
\nc{\BZ}{{\mathbb{Z}}}
\nc{\BS}{{\mathbb{S}}}
\nc{\Deep}{{\bf{deep}}}
\nc{\deep}{deep}

\nc{\CA}{{\mathcal{A}}}
\nc{\CB}{{\mathcal{B}}}

\nc{\CE}{{\mathcal{E}}}
\nc{\CF}{{\mathcal{F}}}
\nc{\CH}{{\mathcal{H}}}

\nc{\CL}{{\mathcal{L}}}
\nc{\CC}{{\mathcal{C}}}
\nc{\CG}{{\mathcal{G}}}
\nc{\CalD}{{\mathcal{D}}}
\nc{\CM}{{\mathcal{M}}}
\nc{\CN}{{\mathcal{N}}}
\nc{\CK}{{\mathcal{K}}}
\nc{\CO}{{\mathcal{O}}}
\nc{\CP}{{\mathcal{P}}}
\nc{\CQ}{{\mathcal{Q}}}
\nc{\CR}{{\mathcal{R}}}
\nc{\CS}{{\mathcal{S}}}
\nc{\CT}{{\mathcal{T}}}
\nc{\CU}{{\mathcal{U}}}
\nc{\CV}{{\mathcal{V}}}
\nc{\CW}{{\mathcal{W}}}
\nc{\CX}{{\mathcal{X}}}
\nc{\CY}{{\mathcal{Y}}}
\nc{\CZ}{{\mathcal{Z}}}
\nc{\CI}{{\mathcal{I}}}

\nc{\csM}{{\check{\mathcal A}}{}}
\nc{\oM}{{\overset{\circ}{\mathcal M}}{}}
\nc{\obM}{{\overset{\circ}{\mathbf M}}{}}
\nc{\oCA}{{\overset{\circ}{\mathcal A}}{}}
\nc{\obA}{{\overset{\circ}{\mathbf A}}{}}
\nc{\ooM}{{\overset{\circ}{M}}{}}
\nc{\osM}{{\overset{\circ}{\mathsf M}}{}}
\nc{\vM}{{\overset{\bullet}{\mathcal M}}{}}
\nc{\nM}{{\underset{\bullet}{\mathcal M}}{}}
\nc{\oD}{{\overset{\circ}{\mathcal D}}{}}
\nc{\obD}{{\overset{\circ}{\mathbf D}}{}}
\nc{\oA}{{\overset{\circ}{\mathbb A}}{}}
\nc{\op}{{\overset{\bullet}{\mathbf p}}{}}
\nc{\cp}{{\overset{\circ}{\mathbf p}}{}}
\nc{\oU}{{\overset{\bullet}{\mathcal U}}{}}
\nc{\oZ}{{\overset{\circ}{\mathcal Z}}{}}
\nc{\ofZ}{{\overset{\circ}{\mathfrak Z}}{}}
\nc{\oF}{{\overset{\circ}{\fF}}}

\nc{\fa}{{\mathfrak{a}}}
\nc{\fb}{{\mathfrak{b}}}
\nc{\fd}{{\mathfrak{d}}}
\nc{\ff}{{\mathfrak{f}}}
\nc{\fg}{{\mathfrak{g}}}
\nc{\fgl}{{\mathfrak{gl}}}
\nc{\fh}{{\mathfrak{h}}}
\nc{\fj}{{\mathfrak{j}}}
\nc{\fl}{{\mathfrak{l}}}
\nc{\fm}{{\mathfrak{m}}}
\nc{\fn}{{\mathfrak{n}}}
\nc{\fu}{{\mathfrak{u}}}
\nc{\fp}{{\mathfrak{p}}}
\nc{\fr}{{\mathfrak{r}}}
\nc{\fs}{{\mathfrak{s}}}
\nc{\ft}{{\mathfrak{t}}}
\nc{\fz}{{\mathfrak{z}}}
\nc{\fsl}{{\mathfrak{sl}}}
\nc{\hsl}{{\widehat{\mathfrak{sl}}}}
\nc{\hgl}{{\widehat{\mathfrak{gl}}}}
\nc{\hg}{{\widehat{\mathfrak{g}}}}
\nc{\chg}{{\widehat{\mathfrak{g}}}{}^\vee}
\nc{\hn}{{\widehat{\mathfrak{n}}}}
\nc{\chn}{{\widehat{\mathfrak{n}}}{}^\vee}

\nc{\fA}{{\mathfrak{A}}}
\nc{\fB}{{\mathfrak{B}}}
\nc{\fD}{{\mathfrak{D}}}
\nc{\fE}{{\mathfrak{E}}}
\nc{\fF}{{\mathfrak{F}}}
\nc{\fG}{{\mathfrak{G}}}
\nc{\fK}{{\mathfrak{K}}}
\nc{\fL}{{\mathfrak{L}}}
\nc{\fM}{{\mathfrak{M}}}
\nc{\fN}{{\mathfrak{N}}}
\nc{\fP}{{\mathfrak{P}}}
\nc{\fU}{{\mathfrak{U}}}
\nc{\fV}{{\mathfrak{V}}}
\nc{\fZ}{{\mathfrak{Z}}}

\nc{\bb}{{\mathbf{b}}}
\nc{\bc}{{\mathbf{c}}}
\nc{\bd}{{\mathbf{d}}}
\nc{\bbf}{{\mathbf{f}}}
\nc{\be}{{\mathbf{e}}}
\nc{\bi}{{\mathbf{i}}}
\nc{\bj}{{\mathbf{j}}}
\nc{\bn}{{\mathbf{n}}}
\nc{\bp}{{\mathbf{p}}}
\nc{\bq}{{\mathbf{q}}}
\nc{\bu}{{\mathbf{u}}}
\nc{\bv}{{\mathbf{v}}}
\nc{\bx}{{\mathbf{x}}}
\nc{\bs}{{\mathbf{s}}}
\nc{\by}{{\mathbf{y}}}
\nc{\bw}{{\mathbf{w}}}
\nc{\bA}{{\mathbf{A}}}
\nc{\bK}{{\mathbf{K}}}
\nc{\bB}{{\mathbf{B}}}
\nc{\bC}{{\mathbf{C}}}
\nc{\bG}{{\mathbf{G}}}
\nc{\bD}{{\mathbf{D}}}
\nc{\bH}{{\mathbf{H}}}
\nc{\bM}{{\mathbf{M}}}
\nc{\bN}{{\mathbf{N}}}
\nc{\bV}{{\mathbf{V}}}
\nc{\bW}{{\mathbf{W}}}
\nc{\bX}{{\mathbf{X}}}
\nc{\bZ}{{\mathbf{Z}}}
\nc{\bS}{{\mathbf{S}}}

\nc{\sA}{{\mathsf{A}}}
\nc{\sB}{{\mathsf{B}}}
\nc{\sC}{{\mathsf{C}}}
\nc{\sD}{{\mathsf{D}}}
\nc{\sF}{{\mathsf{F}}}
\nc{\sG}{{\mathsf{G}}}
\nc{\sK}{{\mathsf{K}}}
\nc{\sM}{{\mathsf{M}}}
\nc{\sO}{{\mathsf{O}}}
\nc{\sW}{{\mathsf{W}}}
\nc{\sQ}{{\mathsf{Q}}}
\nc{\sP}{{\mathsf{P}}}
\nc{\sZ}{{\mathsf{Z}}}
\nc{\sfp}{{\mathsf{p}}}
\nc{\bsfp}{{\mathsf{\bar p}_P}}
\nc{\sfq}{{\mathsf{q}}}
\nc{\sr}{{\mathsf{r}}}
\nc{\bk}{{\mathsf{k}}}
\nc{\sg}{{\mathsf{g}}}
\nc{\sff}{{\mathsf{f}}}
\nc{\sfb}{{\mathsf{b}}}
\nc{\sfc}{{\mathsf{c}}}
\nc{\sd}{{\mathsf{d}}}

\nc{\BK}{{\bar{K}}}

\nc{\tA}{{\widetilde{\mathbf{A}}}}
\nc{\tB}{{\widetilde{\mathcal{B}}}}
\nc{\tg}{{\widetilde{\mathfrak{g}}}}
\nc{\tG}{{\widetilde{G}}}
\nc{\TM}{{\widetilde{\mathbb{M}}}{}}
\nc{\tO}{{\widetilde{\mathsf{O}}}{}}
\nc{\tU}{{\widetilde{\mathfrak{U}}}{}}
\nc{\TZ}{{\tilde{Z}}}
\nc{\tx}{{\tilde{x}}}
\nc{\tbv}{{\tilde{\bv}}}
\nc{\tfP}{{\widetilde{\mathfrak{P}}}{}}
\nc{\tz}{{\tilde{\zeta}}}
\nc{\tmu}{{\tilde{\mu}}}

\nc{\urho}{\underline{\rho}}
\nc{\uB}{\underline{B}}
\nc{\uC}{{\underline{\mathbb{C}}}}
\nc{\ui}{\underline{i}}
\nc{\uj}{\underline{j}}
\nc{\ofP}{{\overline{\mathfrak{P}}}}
\nc{\oB}{{\overline{\mathcal{B}}}}
\nc{\og}{{\overline{\mathfrak{g}}}}
\nc{\oI}{{\overline{I}}}

\nc{\eps}{\varepsilon}
\nc{\hrho}{{\hat{\rho}}}

\nc{\one}{{\mathbf{1}}}
\nc{\two}{{\mathbf{t}}}

\nc{\Rep}{{\mathop{\operatorname{\rm Rep}}}}
\nc{\Tot}{{\mathop{\operatorname{\rm Tot}}}}
\nc{\Ker}{{\mathop{\operatorname{\rm Ker}}}}
\nc{\im}{{\mathop{\operatorname{\rm Im}}}}
\nc{\Hilb}{{\mathop{\operatorname{\rm Hilb}}}}
\nc{\End}{{\mathop{\operatorname{\rm End}}}}
\nc{\Ext}{{\mathop{\operatorname{\rm Ext}}}}
\nc{\CHom}{{\mathop{\operatorname{{\mathcal{H}}\it om}}}}
\nc{\GL}{{\mathop{\operatorname{\rm GL}}}}
\nc{\gr}{{\mathop{\operatorname{\rm gr}}}}
\nc{\HN}{{\mathop{\operatorname{\rm HN}}}}
\nc{\Id}{{\mathop{\operatorname{\rm Id}}}}
\nc{\de}{{\mathop{\operatorname{\rm def}}}}
\nc{\length}{{\mathop{\operatorname{\rm length}}}}
\nc{\supp}{{\mathop{\operatorname{\rm supp}}}}

\nc{\Cliff}{{\mathsf{Cliff}}}
\nc{\Fl}{\on{Fl}}
\nc{\Fib}{{\mathsf{Fib}}}
\nc{\Coh}{{\on{Coh}}}
\nc{\QCoh}{{\on{QCoh}}}
\nc{\IndCoh}{{\on{IndCoh}}}
\nc{\FCoh}{{\mathsf{FCoh}}}

\nc{\reg}{{\text{\rm reg}}}

\nc{\cplus}{{\mathbf{C}_+}}
\nc{\cminus}{{\mathbf{C}_-}}
\nc{\cthree}{{\mathbf{C}_*}}
\nc{\Qbar}{{\bar{Q}}}
\nc\Eis{\on{Eis}}
\nc\Eisb{\ol\Eis{}}
\nc\Eisr{\on{Eis}^{rat}{}}
\nc\wh{\widehat}
\nc{\Def}{\on{Def_{\check{\fb}}(E)}}
\nc{\barZ}{\overline{Z}{}}
\nc{\barbarZ}{\overline{\barZ}{}}
\nc{\barpi}{\overline\pi}
\nc{\barbarpi}{\overline\barpi}
\nc{\barpip}{\overline\pi{}^+}
\nc{\barpim}{\overline\pi{}^-}

\nc{\fq}{\mathfrak q}

\nc{\fqb}{\ol{\sfq}{}}
\nc{\fpb}{\ol{\sfp}{}}
\nc{\fpr}{{\sfp^{rat}}{}}
\nc{\fqr}{{\sfq^{rat}}{}}

\nc{\hattimes}{\wh\otimes}

\nc{\bh}{{\bar{h}}}
\nc{\bOmega}{{\overline{\Omega(\check \fn)}}}

\nc{\seq}[1]{\stackrel{#1}{\sim}}

\nc{\cT}{{\check{T}}}
\nc{\cG}{{\check{G}}}
\nc{\cM}{{\check{M}}}
\nc{\cB}{{\check{B}}}

\nc{\ct}{{\check{\mathfrak t}}}
\nc{\cg}{{\check{\fg}}}
\nc{\cb}{{\check{\fb}}}
\nc{\cn}{{\check{\fn}}}

\nc{\cLambda}{{\check\Lambda}}

\nc{\cla}{{\check\lambda}}
\nc{\cmu}{{\check\mu}}
\nc{\cnu}{{\check\nu}}
\nc{\ceta}{{\check\eta}}

\nc{\DefbE}{{\on{Def}_{\cB}(E_\cT)}}

\nc{\imathb}{{\ol{\imath}}}
\nc{\rlr}{\overset{\longrightarrow}{\underset{\longrightarrow}\longleftarrow}}

\nc{\oBun}{\overset{\circ}\Bun}
\nc{\LocSys}{\on{LocSys}}
\nc{\BunBbb}{\ol{\ol{Bun}}_B}
\nc{\BunBr}{\Bun_B^{rat}}
\nc{\BunBrsg}{\Bun_B^{rat,\on{s.g.}}}
\nc{\BunBrp}{\Bun_B^{rat,polar}}
\nc{\BunBrpbg}{\Bun_B^{rat,polar,\on{b.g.}}}
\nc{\BunBrpsg}{\Bun_B^{rat,polar,\on{s.g.}}}
\nc{\BunTrp}{\Bun_T^{rat,polar}}
\nc{\BunTrpbg}{\Bun_T^{rat,polar,\on{b.g.}}}
\nc{\BunTrpsg}{\Bun_T^{rat,polar,\on{s.g.}}}
\nc{\BunNr}{\Bun_N^{rat}}
\nc{\BunNre}{\Bun_N^{enh,rat}}
\nc{\BunTr}{\Bun_T^{rat}}
\nc{\Vect}{\on{Vect}}
\nc{\Whit}{\on{Whit}}
\nc{\CTb}{\ol{\on{CT}}}
\nc{\Ran}{\on{Ran}}
\nc{\CTr}{\on{CT}^{rat}{}}
\nc\jmathr{\jmath^{rat}{}}
\nc{\ux}{\underline{x}}
\nc{\clambda}{{\check\lambda}}
\nc{\calpha}{{\check\alpha}}
\nc{\ind}{{\mathbf{ind}}}
\nc{\oblv}{{\mathbf{oblv}}}
\nc{\ox}{{\overline{x}}}
\nc{\cLa}{\check{\Lambda}}
\nc{\StinftyCat}{\on{DGCat}}
\nc{\inftyCat}{\infty\on{-Cat}}
\nc{\inftygroup}{\infty\on{-Grpd}}
\nc{\Dmod}{\on{D-mod}}
\nc{\CMaps}{{\mathcal Maps}}
\nc{\Maps}{\on{Maps}}
\nc{\affSch}{\on{Sch}^{\on{aff}}}
\nc{\dr}{{\on{dR}}}
\nc{\rD}{{\blacktriangle}}
\nc{\oCY}{\overset{\circ}\CY}
\nc{\leqG}{\underset{G}\leq}
\nc{\leqM}{\underset{M}\leq}
\nc{\leqGad}{\underset{G_{ad}}\leq}
\nc{\leqMad}{\underset{M_{ad}}\leq}
\nc{\psId}{\on{Ps-Id}}

\nc{\sotimes}{\overset{!}\otimes}

\begin{document}

\title[The category of D-modules on $\Bun_G$]{Compact generation of the category of D-modules on the stack of $G$-bundles on a curve}

\author{V.~Drinfeld and D.~Gaitsgory}

\date{\today}

\begin{abstract}
Let $G$ be a reductive group. Let $\Bun_G$ denote the stack of $G$-bundles on a smooth complete curve
over a field of characteristic 0, and let $\Dmod(\Bun_G)$ denote the DG category of D-modules on $\Bun_G$. 
The main goal of the paper is to show that $\Dmod(\Bun_G)$ is compactly generated (this is not automatic because 
$\Bun_G$ is not quasi-compact). The proof is based on the following observation: $\Bun_G$ can be written as a 
union of quasi-compact open substacks $j:U\hookrightarrow\Bun_G$, which are "co-truncative", 
i.e., the functor $j_!$ is defined on the entire category $\Dmod(U)$. 
\end{abstract}

\maketitle

\tableofcontents

\section*{Introduction}

\ssec{The main result}

Let $k$ be an algebraically closed field of characteristic $0$.
Let $X$ be a smooth complete connected curve over $k$ and let $\Bun_G$ denote
the moduli stack of principal $G$-bundles on $X$, where $G$ is a connected reductive group.

\sssec{}

The object of study of this paper is the DG category $\Dmod(\Bun_G)$ of D-modules on $\Bun_G$.
Our main goal is to prove the following theorem:

\begin{thm} \label{t:premain}
The DG category $\Dmod(\Bun_G)$ is compactly generated.
\end{thm}

For the reader's convenience we will review the theory of DG categories, and the notion
of compact generation in \secref{s:DG review}. 

\medskip

Essentially, the property of compact generation is what makes a DG category manageable. 

\sssec{}

The above theorem is somewhat surprising for the following reason. 

\medskip

It is known that if an algebraic stack  $\CY$ is quasi-compact and the automorphism group of every
field-valued point of $\CY$ is affine, then the DG category $\Dmod(\CY)$ is compactly generated. 
This result is established in \cite[Theorem 0.2.2]{finiteness}. In fact, the compact generation of $\Dmod(\CY)$ 
for most stacks $\CY$ that one encounters in practice is much easier than the above-mentioned 
theorem of \cite{finiteness}: it is nearly obvious for stacks of the form $Z/H$, where $Z$ is a 
quasi-compact scheme and $H$ an algebraic group acting on it. 

\medskip

However, if $\CY$ is not quasi-compact then $\Dmod(\CY)$ does not have to be compactly generated. 
We will exhibit two such examples in \secref{ss:counterexamples}; in both of them $\CY$ will 
actually be a smooth non quasi-compact scheme (non-separated in the first example, 
and separated in the second one). 

\sssec{}

So the compact generation of $\Dmod(\CY)$ encodes a certain geometric property of the stack $\CY$.
We do not know how to formulate a necessary and sufficient condition for $\Dmod(\CY)$ to be
compactly generated. 

\medskip

But we do formulate a sufficient condition, which we call ``truncatibility" (see \secref{sss:trunc} or 
Definition \ref{defn:truncatable}).
The idea is that $\CY$ is truncatable if it can be represented as a union of quasi-compact open
substacks $U$ so that for each of them direct image functor
$$j_*:\Dmod(U)\to \Dmod(\CY)$$
has a particularly nice property explained below.

\ssec{Truncativeness, co-truncativeness and truncatability}

\sssec{}

Let $\CY$ be a quasi-compact algebraic stack with affine automorphism groups of points, 
and let $\CZ\overset{i}\hookrightarrow \CY$ be a closed embedding. By \cite[Theorem 0.2.2]{finiteness},
both categories $\Dmod(\CZ)$ and $\Dmod(\CY)$ are compactly generated.

\medskip

We have a pair of adjoint functors
$$i_{\dr,*}:\Dmod(\CZ)\rightleftarrows \Dmod(\CY):i^!.$$

Being a left adjoint, the functor $i_{\dr,*}$ preserves compactness. But there is no reason for $i^!$ to have this property. 
We will say that $\CZ$ is \emph{truncative} in $\CY$ if $i^!$ does preserve compactness. 

\medskip

Truncativeness is a purely ``stacky" phenomenon. In \secref{sss:non-truncative sch} we will show that it never occurs for 
schemes, unless $\CZ$ is a union of connected components of $\CY$. 

\medskip

Let $U\overset{j}\hookrightarrow \CY$ be the embedding of the complementary open substack. We say that
$U$ is \emph{co-truncative} in $\CY$ if $\CZ$ is truncative. This property can be reformulated as saying that
the functor 
$$j_*:\Dmod(U)\to \Dmod(\CY)$$
preserves compactness. We show that the property of co-truncativeness can be also reformulated as the
existence of the functor $j_!:\Dmod(U)\to \Dmod(\CY)$, \emph{left} adjoint to the restriction functor $j^*$.
(A priori, $j_!$ is only defined on the holonomic subcategory.)

\begin{rem}
The property of being compact for an object in $\Dmod(\CY)$ is somewhat subtle (e.g., it is not local
in the smooth topology). In \secref{ss:trunc via coh} we reformulate the notion of truncativeness
and co-truncativeness in terms of the more accessible property of \emph{coherence} instead
of compactness.
\end{rem}

\sssec{}    \label{sss:trunc}

Let us now drop the assumption that $\CY$ be quasi-compact. We say that a closed substack $\CZ$
(resp., open substack $U$) is truncative (resp., co-truncative), if for every quasi-compact open $\oCY\subset \CY$,
the intersection $\CZ\cap \oCY$ (resp., $U\cap \oCY$) is truncative (resp., co-truncative) in $\oCY$.

\medskip

We say that $\CY$ is \emph{truncatable} if it %is covered by
equals the union of its quasi-compact co-truncative open substacks.
%can be written as a union of quasi-compact open substacks that are co-truncative. 
We will show that a union of two co-truncative open substacks is co-truncative. So $\CY$ is 
truncatable if and only if every open quasi-compact substack of $\CY$ is contained in one which
is co-truncative.

%\medskip
%
%Equivalently, $\CY$ is truncatable if the poset of quasi-compact open substacks of $\CY$ 
%(with the order relation given by inclusion) contains a cofinal subset whose elements are 
%co-truncative.

\medskip

We will show (see \propref{p:truncatable cg}) that 
{\em if $\CY$ is truncatable, then $\Dmod(\CY)$ is compactly generated}.
 (This is an easy consequence of \cite[Theorem 0.2.2]{finiteness}.)

\sssec{}   \label{sss:how to cover}
Thus Theorem~\ref{t:premain}, follows from the next statement, which is
the main technical result of this paper.

\begin{thm} \label{t:main}
The stack $\Bun_G$ is truncatable.
\end{thm} 

%This theorem is proved in \secref{s:delo}. The proof amounts to explicitly producing a
%family of quasi-compact open co-truncative substacks of $\Bun_G$. These substacks are
%written down as certain unions of the strata of the Harder-Narasimhan-Shatz stratification
%of $\Bun_G$ according to the degree of instability of a 
%bundle (see \thmref{t:reduction theory}). 

Let us explain how to cover $\Bun_G$ by quasi-compact open co-truncative substacks.
For every dominant rational coweight $\theta$ let $\Bun_G^{(\le \theta)}\subset \Bun_G$
denote the open substack parameterizing $G$-bundles whose Harder-Narasimhan coweight
\footnote{This rational coweight was defined by Harder-Narasimhan \cite{HN} in the case $G=GL(n)$
and by A.~Ramanathan \cite{R1} for any $G$.}
is $\leqG \theta$ (the partial ordering $\leqG$ on coweights is  defined as usual: 
$\lambda_1 \leqG \lambda_2$ if 
$\lambda_2-\lambda_1$ is a linear combination of simple coroots with non-negative coefficients).
Equivalently, $\Bun_G^{(\leq \theta)}$
%$\Bun_G^{(\leqG \theta)}$ 
parameterizes those $G$-bundles
$\CP_G$ that have the following property: for every reduction $\CP_B$ to the Borel, 
the degree of $\CP_B$ (which is a coweight of $G$) is $\leqG \theta$. 

\medskip

The substacks 
$\Bun_G^{(\leq \theta)}$ are quasi-compact and cover $\Bun_G\,$. So 
Theorem~\ref{t:main} is a consequence of the following fact proved in Sect.~\ref{s:delo}: 

\medskip

\noindent {\em The substack 
$\Bun_G^{(\leq \theta)}$ is co-truncative if for every simple root $\check\alpha_i$ one has
\begin{equation}  \label{e:what_is_deep}
%\theta\in (2g-2)\rho +\Lambda_{G}^{+,\BQ}\, ,
\langle \theta \, ,\check\alpha_i\rangle \geq 2g-2,
\end{equation}
where $g$ is the genus of $X$.} 

\medskip

% and $\rho$ is the sum of the fundamental coweights of $G$.
E.g., if $G=GL(2)$ this means that the open substack 
$$\Bun_{GL_2}^{(\leq m)}\cap \Bun^n_{GL_2}\subset \Bun_{GL_2}$$
that parameterizes rank 2 vector bundles of degree $n$ all of whose line sub-bundles have degree 
$\le m$, is co-truncative provided that $2m-n\ge 2g-2$.

\medskip

Condition~\eqref{e:what_is_deep} means that $\theta$ is ``deep enough" inside the dominant chamber
(of course, if $g\le 1$ then the condition holds for any dominant $\theta$).
%$\theta\in\Lambda_{G}^{+,\BQ}$). 

\sssec{Establishing truncativeness}   \label{sss:establishing}

To prove \thmref{t:main}, we will have to show that certain explicitly defined locally closed
substacks of $\Bun_G$ are truncative. 

\medskip

We will do this by using a ``contraction principle", see \propref{p:2contraction principle}. 
In its simplest form, it says that the substack  
$\{ 0\}/\BG_m\hookrightarrow \BA^n/\BG_m$ is truncative (here $\BG_m$ acts on 
$\BA^n$ by homotheties).  

\ssec{Duality}

\sssec{}

Recall the notion of dualizability of a DG category in the sense of Lurie (see \secref{sss:dualizable category}).
Any compactly generated DG category is automatically dualizable. In particular, such is $\Dmod(\CY)$
when $\CY$ is a truncatable algebraic stack.

\sssec{}

However, more is true. As we recall in \secref{sss:Verdier}, if $\CY$ is quasi-compact, not only
is the category $\Dmod(\CY)$ dualizable, but Verdier duality defines an equivalence
$$\Dmod(\CY)^\vee\simeq \Dmod(\CY).$$

\medskip

It is natural to ask for a description of the dual category $\Dmod(\CY)^\vee$ when $\CY$ is no longer
quasi-compact, but just truncatable. 

\sssec{}

As we will see in \secref{ss:category as colimit}-\ref{ss:relation_with_dual}, the category $\Dmod(\CY)^\vee$
%, denoted by $\Dmod(\CY)_{\on{co}}$, 
can be described explicitly, but it is a priori \emph{different} from $\Dmod(\CY)$. 

\medskip

There exists a naturally defined functor $$\bD^{\on{Verdier}}_{\CY,\on{naive}}:\Dmod(\CY)^\vee\to \Dmod(\CY),$$
but we show (see \propref{p:idnaive}) that this functor is \emph{not} an equivalence unless the closure
of every quasi-compact open in $\CY$ is again quasi-compact.  

\sssec{}
However, in Sect.~\ref{sss:better} we define a less obvious functor 
$$\bD^{\on{Verdier}}_{\CY,!}:\Dmod(\CY)^\vee\to \Dmod(\CY),$$ which may differ from 
$\bD^{\on{Verdier}}_{\CY,\on{naive}}$ even for $\CY$ quasi-compact. 

\medskip

In general, $\bD^{\on{Verdier}}_{\CY,!}$ is not an equivalence, but there are important and nontrivial examples of quasi-compact and 
non-quasi-compact stacks $\CY$ for which $\bD^{\on{Verdier}}_{\CY,!}$ is an equivalence.

\medskip

In particular, in a subsequent publication\footnote{For a draft see \cite{self-duality}.}
it will be shown  that \emph{the functor $\bD^{\on{Verdier}}_{\CY,!}$ is an equivalence if 
$\CY=\Bun_G$, where $G$ is any reductive group. } 

\medskip

Thus, for any reductive $G$, \emph{the DG category $\Dmod(\Bun_G)$ 
identifies with its dual} (in a non-trivial way and for non-trivial reasons).

\ssec{Generalizations and open questions}

Let us return to the main result of this paper, namely, \thmref{t:premain}.

\sssec{}

In the situation of Quantum Geometric Langlands, one needs to consider the categories of 
twisted D-modules on $\Bun_G$. The corresponding analog of \thmref{t:premain}, with 
the same proof, holds in this more general context. 

\sssec{} \label{sss:parabolics}

Let $x_1,\ldots,x_n\in X$. Instead of $\Bun_G\,$, consider the stack of $G$-bundles on $X$
with a reduction to a parabolic $P_i$ at $x_i$, $1\le i\le n$. Most probably, in this situation 
an analog of \thmref{t:premain} holds and can be proved in a similar way.

\sssec{}

Suppose now that  instead of reductions to parabolics (as in \secref{sss:parabolics}), one considers
deeper level structures at $x_1,\ldots,x_n$
(the simplest case being reduction to the unipotent radical of the Borel). 

\medskip

We do not know whether an analog of \thmref{t:premain} holds in this case, and we do not know what to expect.  
In any case, our strategy of the proof of \thmref{t:premain} fails in this context.

\sssec{}

Here are some more questions:

\begin{quest}
Does the assertion of \thmref{t:premain} (and its strengthening, \thmref{t:main})
hold for $\CY$ being one of the stacks $\BunBb$, $\Bun_P$, $\BunPb$ and $\wt{\Bun}_P$, where $B$ is the Borel, and
$P$ a general parabolic?
\end{quest}

\medskip

We are quite confident that the answer is ``yes" for $\BunBb$, but are less sure in other cases.

\medskip

\begin{quest} Does the assertion of \thmref{t:premain} hold for an \emph{arbitrary} connected affine
algebraic group $G$ (i.e., without the assumption that $G$ be reductive)?
\end{quest}

\ssec{Organization of the paper}

\sssec{}

In \secref{s:DG review} we review some basic facts regarding DG categories. 

\sssec{}

In \secref{s:D-modules} we review some general facts about the category of D-modules
on an algebraic stack $\CY$.
We first consider the case when $\CY$ is quasi-compact and make a summary of the relevant 
results from \cite{finiteness}. 
We then consider the case when $\CY$ is not quasi-compact and characterize the subcategory
of $\Dmod(\CY)$ formed by compact objects.

\sssec{}

In \secref{s:turncativeness}, we introduce some of the main definitions for this paper:
the notions of truncativeness (for a locally closed substack) and co-truncativeness 
(for an open substack).
We study the behavior of these notions under morphisms, base change, refinement
of stratification, etc. We also discuss the ``non-standard"
functors associated to a truncative closed (or locally closed) substack (see Sects. \ref{sss:voprosial} and 
Remark~\ref{r:non-standard_locally_closed}), in particular, the very unusual functors $i_?$ and $j^?$.

\sssec{} \secref{s:truncatable} is, philosophically, the heart of this article.

\medskip

In Sect.~\ref{ss:truncatable} we introduce the notion of truncatable stack. 
We show that if $\CY$ is truncatable then the category $\Dmod(\CY)$ is compactly generated.
In particular, we obtain that \thmref{t:main} implies \thmref{t:premain}.

\medskip 

In Sects.~\ref{ss:category as colimit}--\ref{ss:miracoli} we discuss the behavior of Verdier duality on 
truncatable stacks and the relation beween the category $\Dmod(\CY)$ and its dual.

\sssec{}
In \secref{ss:est trunc} we formulate a \emph{contraction principle}, see  \propref{p:2contraction principle}.
It shows that a closed substack with the property that we call \emph{contractiveness} is truncative. 

\medskip

In \secref{ss:adjointness} we explicitly describe the non-standard functors $i^*$ and $i_?$ in the setting of 
Proposition~\ref{p:2contraction principle}.

\sssec{}

In \secref{ss:SL 2}
%At the end of \secref{s:truncatable} 
we prove \thmref{t:main} in the particular case of $G=SL_2$. The proof
in the general case follows the same idea, but is more involved combinatorially. 

\sssec{}

In \secref{s:reduction theory} we recall the %Harder-Narasihman-Shatz 
stratification of $\Bun_G$ according to the Harder-Narasihman coweight of the $G$-bundle.  
We briefly indicate a way to establish the existence of such a stratification using the relative 
compactification of the map $\Bun_P\to \Bun_G$. 

\sssec{}

In \secref{s:compl red} we introduce a book-keeping device that allows to produce locally closed
substacks of $\Bun_G$ from locally closed substacks of $\Bun_M$, where $M$ is a Levi subgroup 
of $G$. Certain locally closed substacks of $\Bun_G$ obtained in this way, will turn out to be
\emph{contractive}, and hence \emph{truncative}, and as such will play a crucial role in the
proof of \thmref{t:main}.

\sssec{}      \label{sss:proof_amounts_to}

In \secref{s:delo}--\ref{s:constr contr}  we finally prove \thmref{t:main}. The proof amounts to combining the 
Harder-Narasimhan-Shatz strata of $\Bun_G$ (i.e., the strata corresponding to a fixed value of the Harder-Narasihman coweight) 
into certain larger locally closed substacks 
and applying the contraction principle. A more detailed explanation of the 
idea of the proof can be found in  \secref{ss:outline of proof}.

%\sssec{}  
In \secref{s:delo} we prove \thmref{t:main} modulo  a key \propref{p:key}. The latter is proved in
\secref{s:estimates}--\ref{s:constr contr}.

%In \secref{s:estimates} we single out which unions of Harder-Narasihman-Shatz strata
%are contractive by proving some estimates on cohomology vanishing. 

%\sssec{}

%In \secref{s:constr contr} we prove the contractiveness of the above unions of strata. 

\sssec{}

In \secref{s:counterexamples} we prove the existence of non quasi-compact stacks $\CY$ such that
the category $\Dmod(\CY)$ is not compactly generated.

\medskip

Namely, we show that 
if $\CY=Y$ is a smooth scheme containing a non quasi-compact divisor, then the category
$\Dmod(\CY)$ is not generated by compact objects. 
More precisely, we show that (locally) coherent D-modules on $Y$ that belong to the full
subcategory generated by compact objects cannot have all of $T^*(Y)$ as their singular
support. In particular, the D-module $\CalD_Y$ does not belong to the subcategory.

\sssec{}
In Appendix \ref{s:preordered} we recall an explicit description of open, closed, and locally closed subsets of a
(pre)-ordered set equipped with its natural topology. We use this description (combined with the  Harder-Narasihman map) 
to explicitly construct some locally closed substacks of $\Bun_G\,$, see \secref{sss:who_is_who} and \corref{c:order topology}.

\sssec{}
In Appendix~\ref{s:Langlands} we give a variant of the proof of Theorem~\ref{t:2cotruncative} that
has some advantages compared with the one from \secref{ss:proof_modulo}. The method is to define a coarsening of the 
Harder-Narasimhan-Shatz stratification such that each stratum is contractive (and therefore truncative).
This is done using the \emph{Langlands retraction} of the space of rational coweights onto the dominant cone.

\sssec{}
In Appendix~\ref{s:stacky_contraction} we prove a ``stacky" generalization of the contraction principle from
Sect.~\ref{ss:elem contr} and of the adjunction from \propref{p:adjointness}.

%\sssec{}
%The material on the Langlands retraction is reviewed  in  \cite[Appendix~A]{old_version}.

%recall the Langlands retraction from the set of all (co)weights onto the dominant cone. We show how this map 
%allows to single out certain unions of Harder-Narasihman-Shatz strata that are contractive. 

%\sssec{}
%In Appendix ~\ref{s:loc and coloc} we prove a general lemma on localizations and colocalizations
%in $\infty$-categories.

\ssec{Conventions and notation}  \label{ss:conventions}

\sssec{}

Our conventions on $\infty$-categories follow those of \cite[Sect. 0.6.1]{finiteness}. 
Whenever we say ``category", by default we mean an $(\infty,1)$-category.
We denote by $\inftyCat$ the $(\infty,1)$-category of $\infty$-categories. 

\medskip

We denote by $\inftygroup\subset \inftyCat$ in the $(\infty,1)$-subcategory spanned by 
$\infty$-groupoids, a.k.a., spaces. We denote by $\bC\mapsto \bC^{\on{grpd}}$ the
functor $\inftyCat\to \inftygroup$ right adjoint to the above embedding. Explicitly, 
$\bC^{\on{grpd}}$ is obtained from $\bC$ by discarding non-invertible 1-morphisms. 

\medskip

For $\bC\in \inftyCat$ and objects $\bc_1,\bc_2\in \bC$ we denote by 
$\Maps_\bC(\bc_1,\bc_2)\in \inftygroup$ the corresponding space of 
maps. We let $\Hom_{\bC}(\bc_1,\bc_2)$ denote the set
$\pi_0\left(\Maps_\bC(\bc_1,\bc_2)\right)$.

\sssec{Schemes and stacks}

This paper deals with categorical aspects of the category of D-modules, i.e., we do not need derived 
algebraic geometry for this paper. Therefore, by a \emph{scheme} we shall understand a \emph{classical scheme}. 
We let $\on{Sch}$ (resp., $\affSch$) denote the category of schemes (resp., affine schemes) over $k$, and $\on{Sch}_{\on{lft}}$
(resp., $\affSch_{\on{ft}}$) its full subcategory
consisting of affine schemes locally of finite type (resp., affine schemes of finite type). 

\medskip

By a prestack we shall mean an arbitrary functor $(\affSch)^{\on{op}}\to \inftygroup$. 

\medskip

By a stack 
we shall mean a prestack that satisfies the fppf descent condition. For the general notion of Artin
stack we refer the reader to \cite[Sect. 4.2]{Stacks}. However, neither general stacks
nor Artin stacks are necessary for this paper. What we need is the more restricted (and standard) 
notion of \emph{algebraic stack}. We adopt the following conventions: a stack $\CY$ is said to be
an algebraic stack if:

\begin{itemize}

\item The diagonal morphism $\CY\to \CY\times \CY$ is schematic, quasi-compact and quasi-separated;

\item There exists a scheme $Z$ equipped with a morphism $f:Z\to \CY$ (this morphism is automatically
schematic, by the previous condition) such that $f$ is smooth and surjective.

\end{itemize} 

The pair $(Z,f)$ is called a \emph{presentation} or \emph{atlas} of $\CY$. 

\medskip

We note that this definition is slightly more restrictive than the one in \cite[Sect. 4.2.8]{Stacks}. 

\sssec{Finite type(ness)}

All schemes, algebraic stacks and prestacks considered in this paper will be \emph{locally of finite type} over $k$. 

\medskip

We recall that a classical prestack, i.e., a functor $(\affSch)^{\on{op}}\to \inftygroup$,
is said to be locally of finite type if it takes limits in $\affSch$ to colimits in $\inftygroup$. Equivalently, 
a classical prestack is locally of finite type if it is the left Kan extension from the full subcategory
$\affSch_{\on{ft}}\subset \affSch$. The upshot is that when considering prestacks locally of finite type,
one can forget about all affine schemes altogether and restrict one's attention to $\affSch_{\on{ft}}$. 

\medskip

An algebraic stack is said to be locally of finite type if it is such when considered as a prestack. This is equivalent 
to requiring that it admit an atlas $(Z,f)$ with $Z$ being locally of finite type. Or, still equivalently, that for any
$Z\in \on{Sch}$ equipped with a smooth map to $\CY$, the scheme $Z$ is of finite type. The equivalence
of these conditions is established, e.g., in \cite[Proposition 4.9.2]{Stacks}.

\sssec{D-modules}

We refer the reader to the paper \cite{Crys} for the theory of D-modules (a.k.a. crystals)
on prestacks locally of finite type. 

\medskip

For a morphism $f:\CY_1\to \CY_2$ of prestacks we have a tautologically defined functor
$$f^!:\Dmod(\CY_2)\to \Dmod(\CY_1).$$ This functor may or may not have a left adjoint,
which we denote by $f_!$. 

\medskip

If $f$ is schematic\footnote{Recall that $f$ is said to be schematic if $\CY_1\underset{\CY_2}\times S$ is a scheme for any scheme $S$ 
equipped with a morphism $S\to \CY_2\,$.} and quasi-compact, we also have a functor of direct image
$$f_{\dr,*}:\Dmod(\CY_1)\to \Dmod(\CY_2).$$

\medskip

However, when $f$ is an open embedding, we will use the notation $j_*$ instead of
$j_{\dr,*}$, and $j^*$ instead of $j^!$, for reasons of tradition. This is not supposed
to cause confusion, as the above functors go to the same-named functors for the
underlying $\CO$-modules. 

\ssec{Acknowledgements}  

The research of V. D. is partially supported by NSF grant DMS-1001660. The research of
D. G. is partially supported by NSF grant DMS-1063470. We thank R.~Bezrukavnikov for
drawing our attention to Langlands' article  \cite{La}. We are grateful to S.~Schieder for
comments on the previous version of the paper.

\section{DG categories}  \label{s:DG review}

Sects.~\ref{ss:DGcat_setting}-\ref{ss:DGlim} are devoted to recollections and conventions regarding DG categories. 
In Sects.~\ref{ss:colimits abs}-\ref{ss:Colimcompgen} we provide a categorical framework for
Sects. \ref{ss:category as colimit}-\ref{ss:descr dual}; this material can definitely be skipped until it is used.

\ssec{The setting}   \label{ss:DGcat_setting}

\sssec{}

Throughout this paper we will work with DG categories over the ground field $k$. We refer the reader to \cite{DG} for
a survey. \footnote{Whenever we talk about a DG category $\bC$, we will always assume
that it is \emph{pre-triangulated}, which by definition means that $\on{Ho}(\bC)$ is \emph{triangulated}.}

\medskip

We let $\Vect$ denote the DG category of chain complexes of $k$-vector spaces. 

\medskip

We let $\StinftyCat$ denote the $\infty$-category of all DG categories. 
\footnote{We will ignore set-theoretic
issues; however, the reader can assume that all DG categories and functors are \emph{accessible} 
in the sense of \cite[Sect. 5.4.2]{Lu1}.}

\sssec{Cocomplete DG categories}

Our basic object of study is the $(\infty,1)$-category $\StinftyCat_{\on{cont}}$ whose objects
are cocomplete DG categories (i.e., ones that contain arbitrary direct sums, or equivalently,
colimits), and where $1$-morphisms are continuous functors (i.e., exact functors that commute
with arbitrary direct sums, or equivalently all colimits). 

\medskip

The construction of $\StinftyCat_{\on{cont}}$ as an $(\infty,1)$-category has not been fully 
documented. A pedantic reader can replace $\StinftyCat_{\on{cont}}$ by the equivalent 
$(\infty,1)$-category of stable $\infty$-categories tensored over $k$, whose construction
is a consequence of \cite[Sects. 4.2 and 6.3]{Lu2}.

\medskip

We have a forgetful functor $\StinftyCat_{\on{cont}}\to \StinftyCat$ that induces
an isomorphism on $2$-morphisms and higher.

\sssec{Terminological deviation \em{(i)}} We will sometimes encounter non-cocomplete DG categories 
(e.g., the subcategory of compact objects in a given DG category). Every time that this happens, we will say
so explicitly.

\sssec{}

The category $\StinftyCat_{\on{cont}}$ has a natural symmetric monoidal structure given
by Lurie's tensor product, denoted by $\otimes$ (see \cite[Sect. 6.3]{Lu2} or \cite[Sect. 1.4]{DG} for a brief review). 

\medskip

Its unit object is the category $\Vect$ of chain complexes of $k$-vector spaces. 

\sssec{Functors}

For $\bC_1,\bC_2\in \StinftyCat_{\on{cont}}$ we will denote by $\on{Funct}_{\on{cont}}(\bC_1,\bC_2)$
their internal Hom in $\StinftyCat_{\on{cont}}$, which is therefore another DG category. 

\sssec{Terminological deviation \em{(ii)}} For two DG categories $\bC_1$ and $\bC_2$ we will sometimes
encounter functors $\bC_1\to \bC_2$ that are not continuous (but still exact). For example, for
a non-compact object $\bc\in \bC$, such is the functor $\CMaps_\bC(\bc,-):\bC\to \Vect$
(see below for the notation). 

\medskip

Every time when we encounter a non-continuous functor, we will say so explicitly. 

\medskip

All exact functors $\bC_1\to \bC_2$ also form a DG category, which we denote by $\on{Funct}(\bC_1,\bC_2)$. 

\sssec{Mapping spaces}

Any DG category $\bC$ can be thought of as an $\infty$-category enriched over $\Vect$ with the
same set of objects. For two objects $\bc_1,\bc_2$, we will denote by $\CMaps_\bC(\bc_1,\bc_2)\in \Vect$ the
corresponding Hom object. 

\medskip

We let $\Maps_\bC(\bc_1,\bc_2)\in \inftygroup$ denote the Hom-space, when we consider
$\bC$ as a plain $\infty$-category. The object $\Maps_\bC(\bc_1,\bc_2)$ equals the image
of $\tau^{\leq 0}(\CMaps_\bC(\bc_1,\bc_2))$ under the Dold-Kan functor
$$\Vect^{\leq 0}\to \inftygroup.$$

\medskip

We denote by $\Hom_\bC(\bc_1,\bc_2)$ the object $H^0\left(\CMaps_\bC(\bc_1,\bc_2)\right)\in \Vect^\heartsuit$.
Its underlying set identifies with $\pi_0(\Maps_\bC(\bc_1,\bc_2))$. 

\sssec{t-structures}

Whenever a DG category $\bC$ has a t-structure, we let $\bC^{\leq 0}$ (resp., $\bC^{\geq 0}$) 
denote the full subcategory of connective (resp., co-connective) objects. We denote by $\bC^\heartsuit$ 
the heart of the t-structure.

\ssec{Compactness and compact generation}  \label{ss:DG compact}

\sssec{}  \label{sss:DG compact}

Recall that an object $\bc$ in a (cocomplete) DG category $\bC$ is called \emph{compact}
if the functor
$$\Hom_\bC(\bc,-):\bC\to \Vect^\heartsuit$$
commutes with arbitrary direct sums. This is equivalent to the (a priori non-continuous) functor
$$\CMaps_\bC(\bc,-):\bC\to \Vect$$
being continuous, or the functor of $\infty$-categories
$$\Maps_\bC(\bc,-):\bC\to \inftygroup$$
commuting with filtered colimits. 

\medskip

For a DG category $\bC$, we let $\bC^c$ denote the full (but not cocomplete) 
DG subcategory that consists of compact objects. 

\sssec{Compact generation}  \label{sss:compact_gen_tion}
Let $\bC$ be a cocomplete DG category. We say that a set of
objects $\bc_\alpha\in\bC$ \emph{generates} $\bC$ if for every $\bc\in\bC$ the following implication holds:
%, such that 
\begin{equation} \label{e:right orth}
\Hom_\bC(\bc_\alpha,\bc)=0,\,\, \forall \alpha\,\,\Rightarrow\, \bc=0.
\end{equation}
This is known to be equivalent to the following condition: $\bC$ does not contain a proper full
cocomplete DG subcategory that contains all the objects $\bc_\alpha$. 

\medskip

A cocomplete DG category $\bC$ is called \emph{compactly generated} if there exists a set of compact
objects $\bc_\alpha$ that generates $\bC$ in the above sense.

\sssec{}   %\label{sss:used repeatedly}

The following observations will be used repeatedly throughout the paper:

\medskip

Let $\bC_1$ and $\bC_2$
be a pair of DG categories, and let $\sG:\bC_2\to \bC_1$ be a (not necessarily continuous)
functor. If $\sG$ admits a left adjoint functor $\sF :\bC_1\to \bC_2$ then $\sF$  is automatically continuous. 

\medskip

Let $\sF , \sG$ be as above and suppose, in addition, that $\bC_1$ is compactly generated.
Then $\sG$ is continuous if and only if $\sF$ preserves compactness
(i.e., $\sF (\bC_1^c)\subset\bC_2^c$).
This implies  the ``only if" part of the following well-known proposition.

\begin{prop}    \label{p:existence of adjoint}
Let $\bC_1$ be a compactly generated DG category and $\sF :\bC_1\to \bC_2$ a continuous DG functor.
Then $\sF$ has a continuous right adjoint if and only if $\sF (\bC_1^c)\subset\bC_2^c$.
\end{prop}

\begin{proof}[Proof of the ``if" statement.]
The existence of the \emph{not necessarily continuous right adjoint} $\sG$ follows from the Adjoint Functor Theorem,
see \cite[Corollary 5.5.2.9]{Lu1}. To test that $\sG$ is continuous, it is enough to show that the functors
$$\CMaps_{\bC_1}(\bc_1,\sG(-)):\bC_2\to \Vect$$
are continuous for $\bc_1\in \bC_1^c$. The required continuity follows from the assumption on $\sF$.
\end{proof}

\ssec{Ind-completions}

\sssec{}

Let $\bC^0$ be an essentially small (but not cocomplete) DG category. We can functorially assign to it
a cocomplete DG category, denoted $\on{Ind}(\bC^0)$ (and called the \emph{ind-completion} of $\bC^0$),
equipped with a functor $\bC^0\to \on{Ind}(\bC^0)$ and characterized by the property that restriction defines
an equivalence
\begin{equation} \label{e:ind compl}
\on{Funct}_{\on{cont}}(\on{Ind}(\bC^0),\bD)\to \on{Funct}(\bC^0,\bD)
\end{equation}
for a cocomplete category $\bD$ (see \cite[Sect. 5.3.5]{Lu1} for the corresponding construction
for general $\infty$-categories). 

\medskip

The category $\on{Ind}(\bC^0)$ can be explicitly constructed as $\on{Funct}((\bC^0)^{\on{op}},\Vect)$.

\medskip

It is known that the functor $\bC^0\to \on{Ind}(\bC^0)$ is fully faithful, and that its essential image belongs 
to the subcategory $\on{Ind}(\bC^0)^c$. It follows formally from \eqref{e:ind compl} that the essential image
of $\bC^0$ generates $\on{Ind}(\bC^0)$. %Furthermore, we have the following assertion (see \cite[Lemma 5.4.2.4]{Lu1}):

\sssec{}

Thus, the assignment $\bC^0\rightsquigarrow \on{Ind}(\bC^0)$ is a way to obtain compactly generated
categories. In fact, all cocomplete compactly generated DG categories arise in this way. Namely, 
we have the following assertion (see \cite[Proposition 5.3.5.11]{Lu1}):

\begin{lem} \label{l:ind of comp}
Let $\bC$ be a cocomplete compactly generated DG category. Let $\sF^0:\bC^0\to \bC^c$
be a fully faithful functor, such that its essential image \emph{generates} $\bC$.
Then the resulting functor $\sF:\on{Ind}(\bC^0)\to \bC$, obtained from $\sF^0$
via \eqref{e:ind compl}, is an equivalence. 
\end{lem}

%\begin{proof}

%Let us first show that $\sF$ is fully faithful. Since the objects $\bc^0\in \bC^0$ generate $\on{Ind}(\bC^0)$, it is enough to show
%that the map
%\begin{equation} \label{e:Hom out of compact}
%\CMaps_{\on{Ind}(\bC^0)}(\bc^0,{}'\bc_0)\to \CMaps_{\bC}(\sF^0(\bc^0),\sF({}'\bc_0))
%\end{equation}
%is an isomorphism for any $'\bc_0\in \on{Ind}(\bC^0)$.

%\medskip

%Since the functors
%$$\CMaps_{\on{Ind}(\bC^0)}(\bc^0,-) \text{ and } \CMaps_{\bC}(\sF^0(\bc^0),-)$$
%both commute with colimits (because $\bc^0\in \on{Ind}(\bC^0)$ and $\sF^0(\bc^0)\in \bC$
%are compact), and since $\bC^0$ generates $\on{Ind}(\bC^0)$, it is enough to show that
%\eqref{e:Hom out of compact} is an isomorphism for $'\bc_0\in \bC^0$. However, the latter
%is equivalent to the fact that $\sF^0$ is fully faithful.

%\medskip

%Thus, we obtain that $\sF$ defines an equivalence from $\on{Ind}(\bC^0)$ onto a full cocomplete DG subcategory of $\bC$.
%Since $\sF(\on{Ind}(\bC^0))$ also generates $\bC$, we obtain that this subcategory is all of $\bC$.

%\end{proof}

As a consequence, we obtain:

\begin{cor} \label{c:ind of comp}
Let $\bC$ be a cocomplete compactly generated DG category. 
Then the tautological functor $\on{Ind}(\bC^c)\to \bC$
is an equivalence. 
\end{cor}

\ssec{Karoubi-completions}

\sssec{}  \label{sss:karoubian}

Let $\bC^0$ be an essentially small (but non-cocomplete) DG category. We  say that $\bC^0$ is \emph{Karoubian}
if its homotopy category is idempotent-complete. 

\medskip

For example, for a cocomplete compactly generated DG category $\bC$, the corresponding
subcategory $\bC^c$ is Karoubian. 

\sssec{}

Let $\bC^0\to \bC^0_{\on{Kar}}$ be a functor between essentially small (but non-cocomplete) DG categories. 

\medskip

We  say that the above 
functor realizes $\bC^0_{\on{Kar}}$ as a  \emph{Karoubi-completion} of $\bC^0$ if restriction
defines an equivalence
$$\on{Funct}(\bC^0_{\on{Kar}},{}'\bC_0)\to \on{Funct}(\bC^0,{}'\bC_0)$$
for any Karoubian $'\bC_0\,$. Clearly, $\bC^0_{\on{Kar}}\,$, if it exists, is defined up to
a canonical equivalence.

\medskip

The following is a reformulation of  the  Thomason-Trobaugh-Neeman 
localization theorem  (see \cite[Theorem 2.1]{N} or \cite[Propostion 1.4.2]{BeV}):

\begin{lem}  \label{l:Karoubi}  \hfill

\smallskip

\noindent{\em(a)}
Let $\bC^0$ be an essentially small (but not cocomplete) DG category. The canonical
functor $\bC^0\to \on{Ind}(\bC^0)^c$ realizes $\on{Ind}(\bC^0)^c$ as a Karoubi-completion
of $\bC^0$. 

\smallskip

\noindent{\em(b)} Every object of $\on{Ind}(\bC^0)^c$ can be realized as a direct summand
of one in $\bC^0\subset \on{Ind}(\bC^0)^c$. 
\end{lem}

\lemref{l:Karoubi} implies that the functor $\Ho (\bC^0)\to\Ho (\bC^0_{\on{Kar}})$  identifies
$\Ho  (\bC^0_{\on{Kar}})$ with the idempotent completion of $\Ho (\bC^0)$.

\sssec{}   \label{sss:Kar vs cocompl}
We obtain that the assignments 
$$\bC^0\rightsquigarrow \on{Ind}(\bC^0) \text{ and } \bC\rightsquigarrow \bC^c$$
define mutually inverse equivalences between the appropriate $\infty$-categories.

\medskip

The two $\infty$-categories are as follows. One is $\StinftyCat_{\on{Kar}}$, 
whose objects are essentially small Karoubian DG categories and morphisms are exact functors. 
The other is $\StinftyCat^{\on{comp.gen.}}_{\on{cont,pr.comp.}}$, whose objects are
cocomplete compactly generated categories and morphisms are continuous functors
preserving compactness.

\sssec{}  \label{sss:Karoubi gen}

Let $\bC^0$ be an essentially small (but not cocomplete) DG category. Let $\bS$ be a subset of
its objects.  

\medskip

We say that $\bS$ \emph{Karoubi-generates} $\bC^0$ if every object in the homotopy category of
$\bC^0$ can be obtained from objects in $\bS$ by a finite iteration of operations 
of taking the cone of a morphism, and passing to a direct summand of an object.

\medskip 

By combining Lemmas \ref{l:Karoubi} with \ref{l:ind of comp} we obtain:

\begin{cor}  \label{c:Karoubi}
Let $\bC$ be a cocomplete DG category. Let $\bS\subset \bC^c$ be a subset of objects
that generates $\bC$. Then $\bS$ \emph{Karoubi-generates} $\bC^c$.
\end{cor}

\ssec{Symmetric monoidal structure and duality} \label{ss:dual category}

\sssec{The notion of dual of a DG category}     \label{sss:dualizable category}

A DG category $\bC$ is called \emph{dualizable} if it is such as an object of the
symmetric monoidal category $(\StinftyCat_{\on{cont}},\otimes )$.
We refer the reader to \cite[Sect. 4.1]{finiteness} for a 
review of some of the properties of this notion. The most important
ones are listed below.

\medskip 

For a dualizable category $\bC$ we denote by $\bC^\vee$ its dual. One constructs $\bC^\vee$ explicitly as
\begin{equation} \label{e:dual as}
\bC^\vee\simeq \on{Funct}_{\on{cont}}(\bC,\Vect).
\end{equation}

In addition, for any $\bD\in \StinftyCat_{\on{cont}}$, the natural functor
$$\bC^\vee\otimes \bD\to \on{Funct}_{\on{cont}}(\bC,\bD)$$
is an equivalence. 

\sssec{}  \label{sss:dual functor}

If $\sF:\bC_1\to \bC_2$ is a (continuous)
functor between dualizable categories, there exists a canonically defined dual functor 
$\sF^\vee:\bC_2^\vee\to \bC_1^\vee$ (the construction follows, e.g., from \eqref{e:dual as}).
The assignment $\sF\mapsto \sF^\vee$ is functorial in $\sF$. One has
$(\sF^\vee)^\vee=\sF$, $(\sG\circ \sF)^\vee =\sF^\vee\circ \sG^\vee$.

\medskip

From here we obtain that if the functors 
$$\sF:\bC_1\leftrightarrows \bC_2:\sG$$
are mutually adjoint, then so are the functors
$$\sG^\vee: \bC^\vee_2\leftrightarrows \bC^\vee_1:\sF^\vee.$$

\sssec{} \label{sss:dual of c g}

If $\bC$ is compactly generated, then it is dualizable. We have a canonical identification
$$(\bC^\vee)^c\simeq (\bC^c)^{\on{op}}.$$

Vice versa, if $\bC_1$ and $\bC_2$ are two compactly generated categories, then an identification
$$\bC_1^c\simeq (\bC_2^c)^{\on{op}}$$ gives rise to an identification
$$\bC_1^\vee\simeq \bC_2.$$

%\begin{lem}     \label{l:duality&adjunctions}
%Let $\bC_1$ and $\bC_2$ be dualizable DG categories and 
%$F:\bC_1\rightleftarrows \bC_2:G$ a pair of adjoint functors. Then
%
%(i) the dual functors $G^\vee:\bC_1^\vee \rightleftarrows \bC_2^\vee:F^\vee$ also form
%an adjoint pair;
%
%(ii) $F$ is fully faithful if and only if $G^\vee$ is.
%\end{lem}
%
%\begin{proof}
%(i) The adjunctions $a_1:\Id_{ \bC_1}\to G\circ F$ and $a_2:F\circ G\to\Id_{ \bC_2}$ give rise to
%the adjunctions $a_1^\vee:\Id_{ \bC_1^\vee}\to F^\vee\circ G^\vee$ and 
%$a_2^\vee:G^\vee\circ F^\vee\to\Id_{ \bC_2^\vee}\,$.
%
%(ii) The adjunction $a_1$ is an isomorphism if and only if $a_1^\vee$ is.
%\end{proof}

\ssec{Limits of DG categories}  \label{ss:DGlim}

The reason that we work with DG categories rather than with triangulated ones is that
the limit (i.e., projective limit) of DG categories is well-defined as a DG category (while the 
corresponding fact for triangulated categories is false). 

\medskip

More precisely, the $(\infty ,1)$-categories $\StinftyCat_{\on{cont}}$ and $\StinftyCat$
admit limits and the forgetful functor $\StinftyCat_{\on{cont}}\to \StinftyCat$
commutes with limits (this is essentially \cite[Proposition 5.5.3.13]{Lu1})
This is important for us because the DG category of D-modules on an algebraic stack is defined
as a limit (see \secref{sss:Dmod on prestack} below).

\sssec{}

Let  $$i\mapsto \bC_i,\,\, (i\to j)\mapsto (\phi_{i,j}\in \on{Funct}_{\on{cont}}(\bC_i,\bC_j))$$
be a diagram of DG categories, parameterized by an index category $I$. The limit  
$$\bC:=\underset{i\in I}{\underset{\longleftarrow}{lim}}\, \bC_i$$ 
is a priori defined by a universal property in $\StinftyCat_{\on{cont}}$:
for a DG category $\bD$ we have a functorial isomorphism
$$\left(\on{Funct}_{\on{cont}}(\bD,\bC)\right)^{\on{grpd}}\simeq \underset{i\in I}{\underset{\longleftarrow}{lim}}\, 
\left(\on{Funct}_{\on{cont}}(\bD,\bC_i)\right)^{\on{grpd}}$$
where the in the left-hand side the limit is taken in the $(\infty,1)$-category $\inftygroup$. We remind that the superscript 
``grpd" means that we are taking the maximal $\infty$-subgroupoid in the corresponding $\infty$-category. 

\sssec{}  \label{sss:cart sect}

Note that \cite[Corollary 3.3.3.2 ]{Lu1} provides a more explicit description of $\bC$.
Namely, objects of $\bC$ are \emph{Cartesian sections,} i.e., assignments
$$i\rightsquigarrow (\bc_i\in \bC_i),\,\, \phi_{i,j}(\bc_i)\overset{\alpha_{\phi_{i,j}}}\simeq \bc_j,$$
equipped with data making $\alpha_{\phi_{i,j}}$ coherently associative. In fact, this description
follows easily from the above functorial description, by taking $\bD=\Vect$, and using the fact
$\on{Funct}_{\on{cont}}(\Vect ,\bC)\simeq \bC$ as DG categories. 

\medskip

If $\bc:=(\bc_i,\alpha_{\phi_{i,j}})$ and $\wt\bc:=(\wt\bc_i,\wt\alpha_{\phi_{i,j}})$ are two such objects, then
one can upgrade the assignment
$$i\mapsto \CMaps_{\bC_i}(\bc_i,\wt\bc_i)$$
into a homotopy $I$-diagram in $\Vect$, and
$$\CMaps_\bC(\bc,\wt\bc)\simeq \underset{i\in I}{\underset{\longleftarrow}{lim}}\, \CMaps_{\bC_i}(\bc_i,\wt\bc_i)$$
as objects of $\Vect$.

\sssec{}

The following observation will be useful in the sequel. Let $\bC=\underset{i\in I}{\underset{\longleftarrow}{lim}}\, \bC_i$
be as above, and let
$$(\alpha\in A)\mapsto (\bc_\alpha\in \bC)$$ be a collection of objects
of $\bC$ parameterized by some category $A$. In particular, for every $i\in I$ we obtain a functor
$$(\alpha\in A)\mapsto (\bc_{i,\alpha}\in \bC_i).$$

We have:

\begin{lem}  \label{l:dir limits in inv limit}
For every $i$, the map from $\underset{\alpha\in A}{\underset{\longrightarrow}{colim}}\, \bc_{i,\alpha}\in \bC_i$
to the $i$-th component of the object $\underset{\alpha\in A}{\underset{\longrightarrow}{colim}}\, \bc_\alpha$
is an isomorphism.
\end{lem}

In other words, colimits in a limit of DG categories can be computed component-wise. 

\begin{rem}
The assertion of \lemref{l:dir limits in inv limit} can be reformulated as saying that the evaluation
functors $\on{ev}_i:\bC\to \bC_i$ commute with colimits, i.e., are continuous. This is tautological
from the definition of $\bC$ as a limit in the category $\StinftyCat_{\on{cont}}$.
\end{rem}

\ssec{Colimits in $\StinftyCat_{\on{cont}}$}  \label{ss:colimits abs}
The goal of the remaining part of \secref{s:DG review} 
%(Sects.~\ref{ss:colimits abs}-\ref{ss:Colimcompgen}) 
is to provide a categorical framework for Sects.~\ref{ss:category as colimit}-\ref{ss:descr dual}.
This material is not used in other parts of the article.

%The material from the remaining part of \secref{s:DG review} is used only in 
%Sects. \ref{ss:category as colimit}-\ref{ss:descr dual}. 

\sssec{}
As was mentioned in \secref{ss:DGlim}, limits in $\StinftyCat_{\on{cont}}$ are the same as limits in $\StinftyCat$. 
However, \emph{colimits are different} (for example, the colimit taken in $\StinftyCat$ does not have to be cocomplete). 

\medskip

It is known to experts that under suitable set-theoretical conditions, colimits in $\StinftyCat_{\on{cont}}\,$ 
always exist. 
%(to construct the 
%colimit in $\StinftyCat_{\on{cont}}\,$, one first takes the colimit in $\StinftyCat$ and then passes to the
%``co-completion", in an appropriate sense). 
We are unable to find a really satisfactory reference for this fact. 

\medskip

On the other hand, in this paper we work only with colimits of those functors $$\Psi:I\to \StinftyCat_{\on{cont}}$$ that satisfy the 
following condition: for every arrow $i\to j$ in $I$ the corresponding functor $\psi_{i,j}:\Psi (i)\to\Psi (j)$ 
\emph{admits a continuous right adjoint}. In this case existence of the colimit of $\Psi$ is provided by
\propref{p:limit and colimit a} below.

\sssec{The setting}  \label{sss:let us recall}
Let $I$ be a small category, and let $\Psi:I\to \StinftyCat_{\on{cont}}$  be a functor
$$i\rightsquigarrow \bC_i,\quad (i\to j)\in I \rightsquigarrow \psi_{i,j}\in \on{Funct}_{\on{cont}}(\bC_i,\bC_j).$$
Assume that for every arrow $i\to j$ in $I$, the above functor $\psi_{i,j}$ admits a continuous right adjoint, $\phi_{j,i}\,$. 

\medskip

We can then view the assignment
$$i\rightsquigarrow \bC_i,\quad (i\to j)\in I \rightsquigarrow \phi_{j,i}\in \on{Funct}_{\on{cont}}(\bC_j,\bC_i).$$
as a functor $\Phi:I^{\on{op}}\to \StinftyCat_{\on{cont}}\,$.

\begin{rem}   \label{sss:understandable}
Some readers may prefer to assume, in addition, that each DG category $\bC_i$ is compactly generated.
As explained in \secref{ss:Colimcompgen} below, this special case of the situation of \secref{sss:let us recall} is 
very easy (Propositions~\ref{p:limit and colimit a} and  \ref{p:dual of limit} formulated below become obvious, and it is easy to 
understand ``who is who"). Moreover, this case is enough for the applications in Sects.~\ref{ss:category as colimit}-\ref{ss:descr dual}.
\end{rem} 

\sssec{}
The following proposition is a variant of \cite[Corollary 5.5.3.4]{Lu1}; a digest of the proof is given in
\cite[Lemma 1.3.3]{DG}.

\begin{prop}  \label{p:limit and colimit a} 
In the situation of \secref{sss:let us recall}, the colimit
$$\underset{i\in I}{\underset{\longrightarrow}{colim}}\, \bC_i:=\underset{I}{\underset{\longrightarrow}{colim}}\, \Psi \in \StinftyCat_{\on{cont}}
$$ 
exists and is canonically equivalent to the limit
$$\underset{i\in I^{\on{op}}}{\underset{\longleftarrow}{lim}}\, \bC_i:=\underset{I^{\on{op}}}{\underset{\longleftarrow}{lim}}\, \Phi
\in \StinftyCat_{\on{cont}}\,;$$
the equivalence is uniquely characterized by the condition that for $i_0\in I$, the evaluation functor
$$\on{ev}_{i_0}:\underset{i\in I^{\on{op}}}{\underset{\longleftarrow}{lim}}\, \bC_i \to \bC_{i_0}$$
is right adjoint to the tautological functor
$$\on{ins}_{i_0}:\bC_{i_0}\to \underset{i\in I}{\underset{\longrightarrow}{colim}}\, \bC_i\, ,$$
in a way compatible with arrows in $I$. 
%In the situation of \secref{sss:let us recall}, the limit 
%$$\underset{i\in I^{\on{op}}}{\underset{\longleftarrow}{lim}}\, \bC_i:=
%\underset{I^{\on{op}}}{\underset{\longleftarrow}{lim}}\, \Phi
%\in \StinftyCat_{\on{cont}}$$ is canonically equivalent 
%to the colimit
%$$\underset{i\in I}{\underset{\longrightarrow}{colim}}\, \bC_i:=
%\underset{I}{\underset{\longrightarrow}{colim}}\, \Psi $$ 
%\Volod{(the colimit is taken in $\StinftyCat_{\on{cont}}$).}
%This equivalence is uniquely characterized by the condition that for $i_0\in I$, the evaluation functor
%$$\on{ev}_{i_0}:\underset{i\in I^{\on{op}}}{\underset{\longleftarrow}{lim}}\, \bC_i \to \bC_{i_0}$$
%is the right adjoint of the tautological functor
%$$\on{ins}_{i_0}:\bC_{i_0}\to \underset{i\in I}{\underset{\longrightarrow}{colim}}\, \bC_i,$$
%in a way compatible with arrows in $I$. 
\end{prop} 

The above proposition can be reformulated as follows. Let $\bC$ denote the limit of the DG categories 
$\bC_i\,$. The claim is that each functor $\on{ev}_i:\bC\to \bC_i$ admits a 
 left adjoint functor $'\!\on{ins}_i:\bC_{i}\to \bC$, and that the functors $'\!\on{ins}_i:\bC_{i}\to \bC$ together with the isomorphisms
 \[
 '\!\on{ins}_j\circ\psi_{i,j}\simeq{}'\!\on{ins}_i \, , \quad (i\to j)\in I,
 \]
 that one obtains by adjunction, make $\bC$ into a colimit of the DG categories $\bC_i\,$. 
 
% Then the functors $\on{ins}_{i_0}:\bC_{i_0}\to \bC$ are, of course, compatible with each other via the functors 
% $\psi_{i,j}\,$, and the second claim is that these data make $\bC$ into a colimit of the DG categories $\bC_i\,$. 

\begin{rem}  \label{r:colim filtered 1}
Let $I$ be filtered. In this case one can show (see \cite[Lemma 1.3.6]{DG}) that if an index $i_0\in I$ is such that for every arrow 
$i_0\to i$ the functor $\psi_{i_0,i}:\bC_{i_0}\to \bC_i$ is fully faithful then 
the functor $\on{ins}_{i_0}$ is fully faithful. If $\bC$ is compactly generated this follows from 
\lemref{e:filtered_lemma}(ii) below.
\end{rem}

\ssec{Colimits and duals}   \label{ss:Colimits and duals}

\sssec{}

Assume now that the categories $\bC_i$ are dualizable. Then we can produce yet another functor
$$\Phi^\vee:I\to \StinftyCat_{\on{cont}}$$
that sends 
$$i\rightsquigarrow \bC^\vee_i ,\quad (i\to j)\in I\rightsquigarrow (\phi_{j,i})^\vee\in \on{Funct}_{\on{cont}}(\bC^\vee_i,\bC^\vee_j).$$

\sssec{}

In this case we have the following result (\cite[Lemma 2.2.2]{DG}):

\begin{prop}  \label{p:dual of limit}
The category 
$$\underset{i\in I^{\on{op}}}{\underset{\longleftarrow}{lim}}\, \bC_i:=\underset{I^{\on{op}}}{\underset{\longleftarrow}{lim}}\, \Phi$$
is dualizable, and its dual is given by
$$\underset{i\in I}{\underset{\longrightarrow}{colim}}\, \bC^\vee_i:=
\underset{I}{\underset{\longrightarrow}{colim}}\, \Phi^\vee.$$
This identification is uniquely characterized by the property
that for $i_0\in I$, we have 
\begin{equation} \label{e:ins dual of ev}
(\on{ins}_{i_0,\Phi^\vee})^\vee\simeq \on{ev}_{i_0,\Phi}\, ,
\end{equation}
in a way compatible with arrows in $I$. 
\end{prop}

In formula \eqref{e:ins dual of ev}, the notation $\on{ins}_{i_0,\Psi^\vee}$ means the functor 
$$\on{ins}_{i_0}:\bC^\vee_{i_0}\to \underset{I}{\underset{\longrightarrow}{colim}}\, \Phi^\vee,$$
and the notation $\on{ev}_{i_0,\Phi}$ means the functor
$$\on{ev}_{i_0}:\underset{I^{\on{op}}}{\underset{\longleftarrow}{lim}}\, \Phi\to \bC_{i_0}.\,$$

\begin{rem}
By adjunction between $\on{ins}$ and $\on{ev}$ (see \propref{p:limit and colimit a}), one gets from \eqref{e:ins dual of ev} a similar isomorphism $(\on{ins}_{i_0,\Psi})^\vee\simeq \on{ev}_{i_0,\Psi^\vee }\,$.
\end{rem}

\ssec{Colimits of compactly generated categories}  \label{ss:Colimcompgen}

The main goal of this subsection is to demonstrate that the results of Sects.~\ref{ss:colimits abs}-\ref{ss:Colimits and duals} 
are very easy under the additional assumption that each DG category $\bC_i$ is compactly generated.

\sssec{Who is who}   \label{sss:2who_is_who}
Suppose that in the situation of \secref{sss:let us recall} each of the categories $\bC_i$ is compactly generated, so
\[
\bC_i\simeq\on{Ind}(\bC_i^c)
\]
or equivalently,
\begin{equation}  \label{e:op-pa}
\bC_i\simeq\on{Funct}((\bC_i^c)^{\on{op}},\Vect ).
\end{equation}
By \propref{p:existence of adjoint}, the assumption that the functor $\psi_{i,j}:\bC_i\to\bC_j$ has a continuous right adjoint just means that 
$\psi_{i,j}(\bC_i^c)\subset\bC_j^c$ (so $\psi_{i,j}$ is the ind-extension of a functor 
$\psi_{i,j}^c:\bC_i^c\to\bC_j^c$). 

\medskip

Moreover, the right adjoint functor $\phi_{j,i}:\bC_j\to\bC_i$ is just the restriction functor
\[
\on{Funct}((\bC_j^c)^{\on{op}},\Vect )\to\on{Funct}((\bC_i^c)^{\on{op}},\Vect )
\]
corresponding to $\psi_{i,j}^c:\bC_i^c\to\bC_j^c\,$.

\sssec{On Propositions ~\ref{p:limit and colimit a} and \ref{p:dual of limit} in the compactly generated case} \hfill

\medskip

Set $\bC:=\underset{i\in I}{\underset{\longrightarrow}{colim}}\, \bC_i\,$. 
In our situation the existence of this colimit is clear: in fact,
\begin{equation}  \label{e:ind-colim}
\bC \simeq\on{Ind} (\,\underset{i\in I}{\underset{\longrightarrow}{colim}}\, \bC_i^c),
\end{equation}
where the colimit in the right hand side is computed in $\StinftyCat$.

\medskip

Just as in \secref{sss:2who_is_who}, we can rewrite \eqref{e:ind-colim} as
\begin{equation}  \label{e:2op-pa}
\bC\simeq\on{Funct}(\,\underset{i\in I}{\underset{\longrightarrow}{colim}}\, (\bC_i^c)^{\on{op}},\Vect ). 
\end{equation}
Now the canonical equivalence 
\[
\bC\simeq\underset{i\in I^{\on{op}}}{\underset{\longleftarrow}{lim}}\, \bC_i
\]
from Proposition ~\ref{p:limit and colimit a} becomes obvious: this is just the composition

%Proposition ~\ref{p:limit and colimit a} says that 
%$\underset{i\in I^{\on{op}}}{\underset{\longleftarrow}{lim}}\, \bC_i\,$ canonically identifies with $\bC$. 
%This is clear:
\[
\bC\simeq\on{Funct}(\,\underset{i\in I}{\underset{\longrightarrow}{colim}}\, (\bC_i^c)^{\on{op}},\Vect )
\simeq \underset{i\in I^{\on{op}}}{\underset{\longleftarrow}{lim}}\, 
\on{Funct}((\bC_i^c)^{\on{op}},\Vect )\simeq
\underset{i\in I^{\on{op}}}{\underset{\longleftarrow}{lim}}\, \bC_i\, ,
\]
where the first equivalence is \eqref{e:2op-pa} and the third one comes from \eqref{e:op-pa}.

\medskip

Proposition~\ref{p:dual of limit} says that $\bC$ is dualizable and 
\begin{equation}   \label{e:dual_as_colim}
\bC^\vee\simeq\underset{i\in I}{\underset{\longrightarrow}{colim}}\, \bC_i^\vee\,.
\end{equation}
This is clear because by formula \eqref{e:ind-colim} and \secref{sss:dual of c g}, both sides of
\eqref{e:dual_as_colim} canonically identify with %the ind-completion of
\[
\on{Ind} (\,\underset{i\in I}{\underset{\longrightarrow}{colim}}\, (\bC_i^c)^{\on{op}})
\]
(the colimit in this formula is computed in $\StinftyCat$).

\sssec{} 
%The following corollary will be used in Sects. \ref{ss:category as colimit}-\ref{ss:descr dual}. 
As a consequence of \eqref{e:ind-colim}, we obtain the following
\begin{cor}   \label{c:general not filtered}
In the situation of \secref{sss:2who_is_who} the category $\bC$ is compactly generated.
More precisely, objects of  $\bC$ of the form
\begin{equation}   \label{e:desired form}
\on{ins}_i(\bc),   \quad \quad  i\in I, \, \bc\in\bC_i^c
\end{equation}
are compact and generate $\bC$. \qed
\end{cor}

In addition, one has the following lemma.

\begin{lem}    \label{e:filtered_lemma}
Suppose that in the situation of \secref{sss:2who_is_who} the category $I$ is filtered. Then 

\smallskip

\noindent{\em(i)} %In addition, suppose that the category $I$ is filtered. Then 
every compact object of $\bC$ is of the form \eqref{e:desired form};

\smallskip

\noindent{\em(ii)} for any $i,i'\in I$, $\bc\in \bC^c_{i}$, and $\bc'\in \bC^c_{i'}$ the canonical map
$${\underset{j,\,\alpha:i\to j,\,\beta:i'\to j}{colim}}\; \CMaps_{\bC_j}(\psi_{i,j}(\bc),\psi_{i',j}(\bc'))\to
\CMaps_{\bC}(\on{ins}_i(\bc),\on{ins}_{i'}(\bc'))$$
is an isomorphism. 
\end{lem}

\begin{proof}
For statement (ii), see \cite{Roz}.

\medskip

Using (ii) and the assumption that $I$ is filtered, it is easy to see that the class of objects of the form 
\eqref{e:desired form} is closed under cones and direct summands. So (i) follows from (ii) 
and Corol\-lary~\ref{c:Karoubi}.
\end{proof}

\section{Preliminaries on the DG category of D-modules on an algebraic stack}  \label{s:D-modules}

In this section we recall some definitions and results from \cite{finiteness}.

\ssec{D-modules on prestacks and algebraic stacks}

\sssec{} \label{sss:Dmod on prestack}

Let $\CY$ be a prestack (always assumed locally of finite type). Recall following \cite[Sect. 6.1]{finiteness} that the category
$\Dmod(\CY)$ is defined as the limit
\begin{equation} \label{e:Dmods as limit}
\underset{S\in (\affSch_{\on{ft}})_{/\CY}}{\underset{\longleftarrow}{lim}}\, \Dmod(S),
\end{equation}
where the limit is taken in the $(\infty,1)$-category $\StinftyCat_{\on{cont}}$. Here
$$S\mapsto \Dmod(S)$$
is the functor 
$$(\affSch_{\on{ft}})^{\on{op}}\to \StinftyCat_{\on{cont}}$$
were for $f:S'\to S$ the corresponding map $\Dmod(S)\to \Dmod(S')$ is $f^!$. 

\medskip

I.e., as was explained in \secref{sss:cart sect}, informally, an object $\CF\in \Dmod(\Bun_G)$ is an assignment for every $S\to \CY$ of
an object $\CF_S\in \Dmod(S)$, and for every $f:S'\to S$ over $\CY$ of an isomorphism $f^!(\CF_S)\simeq \CF_{S'}$.

\medskip

In particular, for $\CF_1,\CF_2\in \Dmod(\Bun_G)$, the complex
$\CMaps(\CF_1,\CF_2)$ is calculated as
$$\underset{S\in (\affSch_{\on{ft}})_{/\CY}}{\underset{\longleftarrow}{lim}}\, \CMaps_{\Dmod(S)}((\CF_1)_S,(\CF_2)_S).$$

\medskip

This definition has several variants. For example, we can replace the category of affine schemes by that
of quasi-compact schemes, or all schemes. 

\sssec{}

Assume now that $\CY$ is an Artin stack (see \cite[Sect. 4]{Stacks} for our conventions regarding 
Artin stacks). 

\medskip

In this case, as in \cite[Corollary 11.2.3]{IndCoh}, in the formation of the limit in
\eqref{e:Dmods as limit}, we can replace the category
$(\affSch_{\on{ft}})_{/\CY}$ by its non-full subcategory $(\affSch_{\on{ft}})_{/\CY,\on{smooth}}$, where we restrict objects
to be those pairs $(S,g:S\to \CY)$ for which the map $g$ is smooth, and $1$-morphisms to smooth
maps between affine schemes.

\medskip

As before, we can replace the word ``affine" by ``quasi-compact", or just consider all schemes. 

\ssec{D-modules on a quasi-compact algebraic stack}   \label{ss:QCA}
\sssec{QCA and locally QCA stacks} \label{sss:QCA}  \hfill

\medskip

QCA is shorthand for ``quasi-compact and with affine automorphism groups".

\begin{defn}   \label{d:QCA}
We say that an algebraic stack $\CY$ is \emph{locally QCA} if the automorphism groups of its field-valued points are affine. We say that $\CY$ is \emph{QCA} if it is quasi-compact and locally QCA.
\end{defn}

\medskip

Convention: \emph{in this article all stacks will be assumed to be locally QCA.}
The reason is clear from Theorem~\ref{t:compact generation qc} below.

%\medskip
%\
%\In this subsection, unless explicitly stated otherwise, we will assume that $\CY$ is quasi-compact 
%\(and therefore QCA). 
%\
%\\medskip

\sssec{A property of QCA stacks}

The following result is established in \cite[Theorem 8.1.1]{finiteness}.

\begin{thm}  \label{t:compact generation qc}
Let $\CY$ be a QCA stack. Then the category $\Dmod(\CY)$ is compactly generated.
\end{thm}

\begin{rem}
In fact, \cite[Theorem 8.1.1]{finiteness} produces an explicit set of compact generators of
$\Dmod(\CY)$. These are objects  induced from coherent  sheaves on $\CY$. 
\end{rem}

\begin{rem}
Before \cite{finiteness}, the above result was known for algebraic stacks that can be represented as $Z/G$, 
where $Z$ is a quasi-compact scheme and $G$ is an affine algebraic group acting on $S$. 
Most quasi-compact Artin stacks that appear in practice (e.g., all quasi-compact open substacks of $\Bun_G$) 
admit such a representation. More generally, it was known for algebraic stacks that are perfect in the
sense of \cite{BFN}.
\end{rem}

\sssec{Cartesian products}    \label{sss:times and otimes}

The following result is established in \cite[Corollary 8.3.4]{finiteness}.

\begin{prop}  \label{p:D on prod}
Let $\CY$ and $\CY'$ be QCA stacks. Then the natural functor
$$\on{D-mod}(\CY)\otimes \on{D-mod}(\CY')\to \on{D-mod}(\CY\times \CY')$$
is an equivalence.
\end{prop}

\begin{rem} \label{r:D on prod}
In fact, as is remarked in the proof of \cite[Corollary 8.3.4]{finiteness}, the assertion of \propref{p:D on prod}
is valid for any pair of prestacks $\CY$ and $\CY'$ as long as either $\Dmod(\CY)$ or $\Dmod(\CY')$ is
dualizable (see Sect.~\ref{sss:dualizable category}). 
\end{rem}

\sssec{Compactness and coherence}    \label{sss:compcoh}

Let $Z$ be a quasi-compact scheme. An object of $\Dmod(Z)$ is said to be \emph{coherent} if it is
a bounded complex whose cohomology sheaves are coherent D-modules. 

\medskip

It is known that the (non cocomplete) subcategory $\Dmod_{\on{coh}}(Z)$ that consists of coherent objects
coincides with $\Dmod(Z)^c$ (see \cite[Sect. 5.1.17]{finiteness}). 
Recall from \secref{sss:DG compact} that for a DG category $\bC$ we denote by $\bC^c$ the full subcategory of compact objects.

\medskip

For an algebraic stack $\CY$, an object $\CF\in \Dmod(\CY)$ is said to be \emph{coherent} if
$f^!(\CF)$ (or equivalently, $f^*_\dr(\CF)$) is coherent for any smooth map $f:Z\to \CY$, where $Z$ is a 
quasi-compact scheme. So by definition, the property of coherence is local for the smooth topology.
The full (but non-cocomplete) subcategory of coherent objects of $\Dmod(\CY )$ is denoted by 
$\Dmod_{\on{coh}}(\CY)$.

\medskip 

\begin{thm}     \label{t:ccoh} 
Let $\CY$ be a QCA stack. %\hfill
\begin{enumerate}

\item[(i)] We have the inclusion $\Dmod(\CY)^c\subset \Dmod_{\on{coh}}(\CY)$.

\smallskip

\item[(ii)] The above inclusion is an equality if and only if for every geometric point $y$ of $\CY$, the 
quotient of the automorphism group $\on{Aut} (y)$ by its unipotent radical is finite.
\end{enumerate}
\end{thm}

This theorem is proved in \cite[Lemma 7.3.3 and Corollary 10.2.7]{finiteness}.

\begin{rem}
One may wonder how far coherence is from compactness. The answer is provided by the notion of
\emph{safety}, introduced in \cite[Sect. 9.2]{finiteness}. In \cite[Proposition 9.2.3]{finiteness} it is shown that 
an object of $\Dmod_{\on{coh}}(\CY)$ is compact if and only if it is safe. 
\end{rem}

\begin{rem}  \label{r:c coh non-qc}
Note that the notion of coherence of D-modules makes sense for \emph{any} algebraic stack $\CY$, 
i.e., it does not have to be quasi-compact: we test it by smooth maps $Z\to \CY$, where $Z$ is
a quasi-compact scheme. The inclusion of point (i) of \thmref{t:ccoh} remains valid in this context.
The proof is very easy: for a map $f:Z\to \CY$, the functor $f^*_\dr$
sends compacts to compacts because it admits a continuous right adjoint, namely $f_{\dr,*}$.
\end{rem}
 
\sssec{Verdier duality}  \label{sss:Verdier}
Let $\CY$ be a QCA stack. Accoding to \cite[Sect. 7.3.4]{finiteness}, 
the (non-cocomplete) DG category $\Dmod_{\on{coh}}(\CY)$ carries a natural anti-involution
$$\BD_{\CY}^{\on{Verdier}}:(\Dmod_{\on{coh}}(\CY))^{\on{op}}\to \Dmod_{\on{coh}}(\CY),$$
which we refer to as \emph{Verdier duality}.

\medskip

The following key feature of this functor is established in \cite[Corollary 8.4.2]{finiteness}:

\begin{thm} \label{t:Verdier QCA}
The functor $\BD_{\CY}^{\on{Verdier}}$ sends the subcategory $$(\Dmod(\CY)^c)^{\on{op}}\subset (\Dmod_{\on{coh}}(\CY))^{\on{op}}$$
to $\Dmod(\CY)^c\subset \Dmod_{\on{coh}}(\CY)$.
\end{thm}

\sssec{}     \label{sss:2Verdier}

By \secref{sss:dual of c g}, we obtain that the resulting functor
$$\BD_{\CY}^{\on{Verdier}}:(\Dmod(\CY)^c)^{\on{op}}\to \Dmod(\CY)^c$$
uniquely extends to an equivalence
\begin{equation} \label{e:Verdier}
\bD^{\on{Verdier}}_\CY:\Dmod(\CY)^\vee\simeq \Dmod(\CY).
\end{equation}

\medskip

Alternatively, we can view the Verdier duality functor as follows. By \secref{sss:dualizable category},
the DG category $\on{Funct}_{\on{cont}}(\Dmod(\CY),\Dmod(\CY))$
identifies tautologically with 
$$\Dmod(\CY)^\vee\otimes \Dmod(\CY).$$

The equivalence \eqref{e:Verdier} is characterized by the property that the identity functor on $\Dmod(\CY)$ corresponds to the
object of $\Dmod(\CY)\otimes \Dmod(\CY)$ that identifies via \propref{p:D on prod} with
$$(\Delta_\CY)_{\dr,*}(\omega_\CY)\in \Dmod(\CY\times\CY ).$$
Here $\omega_\CY\in \Dmod(\CY)$ is the dualizing object and $\Delta_\CY:\CY\to\CY\times\CY$ is the diagonal.
%, and $\Dmod(\CY\times\CY )$ is identified with 
%$\Dmod(\CY)\otimes \Dmod(\CY)$ by \propref{p:D on prod}. 

\medskip

Let $\CY_1,\CY_2$ be QCA stacks. If $\sF:\Dmod (\CY_1 )\to  \Dmod (\CY_2 )$ is a continuous functor then the dual functor 
$\sF^\vee :\Dmod (\CY_2 )^\vee\to\Dmod (\CY_1 )^\vee$ (see Sect.~\ref{sss:dual functor})
will be considered, via \eqref{e:Verdier}, as a functor $\Dmod (\CY_2 )\to\Dmod (\CY_1 )$.

\medskip

We will use the following fact \cite[Proposition 8.4.8]{finiteness}.

\medskip

\begin{prop}   \label{p:Lurie-duality}
For any schematic quasi-compact morphism  $f:\CY_1\longrightarrow\CY_2\;$, the functors
\[
f_{\dr,*}: \Dmod (\CY_1 )\to  \Dmod (\CY_2 ), \quad\quad f^!: \Dmod (\CY_2 )\to  \Dmod (\CY_1)
\]
are dual to each other in the sense of Sect.~\ref{sss:dual functor}.
\end{prop}

\ssec{Non quasi-compact algebraic stacks}
Let $\CY$ be now a stack, which is only assumed to be \emph{locally} QCA.
Then every quasi-compact open substack $U\subset \CY$ is QCA,
so the category $\Dmod(U)$ is compactly generated by \thmref{t:compact generation qc}. 
However, it is not true, in general, that the category $\Dmod(\CY)$ is compactly generated.
For a counterexample, see \secref{s:counterexamples}.

\medskip

In this subsection we give a description of the subcategory of compact objects
$$\Dmod(\CY)^c\subset\Dmod(\CY),$$ see \propref{p:tautological compact} below.

\sssec{The category $\Dmod(\CY)$ as a limit}
The following statement immediately follows from the definition of $\Dmod(\CY)$, see 
\secref{sss:Dmod on prestack}.

\begin{lem}  \label{l:Dmod on nonqc via open}
The restriction functor
$$\Dmod(\CY)\to \underset{U\subset \CY}{\underset{\longleftarrow}{lim}}\, \Dmod(U),$$
is an equivalence, where the limit is taken over the poset of open quasi-compact substacks
of $\CY$.
\end{lem}

In particular, we obtain that for $\CF_1,\CF_2\in \Dmod(\CY)$, the natural map
\begin{equation} \label{e:calc Hom via opens}
\CMaps_{\Dmod(\CY)}(\CF_1,\CF_2)\to 
\underset{U\subset \CY}{\underset{\longleftarrow}{lim}}\, \CMaps_{\Dmod(U)}(\CF_1|_U,\CF_2|_U)
\end{equation}
is an isomorphism.

\medskip

The following observation will be useful in the sequel:

\begin{cor}   \label{c:prod}
Suppose that a family of objects $\CF_{\alpha}\in \Dmod(\CY )$ is locally finite, i.e., for every quasi-compact open
$U\subset\CY$ the set of $\alpha$'s such that $\CF_{\alpha}|_U\neq 0$  is finite. Then the map
$$\underset{\alpha}\oplus\, \CF_\alpha\to \underset{\alpha}\, \Pi\, \CF_\alpha$$
is an isomorphism.
\end{cor}

\begin{proof}
Follows immediately from \eqref{e:calc Hom via opens} and \lemref{l:dir limits in inv limit}.
\end{proof}

\sssec{The functors $j^*$ and $j_!$}    \label{sss:! functor}

Let $U\overset{j}\hookrightarrow \CY$ be an open substack. We have a pair of (continuous) adjoint functors 
$$j^*:\Dmod(\CY) \rightleftarrows \Dmod(U):j_*.$$

In particular, the functor $j^*$ sends $\Dmod(\CY)^c$ to $\Dmod(U)^c$. 

\medskip

Now, the functor $j^*$ has a \emph{partially defined} left adjoint, denoted $j_!$. It again follows 
automatically that if for $\CF_U\in \Dmod(U)^c$, the object $j_!(\CF_U)\in \Dmod(\CY)$ is defined,
then it is compact. 

\medskip

We claim:

\begin{lem} \label{l:*!} 
Let $\CF_U\in \Dmod(U)$ be such that $j_!(\CF_U)$ is defined.

\smallskip

\noindent{\em(a)}
The canonical map 
\begin{equation} \label{e:*!}
\CF_U\to j^*(j_!(\CF_U))
\end{equation}
is an isomorphism. 

\smallskip

\noindent{\em(b)} 
If $j':U'\hookrightarrow \CY$ is another open substack, then 
$$(j')^*(j_!(\CF_U))\simeq \wt{j}_!(\CF_U|_{U\cap U'}),$$
(where $\wt{j}:U\cap U'\hookrightarrow U'$). In particular, $\wt{j}_!(\CF_U|_{U\cap U'})$ is defined. 

\end{lem}

\begin{proof}

The functor $j^*\circ j_!$ is the partially defined left adjoint of $j^*\circ j_*$, and the natural transformation
$\on{Id}\to j^*\circ j_!$ is obtained by adjunction from the co-unit $j^*\circ j_*\to \on{Id}$. However, the latter
is an isomorphism since $j_*$ is fully faithful.

\medskip

Statement (b) follows similarly.
\end{proof}

%Since $j_*$ is fully faithful, the canonical map $j^*\circ j_*(\CF_U)\to \CF_U$ is an isomorphism.
%Hence, by the $(j_!,j^*)$-adjunction, we obtain a map
%\begin{equation} \label{e:! to *}
%j_!(\CF_U)\to j_*(\CF_U).
%\end{equation}
%For the same reason, it is easy to see that for any $\CF'_U\in \Dmod(U)$, the induced map
%$$\Hom(j_*(\CF_U),j_*(\CF'_U))\to \Hom(j_!(\CF_U),j_*(\CF'_U))$$
%is an isomorphism, i.e.,
%$$\Hom\left(\on{Cone}(j_!(\CF_U)\to j_*(\CF_U)),j_*(\CF'_U)\right)=0.$$
%Hence, $\on{Cone}(j_!(\CF_U)\to j_*(\CF_U))$ is supported on $\CY-U$, i.e., 
%$$j^*\left(\on{Cone}(j_!(\CF_U)\to j_*(\CF_U))\right)=0,$$ and hence, the map 
%$$j^*\circ j_!(\CF_U)\to j^*\circ j_*(\CF_U),$$
%induced by \eqref{e:! to *} is an isomorphism. However, by construction, the map 
%$$\CF_U\overset{\text{\eqref{e:*!}}}\longrightarrow j^*\circ j_!(\CF_U)\to j^*\circ j_*(\CF_U),$$
%is the inverse of the isomorphism $j^*\circ j_*(\CF_U)\to \CF_U$. This implies that
%\eqref{e:*!} is an isomorphism as well.

\sssec{A description of the subcategory $\Dmod(\CY)^c\subset\Dmod(\CY)$}

%The next proposition gives a description of the subcategory $\Dmod(\CY)^c$:

\begin{prop}  \label{p:tautological compact}
An object $\CF\in \Dmod(\CY)$ is compact if and only if 
\begin{equation}  \label{e:F=j_!}
\CF=j_!(\CF_U)
\end{equation}
for some open quasi-compact $U\overset{j}\hookrightarrow\CY$ and some $\CF_U\in \Dmod(U)^c$.
\end{prop}

Formula \eqref{e:F=j_!} should be understood as follows: the partially defined functor $j_!$ is defined on 
$\CF_U$, and the resulting object is isomorphic to $\CF$. 

\begin{rem}
By \lemref{l:*!}(a), %we obtain that 
the object $\CF_U$ can be recovered from $\CF$ as $\CF|_U:=j^*(\CF)$.
\end{rem}

\subsubsection{Proof of \propref{p:tautological compact}}
First, let us give two more reformulations 
of condition \eqref{e:F=j_!}:

\begin{lem}  \label{l:tautological}
For $\CF\in \Dmod(\CY )$ the following conditions are equivalent:

\smallskip

\noindent{\em(1)} $\CF=j_!(\CF_U)$ for some $\CF_U\in \Dmod(U)$.

\smallskip

\noindent{\em(2)} For any $\CF_1\in \Dmod(\CY)$, supported on 
$\CY-U$, we have $\Hom_{\Dmod(\CY)}(\CF,\CF_1)=0$. 

\smallskip

\noindent{\em(3)}
For any $U\overset{\wt{j}}\hookrightarrow U'\overset{j'}\hookrightarrow \CY$, where $U'$ is another open quasi-compact substack
of $\CY$, we have:
$$\CF|_{U'}\simeq \wt{j}_!(\CF_U),$$
in particular, the object $\wt{j}_!(\CF_U)$ is defined.
\end{lem}

\begin{proof}
By adjunction, (1)$\Leftrightarrow$(2). The implication (1)$\Rightarrow$(3) follows from
\lemref{l:*!}(b). 

\medskip

Let us show that (3) implies (2). By formula \eqref{e:calc Hom via opens}, 
for any $\CF,\CF_1\in \Dmod(\CY)$ one has
\begin{equation} \label{e:Hom as limit}
\CMaps_{\Dmod(\CY)}(\CF,\CF_1)\simeq \underset{U'}{\underset{\longleftarrow}{lim}}\, 
\CMaps_{\Dmod(U')}(\CF|_{U'},\CF_1|_{U'}).
\end{equation}
If $\CF_1$ is supported on $\CY-U$ then all the terms in the RHS are zero, so the LHS
is zero.
\end{proof}

Let us now prove \propref{p:tautological compact}.
%\subsubsection{Proof of \propref{p:tautological compact}}
\begin{proof}
As was remarked in \secref{sss:! functor}, 
if \eqref{e:F=j_!} holds then the compactness of $\CF$ follows  by adjunction. 

\medskip

Conversely, suppose $\CF\in \Dmod(\CY)$ is compact. Then by \secref{sss:! functor},
for every open $U\subset\CY$ the object
$\CF|_U\in \Dmod(U)$ is compact. So it remains to show that 
\eqref{e:F=j_!} holds for some quasi-compact open $U\overset{j}\hookrightarrow\CY$.

\medskip

Assume the contrary. Using the equivalence (1)$\Leftrightarrow$(3) of Lemma~\ref{l:tautological}, we obtain that for every 
quasi-compact open $U\subset\CY$
there is a quasi-compact open $U'\subset\CY$ containing $U$ such that $(j_{U,U'})_!(\CF|_U)\ne(\CF|_{U'})$
(here $j_{U,U'}:U\hookrightarrow U'$). 

\medskip

Thus, we obtain an increasing sequence of open quasi-compact substacks 
$U_i\subset\CY$ such that 
$(j_{U_i,U_{i+1}})_!(\CF|_{U_i})\ne\CF|_{U_{i+1}}$.
Therefore, by Lemma~\ref{l:tautological}, for each $i$ there exists $\CE_i\in \Dmod(U_{i+1})$ such that
$\CE_i|_{U_i}=0$ but $\Hom (\CF|_{U_{i+1}}, \CE_i)\ne 0$. 

\medskip

Let $V$ be the union of the $U_i$'s and let $\tilde\CE_i\in \Dmod(V)$ be the direct image of $\CE_i$
under $U_i\hookrightarrow V$.
Then 
\begin{equation}  \label{e:nonzero}
\Hom (\CF|_V,\tilde\CE_i)=\Hom (\CF|_{U_{i+1}}, \CE_i)\ne 0.
\end{equation}
By Corollary~\ref{c:prod},
\begin{equation}   \label{e:prod}
\Hom (\CF|_V,\underset{i}\oplus\, \tilde\CE_i)\simeq\underset{i}\prod\Hom(\CF|_V,\tilde\CE_i).
\end{equation}
On the other hand, by \secref{sss:! functor}, $\CF|_V$ is compact, so
$\Hom (\CF|_V,\underset{i}\oplus\, \tilde\CE_i)\simeq \underset{i}\oplus\Hom(\CF|_V,\tilde\CE_i)$.
This contradicts \eqref{e:prod} because of \eqref{e:nonzero}.  %\qed
\end{proof}

\section{Truncativeness and co-truncativeness}  \label{s:turncativeness}

Until the last subsection of this section we let $\CY$ be a QCA stack. 

\ssec{The notion of truncative substack}  \label{ss:truncativeness}

\sssec{}

Let $\CZ\overset{i}\hookrightarrow \CY$ be a closed substack, and let $\CY\overset{j}\hookleftarrow U$ be the complementary open. 
Consider the corresponding pairs of adjoint functors
$$i_{\dr,*}:\Dmod(\CZ)\rightleftarrows \Dmod(\CY):i^! , \quad\quad 
j^*:\Dmod(\CY)\rightleftarrows \Dmod(U):j_*\, .$$

Recall that by \thmref{t:compact generation qc}, all the categories involved are compactly generated.

\begin{prop}  \label{p:truncativeness}
The following conditions are equivalent:

\smallskip

\noindent{\em(i)} The functor $i^!$ sends $\Dmod(\CY)^c$ to $\Dmod(\CZ)^c$. 

\smallskip

\noindent{\em(i$'$)} The functor $i^!$ admits a continuous right adjoint. 

\smallskip

\noindent{\em(ii)} The functor $j_*$ sends $\Dmod(U)^c$ to $\Dmod(\CY)^c$.

\smallskip

\noindent{\em(ii$'$)}  The functor $j_*$ admits a continuous right adjoint. 

\smallskip

\noindent{\em(iii)} The functor $j_!$, left adjoint to $j^*$, is defined on all of $\Dmod(U)$.

\smallskip

\noindent{\em(iii$'$)} The functor $j_!$, left adjoint to $j^*$, is defined on $\Dmod(U)^c$.

\smallskip

\noindent{\em(iv)} The functor $i^*_{\dr}$, left adjoint to  $i_{\dr,*}$, is defined on all of 
$\Dmod(\CY )$.

\smallskip

\noindent{\em(iv$'$)} The functor $i^*_\dr$, left adjoint to  $i_{\dr,*}$, is defined on $\Dmod(\CY )^c$.
\end{prop}

Note that in the situation of (iii) and (iv) if the functors $j_!$ and $i^*_\dr$ are defined they are automatically continuous by adjunction.

\medskip

To prove the proposition, we need the following lemma.

\begin{lem}  \label{l:karoubi-generation}
The essential image of $\Dmod(\CY )^c$ under $j^*:\Dmod(\CY )\to \Dmod(U)$ Karoubi-generates 
$\Dmod(U)^c$.
\end{lem}

\begin{proof}
Since $j^*$ has a continuous right adjoint $j_*\,$, we have 
$j^*(\Dmod(\CY )^c)\subset \Dmod(U)^c$. 
Since the functor $j_*$ is conservative $j^*(\Dmod(\CY )^c)$ generates $\Dmod(U)$.
By \corref{c:Karoubi}, this implies that $j^*(\Dmod(\CY )^c)$ Karoubi-generates $\Dmod(U)^c$.
\end{proof}

\begin{proof}[Proof of Proposition~\ref{p:truncativeness}]

Since $j^*$ preserves compactness and $i_{\dr,*}$ is fully faithful and continuous, the fact that (ii) implies (i) follows  from the
exact triangle
$$i_{\dr,*}(i^!(\CF))\to \CF\to j_*\circ (j^*(\CF)).$$
The implication (i)$\Rightarrow$(ii) follows from Lemma~\ref{l:karoubi-generation} 
and the same exact triangle.

\medskip

The equivalences (i)$\Leftrightarrow$(i$'$) and (ii)$\Leftrightarrow$(ii$'$) 
%are tautological
follow from (the tautological) \propref{p:existence of adjoint}.

\medskip

Let us show that (iii$'$)$\Leftrightarrow$(iii)$\Leftrightarrow$(ii$'$). The full subcategory of objects of $\Dmod (U)$ 
on which $j_!$ is defined is closed under colimits. Since $\Dmod (U)$ is generated by $\Dmod (U)^c$ we see that (iii$'$)$\Leftrightarrow$(iii).
%Conditions (iii) and (iii$'$) are equivalent since $\Dmod(U)$ is generated by $\Dmod(U)^c$. 
By \propref{p:Lurie-duality}, the dual of the functor $j^*=j^!:\Dmod(\CY )\to \Dmod(U)$ identifies,
%Note that 
via the self-duality equivalences
$$\bD_U^{\on{Verdier}}:\Dmod(U)^\vee\simeq \Dmod(U), \quad\quad \bD_\CY^{\on{Verdier}}:\Dmod(\CY)^\vee\simeq \Dmod(\CY),$$
with $j_*:\Dmod(U)\to\Dmod(\CY ) $. By duality (see \secref{sss:dual functor}),
the existence of a continuous right adjoint to $j_*$ is equivalent to the existence
of (the automatically continuous) left adjoint of $j_*^\vee\simeq j^*$. I.e., (iii)$\Leftrightarrow$(ii$'$). 
 
\medskip 
 
Similarly to the above proof of (iii$'$)$\Leftrightarrow$(iii)$\Leftrightarrow$(ii$'$),
one shows that (iv$'$)$\Leftrightarrow$(iv)$\Leftrightarrow$(i$'$).
\end{proof}

\sssec{}
The following definition is crucial for this paper.

\begin{defn}     \label{d:trunc1}
A closed substack $\CZ\overset{i}\mono \CY$
%$\CZ\subset \CY$
 is called \emph{truncative} (resp., an open substack $U\overset{j}\mono \CY$
 %$U\subset \CY$
is called \emph{co-truncative}) if it satisfies the equivalent conditions of \propref{p:truncativeness}.
\end{defn}

\sssec{}   \label{sss:non-cocomplete}
Let us reformulate Definition~\ref{d:trunc1} in terms of the non-cocomplete DG categories
$\Dmod(\CY)^c$, $\Dmod(\CZ )^c$, $\Dmod(U)^c$. 

\medskip

First, let $\CZ\overset{i}\hookrightarrow \CY$ be \emph{any} closed substack, and let 
$U\overset{j}\hookrightarrow \CY$  be the complementary open, then we have an 
exact sequence\footnote{By definition, exactness means that $i_*^c$ identifies 
$\Dmod(\CZ )^c$ with a full subcategory of $\Dmod(\CY)^c$,  and  $(j^*)^c$ identifies the 
Karoubi-completion of the quotient $\Dmod(\CY)/\Dmod(\CZ )^c$ with $\Dmod(U)^c$. }
of Karoubian (non-cocomplete) DG categories
\begin{equation}   \label{e:exact_sequence1}
0\to\Dmod(\CZ )^c\overset{(i_{\dr,*})^c}\to\Dmod(\CY)^c\overset{(j^*)^c}\to\Dmod(U)^c\to 0.
\end{equation}

The exactness of \eqref{e:exact_sequence1} follows from the fact that the corresponding sequence of the ind-completions 
$$0\to\Dmod(\CZ )\overset{i_{\dr,*}}\to\Dmod(\CY)\overset{j^*}\to\Dmod(U)\to 0$$ is exact, see \secref{sss:Kar vs cocompl}. 

\medskip

Each of the conditions (i)-(ii) from \propref{p:truncativeness} says that the subcategory 
$$\Dmod(\CZ )^c\subset\Dmod(\CY)^c$$ is \emph{right-admissible}
\footnote{Synonyms: right-admissible=coreflective, left-admissible=reflective.},
which by definition means that the functor $$(i_{\dr,*})^c=(i_!)^c : \Dmod(\CZ )^c\to \Dmod(\CY )^c$$ admits a  
right adjoint $(i^!)^c: \Dmod(\CY )^c\to\Dmod(\CZ )^c$, or equivalently, the functor
$$(j^*)^c : \Dmod(\CY )^c\to \Dmod(U)^c$$ has a right adjoint $(j_*)^c:\Dmod(U)^c\to\Dmod(\CY )^c\,$.

\medskip

Similarly, conditions (iii)-(iv) from 
\propref{p:truncativeness} say that the subcategory $$\Dmod(\CZ )^c\subset\Dmod(\CY)^c$$ is 
\emph{left-admissible}, which by definition means that the functor $$(i_{\dr,*})^c: \Dmod(\CZ )^c\to \Dmod(\CY )^c$$ admits a  
left adjoint $(i^*_\dr)^c: \Dmod(\CY )^c\to\Dmod(\CZ )^c$, or, equivalently, the functor
$$(j^*)^c : \Dmod(\CY )^c\to \Dmod(\CZ )^c$$ has a left adjoint $(j_!)^c:\Dmod(\CZ )^c\to\Dmod(\CY )^c\,$. 

\medskip

In our situation left admissibility is equivalent to right admissibility by Verdier duality.

\medskip

Thus if $i:\CZ\hookrightarrow \CY$ is truncative then in addition to \eqref{e:exact_sequence1} one has the
exact sequences

\begin{equation}   \label{e:exact_sequence2}
0\to\Dmod(U )^c\overset{(j_!)^c}\to\Dmod(\CY)^c\overset{(i^*_\dr)^c}\to\Dmod(\CZ )^c\to 0,
\end{equation}

\begin{equation}   \label{e:exact_sequence3}
0\to\Dmod(U )^c\overset{(j_*)^c}\to\Dmod(\CY)^c\overset{(i^!)^c}\to\Dmod(\CZ )^c\to 0.
\end{equation}

It is convenient to arrange the functors between $\Dmod(\CZ )^c$ and $\Dmod(\CY )^c$ into a sequence
\begin{equation}    \label{e:2the_first_sequence}
(i^*_\dr)^c\, , (i_{\dr,*})^c\, , (i^!)^c
\end{equation}
and the functors between $\Dmod(U )$ and $\Dmod(\CY )$ into a sequence
\begin{equation}    \label{e:2the_second_sequence}
\quad (j_!)^c, (j^*)^c, (j_*)^c\, .
\end{equation}
In each of the sequences, each neighboring pair forms an adjoint pair of functors. 

\ssec{Some examples of (co)-truncative substacks}   \label{ss:non-truncative sch}

\sssec{}   \label{sss:non-truncative sch}

(Co)-truncativeness is a purely ``stacky" phenomenon, i.e., it almost never happens for schemes. 

\medskip

More precisely, it is easy to see that if $j:U\hookrightarrow Y$ is an open embedding of schemes which is not a closed embedding 
then $U$ cannot be co-truncative.  Indeed, choose $\CM\in \Coh(U)$ such that $j_*(\CM)$ is not coherent. 
Then $$j_*(\ind_{\Dmod(U)}(\CM))\simeq \ind_{\Dmod(\CY)}(j_*(\CM))$$
is not in $\Dmod(Y)^c$. Here $\ind_{\Dmod(-)}$ denotes the induction functor from $\IndCoh(-)$ to $\Dmod(-)$, see
\cite[Sect. 5.1.3]{finiteness}.

\sssec{Example}  \label{sss:ex of vector space}

The following example of a co-truncative substack is most important for us:

\medskip

Take $\CY=\BA^n/\BG_m$, where $\BG_m$ acts on $\BA^n$ by dilations. Take $U=(\BA^n-\{0\})/\BG_m\simeq \BP^{n-1}$.
In \secref{ss:est trunc} we will see that $U\hookrightarrow \CY$ is co-truncative. 

\sssec{}  \label{sss:ex of A1}

The most basic case of the above example is when $n=1$. In this case, the co-truncativeness assertion is particularly 
evident. Namely, let us check that condition (iii$'$) of \propref{p:truncativeness} holds.
Indeed, $\Dmod(U)\simeq \Vect$, so it is sufficient to show that $j_!(k)$ is defined, where $k$ is the generator of $\Vect$. This is clear since we are dealing with holonomic D-modules.

\sssec{}  \label{sss:finitely many points}
Here is a generalization of  the example of \secref{sss:ex of A1} in a direction different from 
\secref{sss:ex of vector space}:
if $\CY$ is any QCA stack that has only finitely many isomorphism classes of $k$-points then every open
substack $U\subset\CY$ is co-truncative. Indeed, condition (iii$'$) of \propref{p:truncativeness} is verified
because every object of $\Dmod(U)^c$ is holonomic. 

\medskip

Examples of such $\CY$ include $N\backslash G/B$, or any quasi-compact open of $\Bun_G$ for $X$ of genus $0$.

\ssec{The non-standard functors}  \label{sss:voprosial}

Let $\CZ\overset{i}\mono \CY$ be a truncative closed substack and $U\overset{j}\mono \CY$ the corresponding
co-truncative open substack.

\begin{defn}     \label{d:non-standard}
The functors right adjoint to $i^!$ and $j_*$  are
denoted by 
$$i_?:\Dmod(\CZ )\to\Dmod(\CY ), \quad j^?:\Dmod( \CY)\to\Dmod(U ).$$
\end{defn}

\begin{rem}    \label{r:dual_functors}
The proof of the equivalences (iii)$\Leftrightarrow$(ii$'$) and (iv)$\Leftrightarrow$(i$'$)
from \propref{p:truncativeness} shows that $i_?$ is the dual to $i^*_\dr:\Dmod(\CY )\to\Dmod(\CZ )$ and
$j^?$ is the dual to $j_!:\Dmod( U )\to\Dmod(\CY )$ in the sense of Sect.~\ref{sss:dual functor}. Recall that
these dualities follow from the duality between $i^!$ and $i_{\dr,*}\,$, and between $j_*$ and $j^*$.
\end{rem}

\begin{rem}
Recall that the existence of $i^?$ and/or $j^?$ as a continuous functor is among the equivalent definitions of
truncativeness, see Definition~\ref{d:trunc1} and Proposition~\ref{p:truncativeness}(i$'$, ii$'$).

\medskip

The existence of $i^*_\dr$ and/or $j_!$ as an everywhere defined (and automatically continuous) functor is also among 
the equivalent definitions of truncativeness, see 
Proposition~\ref{p:truncativeness}(iii, iv).

\medskip

The functors $i_?$, $j^?$, $i^*_\dr$, $j_!$ are called \emph{the non-standard functors} associated to 
$\CZ\subset\CY$ (or to $U\subset\CY$). 

\medskip

The functors $i^*$ and $j_!$ are at least, familiar as \emph{partially} defined functors (e.g., they are 
always defined on the holonomic subcategory), but $i_?$ and $j^?$ are quite unfamiliar. 
On the other hand, in some situations the non-standard functors identify with certain standard 
functors, see Example~\ref{sss:baby_Springer} and Remark~\ref{r:on_generalizations} below.
\end{rem}

\sssec{Inventory}   \label{sss:inventory}

It is convenient to arrange the functors between $\Dmod(\CZ )$ and $\Dmod(\CY )$ into a sequence
\begin{equation}    \label{e:the_first_sequence}
i^*_\dr, i_{\dr,*}, i^!, i_?
\end{equation}
and the functors between $\Dmod(U )$ and $\Dmod(\CY )$ into a sequence
\begin{equation}    \label{e:the_second_sequence}
\quad j_!, j^*, j_*, j^?.
\end{equation}
In each of the sequences, each neighboring pair forms an adjoint pair of functors. The first and last functors in 
\eqref{e:the_first_sequence} and in \eqref{e:the_second_sequence} are non-standard, the other functors are standard. 
By \remref{r:dual_functors}, each of the sequences 
\eqref{e:the_first_sequence}-\eqref{e:the_second_sequence} is self-dual
in the sense of  Sect.~\ref{sss:dual functor}.

\sssec{}

We know that the functors $i_{\dr,*}$ and $j_*$ are fully faithful; equivalently, the
adjunctions 
\begin{equation}    \label{e:known_isomorphisms}
i^*_\dr\circ i_{\dr,*}\to\Id_{\Dmod (\CZ )}, \quad \Id_{\Dmod (U)}\to j^?\circ j_*
\end{equation}
are isomorphisms (just as are the adjunctions $j^*\circ j_*\to \Id_{\Dmod (U)}$ and $\Id_{\Dmod (\CZ )}\to i^!\circ i_{\dr,*}$, 
which involve only the standard functors).

\begin{prop}     \label{p:inventory}  \hfill

\smallskip

\noindent{\em(i)} The functors $i_?$ and $j_!$ are fully faithful.

\smallskip

\noindent{\em(ii)} The adjunctions  $i^!\circ i_?\to \Id_{\Dmod (\CZ )}$ and $\Id_{\Dmod (U)}\to j^*\circ j_!$   are isomorphisms.
\end{prop}

Although this proposition is extremely simple, we will give two proofs. 

\begin{proof}[Proof 1]

Statements (i) and (ii) are clearly equivalent, so it suffices to prove (ii).

\medskip

Recall that the adjoint pairs $(i^!, i_?)$ and $(j_*, j^?)$ are dual to the adjoint pairs 
$(i^*_\dr, i_{\dr,*})$ and $(j_!, j^*)$. So statement (ii) follows from the fact that the adjunctions \eqref{e:known_isomorphisms} are isomorphisms.
\end{proof}

\begin{proof}[Proof 2]

We will deduce statement (i) from the following general lemma, which is part of the categorical 
folklore.\footnote{Lemma~\ref{l:loc and coloc} for $\infty$-categories immediately follows from the same statement for usual categories. 
For proofs in the setting of usual categories, see \cite[Lemma 1.3]{DT},
\cite[Proposition 2.3]{KeLa},  and the article on adjoint triples from \cite{nLab} (on the other hand, the reader can easily reconstruct the argument because we essentially used it in the proof of \lemref{l:*!}(a)).
Note that in the case of triangulated categories and functors 
(which is enough for our purpose) Lemma~\ref{l:loc and coloc} is well known.}

\medskip

\begin{lem} \label{l:loc and coloc}
Let $\sF$ be a functor between $\infty$-categories that admits a left adjoint $\sF^L$
and a right adjoint $\sF^R$. Then $\sF^L$ is fully faithful if and only if $\sF^R$ is.
\end{lem}

%The lemma is standard, but as we could not find a reference for a proof, we supply it in 
%Appendix \ref{s:loc and coloc}.

\medskip

Let us apply \lemref{l:loc and coloc} to $\sF:=j^*$. Since 
$(j^*)^R=j_*$ is fully faithful, we obtain that $(j^*)^L=j_!$ is fully faithful.

\medskip

Let us apply \lemref{l:loc and coloc} to $\sF:=i^!$. Since 
$(i^!)^L=i_{\dr,*}$ is fully faithful, we obtain that $(i^!)^R=i_?$
is fully faithful.
\end{proof}

%Let $F:C\rightleftarrows D:G$ be a pair of adjoint functors, with $G$ fully faithful. Then
%L$G^\vee:C^\vee \rightleftarrows D^\vee:F^\vee$
%Lare also adjoint and $F^\vee$ is fully faithful. We apply this to $C=\Dmod(X)$, $D=\Dmod(Y)$, 
%Lwhere $i:Y\mono X$ is a truncative closed
%Lembedding, $F=i^*$, and $G=i_*$.  We know that $G^\vee=i^!$, so $F^\vee=i_?\,$.
%L
%LLemma~\ref{l:duality&adjunctions} says that if
%L$F:\bC_1\rightleftarrows \bC_2:G$ is a pair of adjoint functors then so is 
%L$G^\vee:\bC_1^\vee \rightleftarrows \bC_2^\vee:F^\vee$ and that 
%L$F$ is fully faithful if and only if $G^\vee$ is. How do we apply this?
%LSee \propref{p:Lurie-duality}.

\sssec{}

Regardless of whether the substack $\CZ\subset\CY$ is truncative, one has canonical exact sequences of DG categories
\begin{equation}   \label{e:ind-exact_sequence1}
0\to\Dmod(\CZ )\overset{i_{\dr,*}}\to\Dmod(\CY)\overset{j^*}\to\Dmod(U)\to 0 
\end{equation}
and
\begin{equation}   \label{e:ind-exact_sequence3}
0\to\Dmod(U )\overset{j_*}\to\Dmod(\CY)\overset{i^!}\to\Dmod(\CZ )\to 0,
\end{equation}
where the latter is obtained from the former by passing to right adjoints. 

\medskip

If $\CZ$ is truncative one also has exact sequences
\begin{equation}   \label{e:ind-exact_sequence4}
0\to\Dmod(\CZ )\overset{i_?}\to\Dmod(\CY)\overset{j^?}\to\Dmod(U)\to 0
\end{equation}
and
\begin{equation}   \label{e:ind-exact_sequence2}
0\to\Dmod(U )\overset{j_!}\to\Dmod(\CY)\overset{i^*_\dr}\to\Dmod(\CZ )\to 0,
\end{equation}
where \eqref{e:ind-exact_sequence4} is obtained by passing to right adjoints from
\eqref{e:ind-exact_sequence3}, and \eqref{e:ind-exact_sequence2} is obtained by
passing to left adjoints from \eqref{e:ind-exact_sequence1}.

\medskip

In addition, \eqref{e:ind-exact_sequence1} and \eqref{e:ind-exact_sequence3} are obtained
from one another by passing to the dual categories and functors. Similarly, 
\eqref{e:ind-exact_sequence2} and \eqref{e:ind-exact_sequence4} are obtained
from one another by passing to the dual categories and functors.

\sssec{Example}   \label{sss:baby_Springer}
Consider the situation of \secref{sss:ex of A1}, i.e., the embedding $i:\CZ\mono\CY$, where
$\CY=\BA^1/\BG_m$, $\CZ= \{ 0\}/\BG_m$. Let $\pi:\CY\to\CZ$ be the morphism induced by the map $\BA^1\to  \{ 0\}$. 
Let us show that \emph{the non-standard functors
$$i^*_\dr:\Dmod (\CY )\to\Dmod (\CZ) \text{ and }i_?:\Dmod (\CZ )\to\Dmod (\CY)$$ identify with
the following standard functors:}
\[
i^*_\dr\simeq\pi_{\dr,*}\; , \quad i_?\simeq \pi^!\,;
\]
in other words, $(\pi_{\dr,*}, i_{\dr,*})$ and $(i^!, \pi^!)$ are adjoint pairs. By \propref{p:Lurie-duality},
$\pi_{\dr,*}$ is dual to $\pi^!$ and $i_{\dr,*}$ is dual to $i^!$, so it suffices to show that
$(\pi_{\dr,*}, i_{\dr,*})$ is an adjoint pair. Let us prove that for any $\CM\in\Dmod (\CY )$, $\CN\in\Dmod (\CZ )$ the map
\[
\pi_{\dr,*}:\Hom(\CM,i_{\dr,*}(\CN))\to\Hom(\pi_{\dr,*}(\CM),\pi_{\dr,*}\circ i_{\dr,*}(\CN))=\Hom(\pi_{\dr,*}(\CM),\CN)
\]
is an isomorphism. 

\medskip

This is clear if $\CM\in \Dmod (\CZ )\subset\Dmod (\CY)$. 
The DG category  $\Dmod (\CY)$ is generated by $\Dmod (\CZ)$ and $j_!(k)$, where 
$j:\on{pt}=\CY-\CZ\mono\CY$ is the open embedding. So it remains to consider the case 
$\CM=j_!(k)$. Then $\Hom(\CM,i_{\dr,*}(\CN))=0$ and $\pi_{\dr,*}(\CM)=0$ 
(the latter follows from the fact the de Rham cohomology of $\BA^1$ equals $k$).

\begin{rem}  \label{r:on_generalizations}
Example~\ref{sss:baby_Springer} is a ``baby case" of Proposition~\ref{p:adjointness}.
\end{rem}

\ssec{Truncativeness of locally closed substacks}

%\sssec{}

Let $\CZ\overset{i}\hookrightarrow \CY$ be a locally closed substack. This means that $i$
becomes a locally closed embedding after any base change $Y\to \CY$, where $Y$
is a scheme (in fact, it suffices to verify this condition for just one smooth or flat covering 
$Y\to \CY$).

\begin{defn}   \label{d:trunc2}
A locally closed substack $\CZ\overset{i}\hookrightarrow \CY$ is said to be 
\emph{truncative} if the functor $i^!$ preserves compactness (or equivalently, has a continuous right adjoint functor $i_?$).
\end{defn}

For instance, any open substack is truncative. 

\sssec{}

Definition \ref{d:trunc2} immediately implies that truncativeness is transitive:

\begin{lem}  \label{l:trans of trunc}
Let $\CY_1\hookrightarrow \CY_2\hookrightarrow \CY_3$ be locally closed embeddings. If
$\CY_1$ is truncative in $\CY_2$ and $\CY_2$ is truncative in $\CY_3$, then
$\CY_1$ is truncative in $\CY_3$. \qed
\end{lem}

As in the case of schemes, 
%it is easy to see that 
every locally closed embedding  $\CZ\hookrightarrow \CY$ can be factored
(and even canonically so) as
\begin{equation}  \label{e:closed_open}
\CZ\overset{i'}\hookrightarrow \CY'\overset{j}\hookrightarrow \CY,
\end{equation}
where $i'$ is a closed embedding, and $j$ is an open embedding. Namely, 
$\CY':=\CY-(\bar\CZ-\CZ)$, where $\bar\CZ$ is the closure
of $\CZ$ in $\CY$ (so that $\CZ$ is open in $\bar\CZ$).

\begin{lem}  \label{l:trunct for lc via cl}
A locally closed substack $\CZ\overset{i}\hookrightarrow \CY$ is truncative if and only if for some/any 
factorization \eqref{e:closed_open} with $i'$ being closed and $j$ open,
%open substack $\CY'\subset \CY$, such that $i$ factors as
%$$\CZ\overset{i'}\hookrightarrow \CY'\overset{j}\hookrightarrow \CY,$$
%with $i'$ being a closed embedding, 
$\CZ$ in truncative in $\CY'$.
\end{lem}

\begin{proof}
The ``if" statement follows from Lemma~\ref{l:trans of trunc}. It remains to show that
if the composition \eqref{e:closed_open} is truncative then so is $\CZ\overset{i'}\hookrightarrow \CY'$.
This follows from the fact that the essential image of $\Dmod(\CY)^c$ under $j^*$ Karoubi-generates
$\Dmod(\CY')^c$, see Lemma~\ref{l:karoubi-generation}.
\end{proof}

%The above lemma implies:
%
%\begin{cor}
%A locally closed substack $\CZ\overset{i}\hookrightarrow \CY$ is truncative if and only if the functor $i^!$ 
%preserves coherence.
%\end{cor}
%
%\begin{cor} \label{s:base change trunc lc}
%The property of $\CZ$ being truncative in $\CY$ is stable under base change by locally closed
%embeddings $\wt\CY\to \CY$. 
%\end{cor}

\begin{rem}  \label{r:non-standard_locally_closed}
In the case of locally closed substacks the situation with the non-standard functors is as follows.
By duality (in the sense of Sect.~\ref{sss:dual functor}), a locally closed substack 
$\CZ\overset{i}\hookrightarrow \CY$ is truncative if and only if the functor $i_{\dr,*}$ has a left adjoint functor $i^*_\dr$
(which is automatically continuous). 

\medskip

Thus for a truncative locally closed substack we have adjoint 
pairs of continuous functors $(i^*_\dr,i_{\dr,*})$ and $(i^!,i_?)$ dual to each other. 
Just as in the case of closed embeddings (see \propref{p:inventory}), the functors $i_{\dr,*}$ and $i_?$ are fully faithful; 
equivalently, the adjunctions $i^!\circ i_?\to \Id_{\Dmod (\CZ )}$ and $\Id_{\Dmod (\CZ )}\to i_{\dr,*}\circ i^*_\dr$  are isomorphisms. 
But if the substack $\CZ$ is not closed then $i_{\dr,*}\ne i_!\,$, so the  functors $i_{\dr,*}$ and $i^!$ 
do not form an adjoint pair.
\end{rem}

\ssec{Truncativeness via coherence}  \label{ss:trunc via coh}

\sssec{}

As was mentioned in \secref{sss:compcoh}, the property of compactness of a D-module on a stack is subtle. For
example, it is not local in the smooth topology. We are going to reformulate the notion of truncativeness
via a more accessible property, namely, coherence. 

%\sssec{}

%Let 
%$$\CZ\overset{i}\hookrightarrow \CY \overset{j}\hookleftarrow U$$
%be as above. We claim:

\begin{prop}  \label{p:trunc and coh} \hfill

\smallskip

\noindent{\em(a)} A locally closed substack $\CZ\overset{i}\hookrightarrow \CY$ is truncative if and only if the functor $i^!$ 
sends $\Dmod_{\on{coh}}(\CY)$ to $\Dmod_{\on{coh}}(\CZ)$. 

\smallskip

\noindent{\em(b)} An open substack $U\overset{j}\mono\CY$ is co-truncative if and only if $j_*$ sends 
$\Dmod_{\on{coh}}(U)$ to $\Dmod_{\on{coh}}(\CY)$. 

\end{prop}

\begin{proof}

To prove the ``if" implications in both (a) and (b) we will use the notion of \emph{safety}
from \cite[Sect. 9.2]{finiteness}, and the fact that for a morphism $f:\CY_1\to \CY_2$
between QCA stacks, the functor $f_{\dr,*}$ always preserves safety, and $f^!$ 
preserves safety if $f$ itself is safe (in particular, when $f$ is schematic); see 
\cite[Lemma 10.4.2]{finiteness}.

\medskip

Thus, the ``if" implications follow from the fact that ``compactness=coherence+safety'', 
see \cite[Proposition 9.2.3]{finiteness}.

\medskip

To prove the ``only if" implication in (a), we will use the following result (see \cite[Lem\-ma~9.4.7(a)]{finiteness}):
%\cite[Somewhere in 8.4 {\bf ??????}]{finiteness}.

\begin{lem}  \label{l:approx}
For a QCA stack $\CY$, an object $\CF\in \Dmod_{\on{coh}}(\CY)$ and an integer $n$, there exists 
$\CF'\in \Dmod(\CY)^c$ and a map $\CF'\to \CF$, such that its cone lies in $\Dmod(\CY)^{< -n}$. \qed
\end{lem}

%\begin{proof}
%This follows from the fact that the abelian category $\Dmod(\CY)^\heartsuit$ is Noetherian, and that every
%object of 
%$$\Dmod_{\on{coh}}(\CY)^\heartsuit:=\Dmod(\CY)^\heartsuit\cap \Dmod_{\on{coh}}(\CY)$$
%receives a surjection from a compact object. Compact objects in question can be taken of
%the form $\ind_{\Dmod(\CY)}(\CM)$ for $\CM\in \Coh(\CY)$.
%\end{proof}

Note that the functor $i^!$ is left t-exact, and has a finite cohomological amplitude, say $k$. For $\CF\in \Dmod_{\on{coh}}(\CY)$,
which lies in $\Dmod(\CY)^{\geq -m}$, choose $\CF'$ as in \lemref{l:approx} with $n>k+m$. Consider the exact triangle
$$i^!(\CF')\to i^!(\CF)\to i^!(\CF''),$$
where $\CF'':=\on{Cone}(\CF'\to \CF)$. By construction, the maps
\begin{equation} \label{e:coh tr}
\tau^{\geq -m}(i^!(\CF'))\to \tau^{\geq -m}(i^!(\CF))\to i^!(\CF)
\end{equation}
are isomorphisms. 

\medskip

By assumption, $i^!(\CF')\in \Dmod(\CZ)^c\subset \Dmod_{\on{coh}}(\CZ)$.
Note also that the truncation functors preserve the subcategory $\Dmod_{\on{coh}}(-)$. Hence
$\tau^{\ge-m}(i^!(\CF'))\in \Dmod_{\on{coh}}(\CZ)$. Hence, \eqref{e:coh tr} implies that
$i^!(\CF)\in \Dmod_{\on{coh}}(\CZ)$, as desired.

\medskip

The ``only if" implication in (b) is proved similarly. 
\end{proof}

\ssec{Stability of truncativeness}

In this subsection $i:\CZ\hookrightarrow \CY$ denotes a locally closed embedding.
%we will establish some basic properties of the truncativeness/co-truncativeness condition.
%Let $\CZ\overset{i}\hookrightarrow \CY \overset{j}\hookleftarrow U$ be as above.

\sssec{Cartesian products}

%First, we have:

\begin{lem}  \label{l:trunc and prod}
Suppose that a %locally closed 
substack $\CZ\overset{i}\hookrightarrow \CY$ is truncative. Then for any QCA stack $\CX$, the substack
%closed embedding
$\CZ\times \CX\hookrightarrow \CY\times \CX$
is also truncative.
\end{lem}

\begin{proof}

By \cite[Corollary 8.3.4]{finiteness}, for a pair of QCA stacks $\CX_1$ and $\CX_2$, the natural functor
$$\Dmod(\CX_1)\otimes \Dmod(\CX_2)\to \Dmod(\CX_1\times \CX_2)$$
is an equivalence. So the functor $(i\times \on{id}_\CX)^!:\Dmod(\CY\times \CX)\to\Dmod(\CZ\times \CX)$
identifies with the functor $i^!\otimes \on{Id}_{\Dmod(\CX)}\,$, which clearly preserves compactness.
%\medskip
%
%We obtain that the functor
%$$(j\times \on{id}_\CX)_!:\Dmod(U\times \CX)\to \Dmod(\CY\times \CX),$$
%left adjoint to $(j\times \on{id}_\CX)^*$, is given by $j_!\otimes \on{Id}_{\Dmod(\CX)}$.
\end{proof}

\sssec{Descent}

%Assume now that there exists a smooth surjective morphism $f:\wt\CY\to \CY$, such that the resulting 
%substack
%$$\wt\CZ:=\CZ\underset{\CY}\times \wt\CY\overset{\wt{i}}\hookrightarrow \wt\CY$$
%is truncative. 

\begin{prop}  \label{p:truncativeness after smooth}
Let $\CZ\subset\CY$ be a locally closed substack, $f:\wt\CY\to \CY$  a smooth morphism, and 
$\wt\CZ\subset\CZ\underset{\CY}\times \wt\CY$ an open substack such that the resulting morphism 
$f':\wt\CZ\to \CZ$ is surjective. 
%\overset{\wt{i}}\hookrightarrow \wt\CY$
If the locally closed embedding $\wt{i}:\wt\CZ\hookrightarrow \wt\CY$ 
is truncative then so is\, $i:\CZ\hookrightarrow \CY$.
\end{prop}

\begin{proof}
By \propref{p:trunc and coh}(a), it suffices to show that %he functor
$i^!$ sends $\Dmod_{\on{coh}}(\CY)$ to $\Dmod_{\on{coh}}(\CZ)$. 
%Let $f'$ denote the resulting morphism $\wt\CZ\to \CZ$. 
The morphism $f'$ is smooth and surjective, so it suffices to show that the functor $f'{}^!\circ i^!$  preserves coherence.
%for $\CF'\in \Dmod(\CZ)$ one has
%$$\CF'\in \Dmod_{\on{coh}}(\CZ) \, \Leftrightarrow\, f'{}^!(\CF')\in \Dmod_{\on{coh}}(\wt\CZ).$$ 
But
$f'{}^!\circ i^!\simeq \wt{i}{}^!\circ f^!$, and each of the functors $\wt{i}^!$ and $f^!$ preserves coherence.
\end{proof}

\begin{cor}       \label{c:Zariski-local}
Let $\CZ\subset\CY$ be a locally closed substack.
Suppose that each $z\in\CZ$ has a Zariski neighborhood
$U\subset\CY$ such that $\CZ\cap U$ is truncative in $U$. Then $\CZ$ is truncative in $\CY$. \qed
\end{cor}

\begin{rem}
The converse to \propref{p:truncativeness after smooth} is false: truncativeness downstairs does \emph{not}
imply truncativeness upstairs (e.g., consider the embedding $\on{pt}/\BG_m\hookrightarrow \BA^1/\BG_m$ smoothly 
covered by $\on{pt}\hookrightarrow \BA^1$). However, the converse to \propref{p:truncativeness after smooth} does hold for \'etale schematic morphisms; this follows from \lemref{l:etale trunc} below.
\end{rem}

\begin{lem}       \label{l:proper_descent}
Suppose that in a Cartesian diagram
$$
\CD
\wt\CZ  @>{\wt{i}}>>  \wt\CY \\
@V{f'}VV    @VV{f}V   \\
\CZ  @>{i}>>  \CY
\endCD
$$
$f$ is schematic, proper and surjective, and $i$ a locally closed embedding.
If $\wt\CZ$ is truncative in $\wt\CY$ then $\CZ$
is truncative in~$\CY$.
\end{lem}

\begin{proof}
First, %as in \cite[Proposition 4.5.3]{IndCoh}, 
by  \cite[Lemma 5.1.6]{finiteness}, the functor $f^!$ is conservative. Hence, the
essential image of $f_{\dr,*}$ generates $\Dmod(\CY)$. Hence, by \corref{c:Karoubi},
the essential image of $\Dmod(\wt\CY)^c$ under $f_{\dr,*}$ Karoubi-generates
$\Dmod(\CY)^c$. Therefore, it is sufficient to show that the functor
$i^!\circ f_{\dr,*}$ preserves compactness. But %However, by base change
$i^!\circ f_{\dr,*}\simeq f'_{\dr,*}\circ \wt{i}^!$, the functor 
$\wt{i}^!$ preserves compactness by assumption, and $f'_{\dr,*}$ preserves
compactness by properness (it has a continuous right adjoint given by $(f')^!$). 
\end{proof}

\sssec{Quasi-finite base change}
%\sssec{Quasi-finite pullback}

\begin{lem}  \label{l:etale trunc}
Suppose that $f:\wt\CY\to \CY$ is \'etale and schematic. 
If a locally closed embedding
$i:\CZ\hookrightarrow \CY$
is truncative then so is $\wt{i}:\CZ\underset{\CY}\times \wt\CY\mono \wt\CY$.
\end{lem}

\begin{proof}
The functor $f_{\dr,*}:\Dmod(\wt\CY)\to \Dmod(\CY)$ is conservative. So by \corref{c:Karoubi},
the essential image of $\Dmod(\CY)^c$ under $f^*_\dr\simeq f^!$ Karoubi-generates $\Dmod(\wt\CY)^c$.
So it is enough to show that $\wt{i}^!\circ f^!$ preserves compactness. However, 
$\wt{i}^!\circ f^!\simeq f'{}^!\circ i^!$. Now, $i^!$ preserves compactness by assumption, and $f'{}^!$ preserves 
compactness because it is isomorphic to $(f')^*_\dr$, which is the left adjoint of a continuous functor, namely, $f'_{\dr,*}$.
\end{proof}

\begin{lem}  \label{l:base change trunc lc}
If $f:\wt\CY\hookrightarrow \CY$ is a locally closed embedding and a locally closed substack 
$\CZ\hookrightarrow \CY$
is  truncative then so is $\CZ\underset{\CY}\times \wt\CY\hookrightarrow\wt \CY$.
\end{lem}

\begin{proof}
If $f$ is an open embedding the statement holds by Lemma~\ref{l:etale trunc}.  If $f$ is a closed embedding
%his case 
 use the fact that an object $\CF\in \Dmod(\wt\CY)$ is compact if and only
if $f_{\dr,*}(\CF)\in \Dmod(\CY)$ is; this follows from the fact that the functor $f_{\dr,*}$ is 
fully faithful and continuous.
\end{proof}

Lemma~\ref{l:base change trunc lc} for a closed embedding $f$ admits the following generalization.

\begin{prop}  \label{p:stab under fin}
Let $f:\wt\CY\to \CY$ be a finite schematic morphism. If a locally closed embedding
$i:\CZ\hookrightarrow \CY$
is truncative then so is\, %$\wt\CZ:=\CZ\underset{\CY}\times \wt\CY$ .
$\wt{i}:\CZ\underset{\CY}\times \wt\CY\overset{}\hookrightarrow \wt\CY$.
\end{prop}

To prove the proposition, we need the following lemma.
\begin{lem}  \label{l:fin coh}
Let $g:\CX'\to \CX$ be a finite schematic morphism. If $\CF'\in \Dmod(\CX')$ is such that
$g_{\dr,*}(\CF')\in \Dmod(\CX)$ is coherent then $\CF'$ is coherent.
\end{lem}

\begin{proof}
Follows immediately from the fact that the functor $g_{\dr,*}$ is t-exact
and conservative.
\end{proof}

\begin{proof}[Proof of Proposition~\ref{p:stab under fin}]
We have to show that the functor $\wt{i}^!$ preserves coherence. Applying Lemma~\ref{l:fin coh} to the
morphism $f':\CZ\underset{\CY}\times \wt\CY\to\CZ$, we see that it suffices to prove that the composition
$f'_{\dr,*}\circ \wt{i}^!$ preserves coherence. But $f'_{\dr,*}\circ \wt{i}^!\simeq i^!\circ f_{\dr,*}$ and
each of the functors $i^!$ and $f_{\dr,*}$ preserves coherence. 
\end{proof}

\begin{rem}
One can combine \lemref{l:etale trunc} and \propref{p:stab under fin} to the following
statement: the assertion of \propref{p:stab under fin} continues to hold when $f$ 
is a quasi-finite compactifiable morphism. 
\end{rem}

\ssec{Intersections and unions of truncative substacks}

\begin{lem}  \label{l:intersect trunc}
If $\CZ_1$ and $\CZ_2$ are locally closed truncative substacks of $\CY$, then so is $\CZ_1\cap \CZ_2$.
\end{lem}

\begin{proof}
By \lemref{l:base change trunc lc}, $\CZ_1\cap \CZ_2$ is truncative in $\CZ_1$. Now, the assertion follows
from \lemref{l:trans of trunc}.
\end{proof}

%Finally, we have the following assertion:

\begin{prop}  \label{p:trunc and strat}
Suppose that a locally closed substack $\CZ\subset\CY$ is equal to the union of (possibly intersecting)
locally closed substacks
$\CZ_i$, $i=1,...,n$. If each $\CZ_i$ is truncative in $\CY$, then so is $\CZ$.
\end{prop}

First, let us prove the following particular case of Proposition~\ref{p:trunc and strat}.

\begin{lem}   \label{l:key_case}
Let $\CZ'\hookrightarrow \CZ\hookrightarrow \CY$
be closed embeddings. %Set $\CZ_0:=\CZ-\CZ_1\hookrightarrow \CY$.
If $\CZ'$ and $\CZ-\CZ'$ are truncative in $\CY$ then so is $\CZ$.
\end{lem}

\begin{proof}
Consider the open substacks $\CY-\CZ\subset \CY-\CZ'\subset \CY$. The fact that
$\CZ'$ is truncative in $\CY$ means, by definition, that $\CY-\CZ'$ is co-truncative 
in $\CY$. By \lemref{l:trunct for lc via cl}, the fact that $\CZ-\CZ'$ is truncative in $\CY$ implies that 
$\CZ-\CZ'$ is truncative in $\CY-\CZ'$, i.e., that $\CY-\CZ$ is co-truncative in $\CY-\CZ'$.
But the relation of co-truncativeness is transitive: this is clear if one uses property~(ii) from
Proposition~\ref{p:truncativeness} as a definition of co-truncativeness.
So $\CY-\CZ$ is co-truncative in $\CY$, i.e., $\CZ$ is truncative in $\CY$.
\end{proof}

\begin{proof}[Proof of Proposition~\ref{p:trunc and strat}] We proceed by induction on $n$. 

\medskip

By \corref{c:Zariski-local}, it suffices to show that each $z\in\CZ$ has a Zariski neighborhood
$U\subset\CY$ such that $\CZ\cap U$ is truncative in $U$. Choose $i$ so that $z\in\CZ_i\,$. After replacing
$\CZ$ by an open neighborhood of $z$, one can assume that $\CZ_i$ and $\CZ$ are closed in $\CY$.

\medskip

Writing $\CZ -\CZ_i$ as a union of the substacks $\CZ_j -(\CZ_i\cap \CZ_j)$, $j\ne i$, and applying the 
induction assumption, we see that $\CZ -\CZ_i$ is truncative in $\CY -\CZ_i$ and therefore in $\CY$. 
It remains to apply Lemma~\ref{l:key_case} to $\CZ_i\hookrightarrow \CZ\hookrightarrow \CY$.
\end{proof}

\ssec{Truncativeness and co-truncativeness for non quasi-compact stacks} \label{ss:non-quasicompact}

Now suppose that $\CY$ is \emph{locally} QCA (but not necessarily quasi-compact).

\sssec{}

We give the following definitions:

%First, let us adapt the notions of truncativeness and co-truncativeness to the non quasi-compact situation:

\begin{defn} \hfill

\smallskip

\noindent{\em(i)}
A locally closed substack $\CZ\hookrightarrow \CY$ is said to be truncative if for every open quasi-compact
substack $\oCY\subset \CY$ the intersection $\CZ\cap \oCY$ is truncative in $\oCY$.

\smallskip

\noindent{\em(ii)}
An open substack $U\subset \CY$ is said to be co-truncative if for every open quasi-compact
substack $\oCY\subset \CY$ the intersection $U\cap \oCY$ is co-truncative in $\oCY$.
\end{defn}

\sssec{}

Clearly, a closed substack $\CZ$ is truncative if and only if its complementary open
is co-truncative.

\medskip

In addition:

\begin{lem} \label{l:union_cotruncatives}
If open substacks $U_1,U_2\subset \CY$ are co-truncative then so is $U_1\cup U_2$. 
\end{lem}

\begin{proof}
This immediately follows from \lemref{l:intersect trunc}.
\end{proof}

\sssec{}

As in \lemref{l:tautological}, it is easy to see that $U$ is co-truncative if and only if the functor $j_!$,
left adjoint to $j^*$, is defined. 

\medskip

This formally implies that if $i:\CZ\hookrightarrow \CY$ is truncative, 
then the functor $i^*_\dr$, left adjoint to $i_{\dr,*}$, is also defined. 

\sssec{}

Finally, we note:

\begin{lem} \label{l:! ff}
For a co-truncative open quasi-compact substack
$U\overset{j}\hookrightarrow \CY$ the functor $$j_!:\Dmod(U)\to \Dmod(\CY)$$ is fully faithful.
\end{lem}

\begin{proof}
Follows from \lemref{l:*!}(a).
\end{proof} 

\section{Truncatable stacks}   \label{s:truncatable}

Let $\CY$ be an algebraic stack which is locally QCA. In this setting
the notions of truncativeness and co-truncativeness were introduced in
\secref{ss:non-quasicompact}.

\ssec{The notion of truncatibility}    \label{ss:truncatable}
We will now formulate a condition on $\CY$ called ``truncatibility".
According to \propref{p:truncatable cg} below, it
implies that the category $\Dmod(\CY)$  is compactly generated.

\begin{defn}  \label{defn:truncatable}
The stack $\CY$ is said to be truncatable if it can be covered by open quasi-compact substacks
that are co-truncative.
\end{defn}

\sssec{}

By Lemma~\ref{l:union_cotruncatives}, we can rephrase Definition \ref{defn:truncatable} as follows:

\begin{lem}  \label{l:cofinal}
A stack $\CY$ is truncatable if and only if every open quasi-compact substack is contained in one which
is co-truncative. Equivalently, $\CY$ is truncatable if and only if the sub-poset of co-truncative open 
quasi-compact substacks in $\CY$ is cofinal among all open quasi-compact substacks.
\end{lem}

\sssec{Notation}   \label{sss:Ctrnk}
The poset of co-truncative open quasi-compact substacks $U\subset\CY$ is denoted
by $\on{Ctrnk}(\CY)$; we will often consider this poset as a category.  Let $\on{Ctrnk}(\CY)^{\on{op}}$ denote the opposite poset (or category).
%(i.e., $\on{Ctrnk}(\CY)^{\on{op}}$ and $\on{Ctrnk}(\CY)$ are equal as sets but the orderings are inverse 
%to each other). 
Lemma~\ref{l:union_cotruncatives} implies that $\on{Ctrnk}(\CY)$ is \emph{filtered}. 

The next statement immediately follows from \lemref{l:cofinal}.

\begin{cor} \label{c:cofinal}
If $\CY$ is truncatable then the natural restriction functor
$$\Dmod(\CY)\to \underset{U\in \on{Ctrnk}(\CY)^{\on{op}}}{\underset{\longleftarrow}{lim}}\, \Dmod(U)$$
is an equivalence.
%; here the limit is taken over the poset of co-truncative  open quasi-compact substacks of $\CY$.
\end{cor}

\begin{prop}  \label{p:truncatable cg}
If $\CY$ is truncatable then the category $\Dmod(\CY)$ is compactly generated.
\end{prop}

\begin{proof}
Let $U\overset{j}\hookrightarrow \CY$ be a co-truncative open quasi-compact substack and 
$\CF_U\in \Dmod(U)^c$. By \propref{p:tautological compact}, the object $j_!(\CF_U)\in \Dmod(\CY)$ 
(which is well-defined by the co-truncativeness assumption) is compact. It suffices to show
that such objects generate $\Dmod(\CY)$. In other words, we have to show that if $\CF\in \Dmod(\CY)$
is right-orthogonal to all such objects, then $\CF=0$.

\medskip

For a given $U$, the fact that $\CF$ is right-orthogonal to all $j_!(\CF_U)$ as above is equivalent,
by adjunction, to the fact that $j^*(\CF)$ is right-orthogonal to $\Dmod(U)^c$. Since $\Dmod(U)$
is compactly generated, this implies that $j^*(\CF)=0$. By \corref{c:cofinal}, this implies that 
$\CF=0$.

\end{proof}

\sssec{}

As was mentioned in the introduction, we 
use \propref{p:truncatable cg} to deduce 
the main result of this paper (namely, 
the compact generation of $\Dmod(\Bun_G)$)
from the following result:

\begin{thm} \label{t:main truncatable}
Let $G$ be a connected reductive group and $X$  a smooth complete connected curve over $k$.
Let $\Bun_G$ denote the stack of $G$-bundles on $X$.
Then $\Bun_G$ is truncatable.
\end{thm}

The proof for any connected reductive group $G$ will be given in Sect.~\ref{s:delo}. 
But its main idea is the same as in the easy case $G=SL_2\,$, which is considered
separately  in \secref{ss:SL 2}. 
%Note that \secref{ss:SL 2} is independent from Sects. \ref{ss:colimits abs}-\ref{ss:miracoli}. 
 
\sssec{}

In Sects. \ref{ss:category as colimit}-\ref{ss:miracoli} below we discuss some general
properties of the category $\Dmod(\CY)$ for a truncatable stack $\CY$.

\ssec{Presentation as a colimit}  \label{ss:category as colimit}

In this subsection we fix $\CY$ to be a truncatable locally QCA stack.
We will use the notation $\on{Ctrnk}(\CY)$ from \secref{sss:Ctrnk}.

\sssec{}  \label{sss:colimit descr}

 Note that for  a morphism $U_1\overset{j_{1,2}}\hookrightarrow U_2$ in $\on{Ctrnk}(\CY)$, the pullback functor 
$$\phi_{U_2,U_1}:=j_{1,2}^*:\Dmod (U_2)\to \Dmod (U_1)$$ 
admits a left adjoint, $\psi_{U_1,U_2}:=(j_{1,2})_!:\Dmod (U_1)\to \Dmod (U_2)$. 

\medskip

Hence, we are in the situation of Sect.~\ref{sss:let us recall} with $I=\on{Ctrnk}(\CY)$. 
In fact, we are in the more restrictive (and possibly more understandable) situation of Sects.~\ref{sss:understandable} and \ref{ss:Colimcompgen}.

\sssec{}

Combining the assertion of \corref{c:cofinal} with that of \propref{p:limit and colimit a}, we obtain:

\begin{cor} \label{c:pres as colimit}
The category $\Dmod(\CY)$ is canonically equivalent to 
$$\underset{U\in \on{Ctrnk}(\CY)}{\underset{\longrightarrow}{colim}}\, \Dmod(U),$$
where the functor $\on{Ctrnk}(\CY)\to \on{DGCat}_{\on{cont}}$ is
$$U\mapsto \Dmod(U),\quad (U_1\overset{j_{1,2}}\hookrightarrow U_2) \mapsto (j_{1,2})_!.$$
Under this equivalence, for a co-truncative open quasi-compact substack
$U_0\overset{j_0}\hookrightarrow \CY$, the functor
$$\on{ins}_{U_0}:\Dmod(U_0)\to \underset{U\in \on{Ctrnk}(\CY)}{\underset{\longrightarrow}{colim}}\, \Dmod(U)\simeq \Dmod(\CY),$$
is $(j_0)_!$.
\end{cor}

\begin{rem}  \label{r:tautological compact}
Note that the assertion of \propref{p:tautological compact} for a truncative QCA stack $\CY$ follows 
also from \lemref{e:filtered_lemma}(i). Note also that the assertion of \lemref{l:*!} for $U$ (resp., $U$ and $U'$) co-truncative
is a particular case of Remark \ref{r:colim filtered 1}.
\end{rem}

%\begin{rem}
%Note that the assertion of \propref{p:tautological compact} for a truncative QCA stack $\CY$ follows formally
%from %\lemref{l:filtered2}. 
%\corref{c:general and filtered}(ii).
%Note also that the assertion of \lemref{l:*!} for $U$ (resp., $U$ and $U'$) co-truncative
%is a particular case of \corref{c:ff into colimit}.
%\end{rem}

\ssec{Description of the dual category}  \label{ss:descr dual}

\sssec{}

Combining \corref{c:cofinal} with \propref{p:dual of limit} we obtain:

\begin{cor}   \label{c:dual as colimit}
The category $\Dmod(\CY)$ is dualizable. Its dual category is canonically equivalent to
\begin{equation}   \label{e:co}
\underset{U\in \on{Ctrnk}(\CY)}{\underset{\longrightarrow}{colim}}\, \Dmod(U),
\end{equation}
where the functor $\on{Ctrnk}(\CY)\to \on{DGCat}_{\on{cont}}$ is
\begin{equation} \label{e:move on}
U\mapsto \Dmod(U),\quad (U_1\overset{j_{1,2}}\hookrightarrow U_2) \mapsto (j_{1,2})_*.
\end{equation}

Under this equivalence, for a co-truncative open quasi-compact substack
$U_0\overset{j_0}\hookrightarrow \CY$, the functor
$$\on{ins}_{U_0}:\Dmod(U_0)\to \underset{U\in \on{Ctrnk}(\CY)}{\underset{\longrightarrow}{colim}}\, \Dmod(U)$$
is the dual of restriction functor $j_0^*:\Dmod(\CY)\to \Dmod(U_0)$. 
\end{cor}

\begin{proof}
Follows from \propref{p:Lurie-duality}.
\end{proof}

\sssec{Notation}  \label{sss:def of co}
The category \eqref{e:co}
%$$\underset{U\in \on{Ctrnk}(\CY)}{\underset{\longrightarrow}{colim}}\, \Dmod(U)$$ for the
%functor $\on{Ctrnk}(\CY)\to \on{DGCat}_{\on{cont}}$ given by \eqref{e:move on}, that 
that appears in \corref{c:dual as colimit} will be denoted by 
$$\Dmod(\CY)_{\on{co}}\,.$$

The equivalence of \corref{c:dual as colimit} will be denoted by
\begin{equation} \label{e:non-qc Verdier}
\bD_{\CY}^{\on{Verdier}}:\Dmod(\CY)^\vee\simeq \Dmod(\CY)_{\on{co}\;}.
\end{equation}

Note that when $\CY$ is quasi-compact, this is the same as the equivalence of
\eqref{e:Verdier}.

\sssec{}

Combining \corref{c:dual as colimit} with \propref{p:limit and colimit a}, we can rewrite $\Dmod(\CY)_{\on{co}}$
also as a limit:

\begin{cor}   \label{c:dual as limit}
The category $\Dmod(\CY)_{\on{co}}$ is canonically equivalent to
$$\underset{U\in \on{Ctrnk}(\CY)^{\on{op}}}{\underset{\longleftarrow}{lim}}\, \Dmod(U),$$
where the functor $\on{Ctrnk}(\CY)^{\on{op}}\to \on{DGCat}_{\on{cont}}$ is
$$U\mapsto \Dmod(U),\quad (U_1\overset{j_{1,2}}\hookrightarrow U_2) \mapsto j_{1,2}^?.$$
\end{cor}

\sssec{}

By construction, for every co-truncative quasi-compact open substack $U\overset{j}\hookrightarrow \CY$,
we have a canonically defined functor 
$$\Dmod(U)\to \Dmod(\CY)_{\on{co}}.$$

We denote this functor by $j_{\on{co},*}$. By construction, in terms of the identifications
$$\bD_U^{\on{Verdier}}:\Dmod(U)^\vee\simeq \Dmod(U) \text{ and } 
\bD_\CY^{\on{Verdier}}:\Dmod(\CY)^\vee\simeq \Dmod(\CY)_{\on{co}},$$
we have
$$(j_{\on{co},*})^\vee\simeq j^*.$$

\medskip

Similarly, from \corref{c:dual as limit}, we have a canonically defined functor
$$j^?:\Dmod(\CY)_{\on{co}}\to \Dmod(U),$$
which is the dual of $j_!:\Dmod(U)\to \Dmod(\CY)$, and the right adjoint of $j_{\on{co},*}\;$.

\sssec{}

We claim:

\begin{lem} \label{l:j_co fullyfaith}
The functor $j_{\on{co},*}$ is fully faithful. 
\end{lem}

\begin{proof}
We need to show that the unit of the adjunction $\on{Id}_{\Dmod(U)}\to j^?\circ j_{\on{co},*}$
is an isomorphism. This is obtained by passing to dual functors (see \secref{sss:dual functor}) in the map
$$\on{Id}_{\Dmod(U)}\to j^*\circ j_!,$$
which is an isomorphism by \lemref{l:! ff}.
\end{proof}

\begin{rem}
Note that \lemref{l:j_co fullyfaith} follows more abstractly from Remark \ref{r:colim filtered 1}. However,
this way to deduce \lemref{l:j_co fullyfaith} is equivalent to the proof given above in view of 
Remark \ref{r:tautological compact}.
\end{rem}

\sssec{}

We claim that the category $\Dmod(\CY)_{\on{co}}$ is compactly generated and that its
compact objects are ones of the form $j_{\on{co},*}(\CF_U)$ for
$\CF_U\in \Dmod(U)^c$, where $U$ is a co-truncative quasi-compact open substack of $\CY$. 

\medskip

This follows from \propref{p:tautological compact} and \secref{sss:dual of c g}. 

\medskip

%Alternatively, this follows from %Lemmas \ref{l:limit and colimit b} and \ref{l:filtered2}.
%\corref{c:general and filtered}.
Alternatively, this follows from \corref{c:general not filtered} and \lemref{e:filtered_lemma}(i).

\ssec{Relation between the category and its dual}   \label{ss:relation_with_dual}

In this subsection we continue to assume that $\CY$ is a truncatable locally QCA stack.

\sssec{}  \label{sss:kernels}

By construction and \secref{sss:dualizable category}, the DG category $\on{Funct}_{\on{cont}}(\Dmod(\CY)_{\on{co}},\Dmod(\CY))$ identifies
canonically with
$$(\Dmod(\CY)_{\on{co}})^\vee\otimes \Dmod(\CY)\simeq \Dmod(\CY)\otimes \Dmod(\CY).$$

In addition, by \propref{p:D on prod} and Remark \ref{r:D on prod}, we have
$$ \Dmod(\CY)\otimes \Dmod(\CY) \simeq \Dmod(\CY\times \CY).$$

Thus, every object $\CQ\in \Dmod(\CY\times \CY)$ defines a functor 
$$\sF_\CQ:\Dmod(\CY)_{\on{co}}\to \Dmod(\CY).$$

\sssec{The naive functor}

Note that if $\CY$ is quasi-compact we have a tautological equivalence $$\Dmod(\CY)_{\on{co}}\simeq \Dmod(\CY).$$
Recall from \secref{sss:Verdier} that the corresponding object in $\Dmod(\CY\times \CY)$ is $(\Delta_\CY)_{\dr,*}(\omega_\CY)$. 

\medskip

For \emph{any} truncatable $\CY$ the functor $\Dmod(\CY)_{\on{co}}\to \Dmod(\CY)$
corresponding to $$(\Delta_\CY)_{\dr,*}(\omega_\CY)\in \Dmod(\CY\times \CY)$$ will be denoted by
$$\psId_{\CY,\on{naive}}:\Dmod(\CY)_{\on{co}}\to \Dmod(\CY)$$
(here $\psId$ stands for ``pseudo-identity").

%For a non quasi-compact truncatable $\CY$ we define the functor
%$$\psId_{\CY,\on{naive}}:\Dmod(\CY)_{\on{co}}\to \Dmod(\CY)$$
%to be given by the object\footnote{$\psId$ stands for ``pseudo-identity".}
%$$(\Delta_\CY)_{\dr,*}(\omega_\CY)\in \Dmod(\CY\times \CY).$$

\medskip

Let $\bD^{\on{Verdier}}_{\CY,\on{naive}}:\Dmod(\CY)^\vee\to\Dmod(\CY)$ denote the composition 
%resulting functor
$$\Dmod(\CY)^\vee\overset{\bD^{\on{Verdier}}_\CY}\simeq \Dmod(\CY)_{\on{co}}\overset{\psId_{\CY,\on{naive}}}\longrightarrow
\Dmod(\CY).$$

\sssec{An alternative description}  \label{sss:alt descr}

Here is a tautologically equivalent description of the functor 
$\psId_{\CY,\on{naive}}:\Dmod(\CY)_{\on{co}}\to \Dmod(\CY)$.

\medskip

By definition, to specify a continuous functor $\sF$ from $\Dmod(\CY)_{\on{co}}$ to an arbitrary
DG category $\bC$, is equivalent to
specifying a compatible collection of functors $\sF_U:\Dmod(U)\to \bC$ for co-truncative quasi-compact
open substacks $U\subset \CY$. The compatibility condition reads that for 
$U_1\overset{j_{1,2}}\hookrightarrow U_2$, we must be given a (homotopy-coherent)
system of isomorphism 
$$\sF_{U_1}\simeq \sF_{U_2}\circ (j_{1,2})_*.$$

Taking $\bC=\Dmod(\CY)$, the corresponding functors $(\psId_{\CY,\on{naive}})_U$ are
$$j_*:\Dmod(U)\to \Dmod(\CY)$$
for $U\overset{j}\hookrightarrow \CY$. 

\sssec{Warning}

For a general truncatable stack $\CY$, the functor $\psId_{\CY,\on{naive}}$ is \emph{not} an equivalence.
In particular, it is \emph{not} an equivalence for $\CY=\Bun_G$ unless $G$ is solvable. 

\medskip

In fact, we have the following assertion:

\begin{prop}   \label{p:idnaive}
If the functor $\psId_{\CY,\on{naive}}:\Dmod(\CY)_{\on{co}}\to \Dmod(\CY)$ is an equivalence then the closure of any
quasi-compact open substack of $\CY$ is quasi-compact.
\end{prop}

The converse statement is also true (for tautological reasons). 

\medskip

The proof of Proposition~\ref{p:idnaive} given below is based on the following lemma.

\begin{lem} \label{holonomic}
Let $Z$ be a quasi-compact scheme, $U$ a QCA  stack, and $f:Z\to U$ a morphism.
Then for any holonomic D-module $\CF$ on $Z$ the object $f_{\dr,*}(\CF)\in \Dmod(U)$ is compact.
\end{lem}

Let us give two proofs:

\begin{proof}[Proof 1]
This follows from the following general observation:

\begin{lem}
Let $\sF:\bC_1\to \bC_2$ be a continuous functor between cocomplete DG categories. Let
$\bc_2\in \bC_2^c$ be such that the partially defined left adjoint $\sF^L$ to $\sF$ is defined on $\bc_2$.
Then $\sF^L(\bc_2)\in \bC_1$ is compact.
\end{lem}

The functor $f_!$, left adjoint to $f^!$ is defined on holonomic objects. Hence, by the above lemma,
$f_!(\BD_Z^{\on{Verdier}}(\CF))\in \Dmod_{\on{coh}}(U)$ is compact. By \thmref{t:Verdier QCA},
$$\BD_U^{\on{Verdier}}(f_!(\BD_Z^{\on{Verdier}}(\CF)))\simeq f_{\dr,*}(\CF)$$
is compact, as required. 

\end{proof}

\begin{proof}[Proof 2]
The object $f_{\dr,*}(\CF)$ is holonomic and therefore coherent. Since $Z$ is a scheme, by \thmref{t:ccoh}(ii),
$\CF$ is safe. By \cite[Lemma 9.4.2]{finiteness} we obtain that $f_{\dr,*}(\CF)$ is also safe. Thus, 
$f_{\dr,*}(\CF)$ is coherent and safe = compact.
\end{proof}

\begin{proof}[Proof of Proposition~\ref{p:idnaive}]

Suppose that $\psId_{\CY,\on{naive}}$ is an equivalence. Since $\CY$ is truncatable, it is enough to show 
that the closure of every co-truncative open quasi-compact substack is quasi-compact.

\medskip

By assumption, the functor $\psId_{\CY,\on{naive}}$ preserves compactness. From 
\secref{sss:alt descr}, we obtain that $\psId_{\CY,\on{naive}}$ sends a compact object 
$j_{\on{op},*}(\CF_U)\in \Dmod(\CY)_{\on{op}}$, $\CF_U\in \Dmod(U)^c$ with $U\overset{j}\hookrightarrow \CY$ co-truncative
and quasi-compact, 
to $j_*(\CF_U)\in \Dmod(\CY)$. Thus, we obtain that $j_*(\CF_U)$ needs to be compact 
for any $\CF_U\in \Dmod(U)^c$ whenever $U$ is co-truncative.

\medskip

Take $\CF_U=f_{\dr,*}(k_Z)$, where $Z$ is any quasi-compact scheme equipped with a morphism $f:Z\to U$ and $k_Z$ is 
the ``constant sheaf" on $Z$. By ~\propref{p:tautological compact}, there exists a quasi-compact open substack
$V\subset\CY$ such that the $*$-stalk of $j_{\dr,*}(\CF_U)=(j\circ f)_{\dr,*}(k_Z)$ over any
point of $\CY-V$ is zero. This means that the closure of the image of $j\circ f:Z\to \CY$
is contained in $V$ and therefore quasi-compact. Taking $f$ surjective we see that the closure of
$U$ is quasi-compact.
\end{proof}

\sssec{A better functor}  \label{sss:better}
Following \cite[Sect. 6]{Ker}, we define
$$\psId_{\CY,!}:\Dmod(\CY)_{\on{co}}\to \Dmod(\CY)$$
to be the functor corresponding in terms of Sect.~\ref{sss:kernels} to the object
$$(\Delta_\CY)_!(k_\CY)\in \Dmod(\CY\times \CY),$$
where $k_\CY\in \Dmod(\CY)$ is the ``constant sheaf" on $\CY$. 
(The above object is well-defined because $k_\CY$ is holonomic.)

%Another functor 
%$$\psId_{\CY,!}:\Dmod(\CY)_{\on{co}}\to \Dmod(\CY)$$
%is constructed in \cite[Sect. 6]{Ker}.
%
%\medskip
%
%Namely, in terms of \secref{sss:kernels}, 
%it is given by the object
%$$(\Delta_\CY)_!(k_\CY)\in \Dmod(\CY\times \CY),$$
%where $k_\CY\in \Dmod(\CY)$ is the ``constant sheaf" on $\CY$. 
%(The above object is well-defined because $k_\CY$ is holonomic.)

\medskip

Let $\bD^{\on{Verdier}}_{\CY,!}:\Dmod(\CY)^\vee\to\Dmod(\CY)$ denote the composition 
$$\Dmod(\CY)^\vee\overset{\bD^{\on{Verdier}}_\CY}\simeq \Dmod(\CY)_{\on{co}}\overset{\psId_{\CY,!}}\longrightarrow
\Dmod(\CY).$$

\sssec{}

Suppose for a moment that $\CY$ is smooth of dimension $n$, and that the diagonal map
$$\Delta_\CY:\CY\to \CY\times \CY$$
is separated. 
%
%\medskip
%
In this case we have an isomorphism 
$$k_\CY\simeq \omega_\CY[-2n],$$
and a natural transformation 
$$(\Delta_\CY)_!\to (\Delta_\CY)_{\dr,*}\, ,$$
which together define a natural transformation
\begin{equation} \label{e:naive to !}
\psId_{\CY,!}\to \psId_{\CY,\on{naive}}[-2n].
\end{equation}

\begin{rem}
%Let $\CY$ be such that the diagonal morphism of $\CY$ is a closed embedding. In this case
%\eqref{e:naive to !} is an isomorphism. However, obviously, for general algebraic stacks, 
%$\Delta_\CY$ is far from being a closed embedding. 
If $\CY$ is separated (i.e., if $\Delta_\CY$ is proper) then \eqref{e:naive to !} is an isomorphism.
However, most stacks are not separated.\footnote{A separated locally QCA stack has to be a Deligne-Mumford stack. Indeed, if $\Delta_\CY$ is proper and affine then it is finite, and in characteristic 0 this means that $\CY$ is Deligne-Mumford.}  Thus $\psId_{\CY,!}$ is usually \emph{different} from 
$\psId_{\CY,\on{naive}}$ (even for $\CY$ smooth and quasi-compact).
\end{rem}

\sssec{}

Here is a basic feature of the functor $\psId_{\CY,!}:\Dmod(\CY)_{\on{co}}\to \Dmod(\CY)$.

\begin{lem}   \label{l:basic feature}
Let $U\overset{j}\hookrightarrow \CY$ be a co-truncative quasi-compact open substack. Then there exists
a canonical isomorphism of functors $\Dmod(U)\to \Dmod(\CY)$:
$$\psId_{\CY,!}\circ j_{\on{co},*}\simeq j_!\circ \psId_{U,!}.$$
\end{lem}

\begin{proof}
Define $\Delta_{U,\CY}:U\to U\times\CY$ by $\Delta_{U,\CY}(u):=(u,j(u))$.
It is easy to check that both functors $\psId_{\CY,!}\circ j_{\on{co},*}$ and $j_!\circ \psId_{U,!}$
correspond to the object $(\Delta_{U,\CY})_!(k_U)\in \Dmod(U\times\CY )$
via the equivalence
\begin{multline*}
\Dmod(U\times\CY)\simeq \Dmod(U)\otimes\Dmod(\CY)\\
\simeq \Dmod(U)^\vee\otimes \Dmod(\CY) \simeq \on{Funct}_{\on{cont}}(\Dmod(U),\Dmod(\CY)).
\end{multline*}
%%%are given by the object
%%%\begin{multline*}
%%%(j\times \on{id}_U)_!(k_U)\in \Dmod(\CY\times U)\simeq \Dmod(\CY)\otimes \Dmod(U)\simeq \\
%%%\simeq \Dmod(\CY)\otimes \Dmod(U)^\vee\simeq \on{Funct}_{\on{cont}}(\Dmod(U),\Dmod(\CY)).
%%%\end{multline*}
%Define $\Delta_{\CY,U}:U\to\CY\times U$ by $\Delta_{\CY,U}(u):=(j(u),u)$.
%It is easy to check that both functors $\psId_{\CY,!}\circ j_{\on{co},*}$ and $j_!\circ \psId_{U,!}$
%correspond to the object $(\Delta_{\CY,U})_!(k_U)\in \Dmod(\CY\times U)$
%via the equivalence
%\begin{multline*}
%\Dmod(\CY\times U)\simeq \Dmod(\CY)\otimes \Dmod(U)\simeq \\
%\simeq \Dmod(\CY)\otimes \Dmod(U)^\vee\simeq \on{Funct}_{\on{cont}}(\Dmod(U),\Dmod(\CY)).
%\end{multline*}
\end{proof}

The meaning of this lemma is that \emph{the functor $\psId_{\CY,!}$ sends objects that are $*$-extensions from a
co-truncative quasi-compact open substack in $\Dmod(\CY)_{\on{op}}$ to objects in $\Dmod(\CY)$ that are !-extensions}
(from the same open). 

\sssec{Self-duality} \label{sss:delf-duality}
Both functors 
%$\psId_{\CY,\on{naive}}:\Dmod(\CY)_{\on{co}}\to \Dmod(\CY)$ and
%$\psId_{\CY,!}:\Dmod(\CY)_{\on{co}}\to \Dmod(\CY)$ 
$$\bD^{\on{Verdier}}_{\CY,\on{naive}}:\Dmod(\CY)^\vee\to\Dmod(\CY), \quad
\bD^{\on{Verdier}}_{\CY,!}:\Dmod(\CY)^\vee\to\Dmod(\CY)$$
are canonically self-dual because the corresponding objects 
$$(\Delta_\CY)_{\dr,*}(\omega_\CY),\, (\Delta_\CY)_!(k_\CY)\in \Dmod(\CY\times \CY)$$
are equivariant with respect to the action of the symmetric group $S_2$ on $\CY\times\CY$.

\ssec{Miraculous stacks}  \label{ss:miracoli}

\sssec{} \label{sss:not always equivalence}

%We will not pursue the study of the functor $\psId_{\CY,!}$ in this paper. However, we 
Now let us give the following definition.

\begin{defn}
A truncatable stack $\CY$ is called \emph{miraculous} if the functor 
$$\psId_{\CY,!}:\Dmod(\CY)_{\on{co}}\to \Dmod(\CY)$$ is an equivalence.
\end{defn}

Clearly this happens if and only if
%Tautologically, for a miraculous stack $\CY$, 
the functor $\bD^{\on{Verdier}}_{\CY,!}:\Dmod(\CY)^\vee\to \Dmod(\CY)$
is an equivalence.

\sssec{}

The following easy lemma shows that
not every algebraic stack is miraculous. %E.g., the following is easy:

\begin{lem}
A separated quasi-compact scheme $Z$ is a miraculous stack if and only if $Z$ has the following ``cohomological smoothness" property: $k_Z$ and 
$\omega_Z$ are locally isomorphic up to a shift. %a lisse D-module.
\end{lem}

\begin{proof}
In our situation $\psId_{Z,!}:\Dmod(Z)\to \Dmod(Z)$ is the functor $M\mapsto M\overset{!}\otimes k_Z\,$.
Applying the functor to skyscrapers, we see that if $\psId_{Z,!}$ is an equivalence then each $!$-stalk of 
$k_Z$ is isomorphic to $k$ up to a shift. It is well known that this implies that $k_Z$ and $\omega_Z$ are locally isomorphic up to a shift. 
\end{proof}

It is also easy to produce an example of a \emph{smooth} quasi-compact algebraic stack $\CY$ which is not miraculous: 
it suffices to take $\CY$ to be
%
%\medskip
%For example take 
the non-separated scheme equal to $\BA^1$ with a double point $0$.
We refer the reader to \cite[Sect. 5.3.5]{Ker}, where this example is analyzed (one easily
shows that in this case the functor $\psId_{\CY,!}$ does not preserve compactness). 

\sssec{}

A basic example of a miraculous stack is $\CY:=\BA^n/\BG_m$; see \cite[Corollary 5.3.4]{Ker}.

\medskip

In addition, the following theorem is proved in \cite{self-duality}:

\begin{thm}
Let $G$ be a reductive group. Then the stack $\Bun_G$ is miraculous.
\end{thm}

This theorem is equivalent to each quasi-compact co-truncative substack of $\Bun_G$ being miraculous. 
The equivalence follows from the next lemma.
%The next lemma shows that Theorem~\ref{?} is equivalent to each 
\begin{lem}
A truncatable stack $\CY$ is miraculous if and only if every quasi-compact co-truncative open substack 
$U\subset\CY$ is.
\end{lem}

\begin{proof}
The ``if" statement follows from \lemref{l:basic feature} and the descriptions of $\Dmod (\CY)$ and
$\Dmod (\CY)_{\on{co}}$ as colimits (see \corref{c:pres as colimit} and Sect.~\ref{sss:def of co}).

\medskip

Let us prove the ``only if" statement. Suppose that $\CY$ is miraculous and $j:U\hookrightarrow\CY$ 
is a quasi-compact co-truncative open substack. The functor $j_!$ has a left inverse (namely, $j^!=j^*$). The functor
$ j_{\on{co},*}:\Dmod(U)_{\on{co}}\to\Dmod(\CY)_{\on{co}}$ also has a left inverse 
(see the first proof of \lemref{l:j_co fullyfaith}).
So \lemref{l:basic feature} implies that $\psId_{U,!}$ has a left inverse. 

\medskip

We obtain that the functor
$\bD^{\on{Verdier}}_{U,!}:\Dmod(U)^\vee\to\Dmod(U)$ has a left inverse. By self-duality of 
$\bD^{\on{Verdier}}_{U,!}$  (see Sect.~\ref{sss:delf-duality}), this implies that it has a right inverse as well.
So $\bD^{\on{Verdier}}_{U,!}$ is an equivalence.
%By Lemmas~\ref{l:! ff} and \ref{l:j_co fullyfaith}, the functors 
%$j_!$ and $j_{\on{co},*}$ are fully faithful, so \lemref{l:basic feature} implies that $\psId_{U,!}$ is fully faithful.
%
%How to prove essential surjectivity?
\end{proof}

\section{Contractive substacks}  \label{ss:est trunc}
In its simplest form, the contraction principle says that the substack $\{ 0\}/\BG_m\hookrightarrow \BA^n/\BG_m$ 
is truncative (here $\BG_m$ acts on $\BA^n$ by homotheties). In this section we will prove a generalization of
this fact, see \propref{p:2contraction principle}. 

\medskip

In \secref{ss:adjointness} we explicitly describe the non-standard functors 
$i^*_\dr$ and $i_?$ in the setting of ~\propref{p:2contraction principle}.

\medskip

We say that a substack of a stack is  \emph{contractive} if it locally satisfies the conditions of 
\propref{p:2contraction principle}; for a precise definition, see \secref{sss:contractive}. 

\medskip

The upshot of this section is that ``contractiveness" $\Rightarrow$ ``truncativeness." 

\ssec{The contraction principle}    \label{ss:elem contr}

\sssec{}   \label{sss:set up for elem contr}
Consider the following set-up. Suppose we  have an affine morphism $p:W\to  S$ between schemes.
Assume that the monoid $\BA^1$ (with respect to multiplication) acts on $W$ over $S$
(so that the action of $\BA^1$ on $S$ is trivial). Assume also that the endomorphism of $W$ corresponding to
$0\in \BA^1$ admits a factorization
$$W\overset{p}\to S\overset{\iota}\to W,$$
where $\iota$ is a section of $p:W\to  S$.
(Informally, we can say that the action of $\BG_m\subset  \BA^1$ ``contracts" $W$ onto 
the closed subscheme $\iota (S)$.) 

\medskip

Set  $\CY :=W/\BG_m$, 
$\CZ:=S/\BG_m=S\times (\on{pt}/\BG_m)$.

\begin{prop} \label{p:2contraction principle}
Under the above circumstances, the closed substack $\CZ\overset{i}\mono \CY$ is truncative.
\end{prop}

The rest of this subsection is devoted to the proof of \propref{p:2contraction principle}.

\sssec{}

Without loss of generality, we can assume that $S$ is quasi-compact.

\medskip

We have $$W=\Spec_S(\CA),$$ where $\CA=\bigoplus\limits_n\CA_n$ is a quasi-coherent sheaf of non-negatively
graded $\CO_S$-algebras with $\CA_0= \CO_S$. The section $\iota$ corresponds to the
projection $\CA\to \CA_0=\CO_S$.

\medskip

For $n\in\BN$, let $\CA^{(n)}\subset \CA$ be the $\CO_S$-subalgebra 
generated by $\CA_n\,$. Choose $n$ so that $\CA$ is finite over $\CA^{(n)}$
(if $\CA$ is generated by $\CA_{m_1},\ldots ,\CA_{m_r}$ then one can take $n$ to be the least common multiple of 
$m_1$,\ldots, $m_r$). Set $W':=\Spec(\CA^{(n)})$, then the morphism
$f:W\to W'$ is finite. Moreover, the embedding $\iota (S)\mono f^{-1}(f(\iota (S)))$ induces an isomorphism 
between the corresponding reduced schemes. So by \propref{p:stab under fin},
it suffices to prove the proposition for $W'$ instead of $W$.  

\sssec{}

Thus, we can assume that $\CA$ is generated by $\CA_n$.
Moreover, since the proposition to be proved is local with respect to $S$ 
(see \corref{c:Zariski-local}), we can assume that $\CA_n$ is a quotient of a locally free $\CO_S$-module $\CE$. 
Let $V$ denote the vector bundle over $S$ corresponding to $\CE^*$ (in other words, $V$ is the spectrum 
of the symmetric algebra of $\CE$). Then $W=\Spec(\CA)$ identifies with a closed conical subscheme in $V$. 

\sssec{}        \label{sss:vector_bundle_situation}
Thus by \lemref{l:base change trunc lc} (for closed embeddings) it suffices to consider the case where 
$W$ is a vector bundle $V$ over $S$ equipped with
an $\BA^1$-action obtained from the standard one by composing it with the homomorphism
\begin{equation} \label{e:raising to power n}
\BA^1\longrightarrow\BA^1, \quad \lambda\mapsto\lambda^n.
\end{equation}
(here $n$ is some positive integer). In this situation we have to prove that 
${\bf 0}/\BG_m\subset V/\BG_m$ is truncative, where  ${\bf 0 }\subset V$ is the zero section.

\sssec{}
Let $\wt{V}\overset{f}\to V$ be the blow-up of $V$ along ${\bf 0}$. Set ${\wt{\bf 0}}:=f^{-1}({\bf 0})$.
Since $f$ is proper and surjective, by \lemref{l:proper_descent}, in order to prove the truncativeness of 
${\bf 0}/\BG_m\subset V/\BG_m$, it suffices to show that ${\wt{\bf 0}}/\BG_m$ is truncative in 
${\wt{V}}/\BG_m$. 

\medskip

Note that $\wt{V}$ is a line bundle over $\BP (V)$ and $\wt{\bf 0}$ is its zero section. 
So we see that it suffices to prove the
statement from \secref{sss:vector_bundle_situation} for line bundles over arbitrary bases. 
Moreover, since the statement is local, it suffices to consider the trivial line bundle over an arbitrary quasi-compact scheme.

\sssec{}
Thus it remains to prove that for any quasi-compact scheme $S$ the substack 
$$S\times (\{ 0\}/\BG_m )\subset S\times (\BA^1/\BG_m )$$ is truncative (here we assume that
$\lambda\in \BG_m$ acts on $\BA^1$ as multiplication by $\lambda^n$ for some $n\in\BN$).
By \lemref{l:trunc and prod}, we can assume that $S=\Spec(k)$. In this case the statement 
follows from the fact that the number of $\BG_m$-orbits in $\BA^1$ is finite 
(see \secref{sss:finitely many points}). \qed

\ssec{Contractiveness}

\sssec{}  \label{sss:contractive}

We say that a  locally closed substack $\CZ'$ of a stack $\CY'$ is \emph{contractive} 
if there exists a commutative diagram

\begin{equation}   \label{e:strongly contractive}
\xymatrix{
\CZ \ar[d]_{}\ar[r]^{}& \CY\ar[d]^{}\\
\CZ'\ar[r]^{}&\CY'
    }
\end{equation}

such that 
\begin{enumerate}
\item[(i)] the upper row of \eqref{e:strongly contractive} is as in \propref{p:2contraction principle};  

\medskip

\item[(ii)] the morphism $\CZ\to\CZ'\underset{\CY'}\times \CY$ is an open embedding;

\smallskip

\item[(iii)]  the vertical arrows of \eqref{e:strongly contractive} are smooth and the left one is surjective.
\end{enumerate}

\medskip

In other words, a substack is contractive if it locally satisfies the conditions of 
~\propref{p:2contraction principle}.

\sssec{}

We obtain:

\begin{cor} \label{c:contraction principle}
A contractive substack is truncative.
\end{cor}

\begin{proof}
With no loss of generality, we can assume that $\CY$ is quasi-compact. Now combine
\propref{p:2contraction principle} and \ref{p:truncativeness after smooth}.
\end{proof}

Note that the above definition of contractive substack makes sense without the characteristic~0 assumption.

\begin{rem}
We are not sure that the notion of contractiveness is really good.
But it is convenient for the purposes of this article.
\end{rem}

\ssec{An adjointness result}    \label{ss:adjointness}

\sssec{}
Let $W\to S\overset{\iota}\to W$ be as in \propref{p:2contraction principle} (in particular, $\BA^1$ acts on $W$).

\medskip

Consider the corresponding morphisms $i: S/\BG_m\mono W/\BG_m$ and 
$\pi: W/\BG_m\to S/\BG_m\,$. By \propref{p:2contraction principle}, the functor
 $i_{\dr,*}$ has a continuous left adjoint (denoted by $i^*_\dr$) and $i^!$ has a continuous right adjoint 
 (denoted by $i_?$). 
 
 \medskip
 
The next proposition identifies the non-standard functors
$$i^*_\dr:\Dmod (W/\BG_m)\to\Dmod (S/\BG_m) \text{ and } i_?:\Dmod (S/\BG_m)\to\Dmod (W/\BG_m)$$  with certain standard functors. 
Namely, $$i^*_\dr\simeq \pi_{\dr,*},\quad i_?= \pi^!.$$

\begin{prop} \label{p:adjointness}
The functors 
\[
\pi_{\dr,*}:\Dmod (W/\BG_m )\rightleftarrows \Dmod (S/\BG_m ):i_{\dr,*}\, \text{ and }
i^!:\Dmod (W/\BG_m ) \rightleftarrows \Dmod (S/\BG_m ):\pi^!
\]
form adjoint pairs with the adjunctions  $$\pi_{\dr,*}\circ i_{\dr,*}\iso\Id_{\Dmod (S/\BG_m )} \text{ and } 
 i^!\circ \pi^!\iso \Id_{\Dmod (S/\BG_m )}$$ coming from the isomorphism $\pi\circ i\iso\Id_{S/\BG_m}\,$. 
\end{prop}

Note that a simple particular case of \propref{p:adjointness} was proved in Sect.~\ref{sss:baby_Springer}.

\begin{rem}
\propref{p:adjointness} clearly implies  \propref{p:2contraction principle}. 
\end{rem}

\propref{p:adjointness} is well known (at least, in the setting of constructible sheaves instead of D-modules). 
It goes back to the works by Verdier \cite[Lemma 6.1]{V} and Springer
\cite[Proposition~1]{Sp}; see also \cite[Lemma A.7]{KL}  and \cite[Lemma 6]{Br}.

\medskip

The reader can easily prove \propref{p:adjointness} by slightly modifying the argument from
Sect.~\ref{ss:elem contr} (which is based on blowing up and properness).

\medskip

On the other hand, in Appendix~\ref{s:stacky_contraction} we give a complete proof of a ``stacky" generalization of 
\propref{p:adjointness} (see Theorem~\ref{t:adjunctions} and \corref{c:adjunctions}). The approach from  Appendix~\ref{s:stacky_contraction} is close to \cite{Sp} (there are no blow-ups, no properness arguments, and we work with the monoid $\BA^1$ rather than with the scheme or
stack on which $\BA^1$ acts).

\ssec{A general lemma on contractiveness}
The reader may prefer to skip this subsection on the first pass. Its main result (\lemref{l:stacky_contraction}) 
will not be used until \propref{p:Z_is_contractive}(b).

\sssec{}  \label{sss:loc str contr}

The notion of contractive substack was defined in \secref{sss:contractive}. 
This notion is clearly local in the following sense:

\medskip

Let $f:\CY'\to\CY$ be a smooth surjective morphism of algebraic stacks and $\CZ\subset\CY$
a locally closed substack. If  $f^{-1}(\CZ )$ is contractive in $\CY'$ then $\CZ$ is
contractive in $\CY$.

\sssec{}

As before, we consider $\BA^1$ as a monoid with respect to multiplication. It contains $\BG_m$
as a subgroup.

\medskip

We have: 

\begin{lem}     \label{l:stacky_contraction}
Let $\pi :\CW\to\CS$ be an affine schematic morphism of algebraic stacks. Suppose that the
monoid $\BA^1$ 
acts on $\CW$ by $\CS$-endomorphisms
(i.e., over $\CS$, with the action of $\BA^1$ on $\CS$ being trivial). Assume that

\smallskip

\noindent{\em(i)} The $\CS$-endomorphism of $\CW$ corresponding to $0\in\BA^1$ equals $i\circ\pi$ for some section $i:\CS\to\CW$;

\smallskip

\noindent{\em(ii)} The action of $\BG_m$ on $\CW$ viewed as a stack over $\on{pt}$ (rather than over $\CS$) is isomorphic to
the trivial action.

Then the substack $\CS\overset{i}\mono \CW$ is contractive. 
\end{lem}

\begin{rem} Condition (i) implies that the action of $\BG_m$ on $\CW$ viewed as a stack over
$\CS$ is nontrivial unless $\pi :\CW\to\CS$ is an isomorphism. This does not contradict (ii):
we are dealing with stacks, and the functor from the groupoid of $\CS$-endomorphisms of $\CW$
to that of $k$-endomorphisms of $\CW$
is not fully faithful.
\end{rem} 

Before proving Lemma~\ref{l:stacky_contraction}, let us consider two examples.

\begin{example}      \label{ex:taut1}
Let $p :W\to S$ be as in Sect.~\ref{sss:set up for elem contr}. Set $$\CW:=W/\BG_m,\quad 
\CS:=S/\BG_m=S\times (\on{pt}/\BG_m).$$
Then the conditions of  Lemma~\ref{l:stacky_contraction} hold for the morphism $\pi:\CW\to\CS$. 
The conclusion of Lemma~\ref{l:stacky_contraction} holds tautologically.
\end{example}

\begin{example}    \label{ex:taut2}
Let $\pi :\CW\to\CS$ be an affine schematic morphism of algebraic stacks.
Suppose that an action of $\BA^1$ on $\CW$ by $\CS$-endomorphisms satisfies condition (i) of
Lemma~\ref{l:stacky_contraction}. Set $\CW':=\CW/\BG_m$ and 
$\CS':=\CS/\BG_m=\CS\times (\on{pt}/\BG_m)$. Then the morphism
$\pi':\CW'\to\CS'$ satisfies both conditions of Lemma~\ref{l:stacky_contraction}. 
The conclusion of Lemma~\ref{l:stacky_contraction} is clear because after
a smooth surjective base change $S\to\CS$ we get the situation of Example~\ref{ex:taut1}.
\end{example}

\sssec{}

To prove Lemma~\ref{l:stacky_contraction}, we need the following assertion:

\begin{lem}    \label{l:gerbes}
Let $\varphi :\CY\to\CY'$ and $\psi :\CY'\to\CY$ be 
morphisms between algebraic stacks such that
$\psi\circ\varphi\simeq\Id_{\CY}$.
Suppose that $\varphi$ is smooth and surjective. Then:

\smallskip

\noindent{\em(a)} The maps
\[
\{\mbox{locally closed substacks of }\CY'\}\to \{\mbox{locally closed substacks of }\CY\},
\quad \CZ'\mapsto\varphi^{-1}(\CZ' ),
\]
\[
\{\mbox{locally closed substacks of }\CY\}\to \{\mbox{locally closed substacks of }\CY'\},
\quad \CZ\mapsto\psi^{-1}(\CZ )
\]
are mutually inverse bijections;

\smallskip

\noindent{\em(b)}
A locally closed substack $\CZ\subset\CY$ is contractive if and only if the corresponding substack $\CZ'\subset\CY'$ is.
\end{lem}

\begin{rem}   \label{r:smooth_surj}
Since $\varphi$ and $\psi\circ\varphi$ are smooth and surjective $\psi$ has the same 
properties.\footnote{By the way, this implies that $\CY'$ is the classifying space of a group over $\CY$.}
\end{rem}

\begin{proof}
The maps from statement (a) are clearly injective. Since $\psi\circ\varphi\simeq\Id_{\CY}$ one has
$\varphi^{-1}(\psi^{-1}(\CZ ))=\CZ$. Statement (a) follows. To prove (b), use statement (a),
Remark~\ref{r:smooth_surj}, and the locality of strong contractiveness, see \secref{sss:loc str contr}.
\end{proof}

\begin{proof}[Proof of Lemma~\ref{l:stacky_contraction}]
Define $\pi':\CW'\to\CS'$ as in Example~\ref{ex:taut2}, then the corresponding embedding
$i':\CS'\mono\CW'$ is contractive. We have a Cartesian square
\[
\CD
\CS  @>i>>  \CW  \\
@VVV    @V\varphi VV  \\
\CS'  @>i'>>  \CW'
\endCD
\]
So by Lemma~\ref{l:gerbes}, it remains to show that the morphism $\varphi :\CW\to\CW'$
admits a %(smooth surjective) 
left inverse $\psi :\CW'\to\CW$. But $\CW'$ is the quotient of $\CW$
by a \emph{trivial} action of $\BG_m\,$. Choosing a trivialization of this action we can identify
$\CW'$ with $\CW\times (\on{pt}/\BG_m)$ and take $\psi$ to be the projection 
$$\CW\times (\on{pt}/\BG_m)\to\CW.$$
\end{proof}

%\ssec{The case of $SL_2$}  \label{ss:SL 2}
\section{The case of $SL_2$}  \label{ss:SL 2}

In this section we will give a proof of \thmref{t:main truncatable} in the case $G=SL_2$, which will be the prototype of the argument in general.

\ssec{The substack $\Bun_G^{(\leq n)}$ }

\sssec{}

For an integer $n\ge 0$, let $\Bun_G^{(\leq n)}\subset \Bun_G$ be the open substack 
consisting of vector bundles that {\it do not admit} line sub-bundles of degree $> n$.
It is easy to see that the substacks $\Bun_G^{(\leq n)}$ are quasi-compact and that their union is all of
$\Bun_G$. 

\medskip

Let $g$ be the genus of $X$.  We will show that for $n\ge \on{max}(g-1,0)$, the open substack
$\Bun_G^{(\leq n)}$ is co-truncative. 

\sssec{}

Let $\Bun_G^{(n)}$ be the locally closed substack
$$\Bun_G^{(n)}:=\Bun_G^{(\leq n)}-\Bun_G^{(\leq n-1)},$$
endowed, say, with the reduced structure. 

\medskip

By \propref{p:trunc and strat}, it suffices to show that if $n> \on{max}(g-1,0)$ then $\Bun_G^{(n)}$ 
is a truncative substack of $\Bun_G^{(\leq n)}$. We will do this by combining 
Propositions \ref{p:truncativeness after smooth} and \ref{p:2contraction principle}.

\sssec{} \label{sss:not}

Note, however, that if $n$ is small relative to the genus of $X$, then the stratum $\Bun_G^{(n)}$ is 
\emph{not} truncative.  Indeed, one can choose $X$ and $n$ so that $\Bun_G^{(n)}$ has a non-empty
intersection with an open substack $\Bun_G$ that is actually a \emph{scheme}; then apply
\secref{sss:non-truncative sch}. 

\ssec{Reducing to a contracting situation}

\sssec{}

For an integer $n$, let $\Bun_B^n$ be the stack classifying short exact sequences
\begin{equation} \label{e:ext}
0\to \CL^{-1}\to \CM\to \CL\to 0,
\end{equation}
where $\CM\in \Bun_{SL_2}$, and $\CL$ is a line bundle of degree $-n$.
Let $\sfp^n: \Bun_B^n\to \Bun_G$ denote the natural projection. If $n>0$ then the image
of $\sfp^n$ equals $\Bun_G^{(n)}$.

\medskip

\begin{lem}   \label{l:facts_used}
Suppose that   $n> \on{max}(g-1,0)$. Then the morphism $\sfp^{-n}: \Bun_B^{-n}\to \Bun_G$
is smooth and its image contains $\Bun_G^{(n)}$.
\end{lem}

\begin{proof}
A point $x\in\Bun_B^{-n}$ corresponds to an exact sequence \eqref{e:ext} with $\deg \CL=n$.
The cokernel of the differential of $\sfp^{-n}$ at $x$ equals $H^1(X,\CL^{\otimes 2} )$, which is 
zero because $\deg \CL^{\otimes 2}=2n>2g-2$.
So $\sfp^{-n}$ is smooth.

A point $y\in \Bun_G^{(n)}$ corresponds to an $SL_2$-bundle $\CM$ that can be represented
as an extension \eqref{e:ext} with $\deg\CL=-n$. Such an extension splits because
$2n>2g-2$. So $\CM$ is also an extension of $\CL^{-1}$ by $\CL$. Hence, $y$ is in the
image of $\sfp^{-n}$.
\end{proof}

\sssec{}

By Lemma~\ref{l:facts_used} and \propref{p:truncativeness after smooth}, 
it suffices to show that if $n>\on{max}(g-1,0)$ then the substack

\begin{equation}   \label{e:preimage}
\Bun_G^{(n)}\underset{\Bun_G}\times \Bun_B^{-n}\subset \Bun_B^{-n}
\end{equation}
is truncative.

\ssec{Applying the contraction principle}

\sssec{}

Let $\Bun^{n}_{GL(1)}$ denote the stack of line bundles on $X$ of 
degree $n$. Note that we have a canonical isomorphism
$$\Bun^{n}_{GL(1)}\simeq \Bun_G^{(n)}\underset{\Bun_G}\times \Bun_B^{-n}$$
that sends a line bundle $\CL\in \Bun^{n}_{GL(1)}$ to 
$$0\to \CL^{-1}\to \CL^{-1}\oplus \CL\to \CL\to 0.$$

\sssec{}

Let $\Pic^{n}$ denote the coarse moduli scheme corresponding to $\Bun^{n}_{GL(1)}$. We 
have a vector bundle $\CV$ on $\Bun^{n}_{GL(1)}$ whose fiber over
$\CL\in\Bun^{n}_{GL(1)}$ equals $\Ext (\CL,\CL^{-1} )$. 

\medskip

Choose a section
$s:\Pic^{n}\to\Bun^{n}_{GL(1)}$ of the morphism $\Bun^{n}\to\Pic^{n}$
(e.g., choose $x_0\in X$ and identify $\Pic^n$ with the stack of line bundles of degree $n$ 
trivialized over $x_0$). 

\medskip

Set $\CV'=s^*(\CV)$. Let ${\bf 0}\subset \CV'$ denote the zero section.
Then $\Bun_B^{-n}$ identifies with the quotient stack
$\CV'/\BG_m$ and the substack 
$$\Bun_G^{(n)}\underset{\Bun_G}\times \Bun_B^{-n}\simeq \Bun^{n}_{GL(1)}\hookrightarrow \Bun_B^{-n}$$
identifies with ${\bf 0}/\BG_m$.
Hence, the substack \eqref{e:preimage} is truncative by Proposition~\ref{p:2contraction principle}. \qed

\section{Recollections from reduction theory}   \label{s:reduction theory}

The goal of this section is to prepare for the proof of \thmref{t:main truncatable} by recalling the 
Harder-Narasimhan-Shatz stratification of $\Bun_G$.
% into quasi-compact locally closed substacks according to the degree of instability. 

\medskip

With future applications in mind, when defining these open substacks, we will remove the assumption
that our ground field is of characteristic $0$, unless we explicitly specify otherwise.
Thus, we let $G$ be a connected reductive group over any algebraically closed field $k$.

\ssec{Notation related to $G$}    \label{ss:Gnotation}

\sssec{}
To simplify the discussion, we will work with a fixed choice of a Borel subgroup $B\subset G$. 

\medskip

Conjugacy classes of parabolics are then in bijection with the set of parabolics that contain
$B$, called \emph{the standard parabolics.} From now on, by a parabolic we will
mean a standard parabolic, unless explicitly stated otherwise. 

\medskip

For a parabolic $P$
we will denote by $U(P)$ its unipotent radical. 

\medskip

We denote by $\Gamma_G$ the set of vertices of the Dynkin diagram of $G$. 
Parabolics in $G$ are in bijection with subsets of $\Gamma_G$. For a parabolic $P$
with Levi quotient $M$ we let $\Gamma_M\subset \Gamma_G$ denote the corresponding
subset; it identifies with the set of vertices of the Dynkin diagram of $M$. 

\sssec{}  \label{sss:G}

Let $\Lambda_G$ denote the coweight lattice of $G$
and $\Lambda^\BQ_G:=\BQ\underset{\BZ}\otimes\Lambda_G\,$. Let
$\Lambda_G^+\subset \Lambda_G$ denote the monoid of dominant coweights and
$\Lambda_G^{pos}\subset \Lambda_G$ the monoid generated by positive simple coroots. 
Let $\Lambda_G^{+,\BQ},\Lambda_G^{pos,\BQ}\subset \Lambda_G^\BQ$ be the corresponding rational cones.

\medskip

Let $\check\alpha_i$, $i\in\Gamma_G$, be the simple roots; we have:
\[
\Lambda_G^{+,\BQ}=\{ \lambda\in \Lambda^\BQ_G\: | \:\langle\lambda,\check\alpha_i\rangle\ge 0
\,\mbox{ for }\, i\in\Gamma_G \}.
\]

\sssec{}    \label{sss:P}

Let $P$ be a parabolic of $G$ and $M$ its Levi quotient.  
Let $Z_0(M)$ be the neutral connected component of the center of $M$, then $\Lambda_{Z_0(M)}\subset\Lambda_G\,$. Set
$\Lambda^\BQ_{G,P}:=\Lambda^\BQ_{Z_0(M)}\subset\Lambda^\BQ_G\,$. Explicitly, 
\[
\Lambda^\BQ_{G,P}=\{ \lambda\in \Lambda^\BQ_G\: | \:\langle\lambda,\check\alpha_i\rangle=0
\,\mbox{ for }\, i\in\Gamma_M \}.
\]

Note that
$$\Lambda^\BQ_{G,G}=\Lambda^\BQ_{Z_0(G)} \text{ and } \Lambda^\BQ_{G,B}=\Lambda^\BQ_G.$$

\sssec{}

Set $\Lambda^{+,\BQ}_{G,P}:=\Lambda_G^{+,\BQ}\cap\Lambda^\BQ_{G,P}$ and
\begin{equation} \label{e:dom_reg}
\Lambda^{++,\BQ}_{G,P}:=\{ \lambda\in \Lambda^\BQ_G\: | \:\langle\lambda,\check\alpha_i\rangle=0
\,\mbox{ for }\, i\in\Gamma_M \,\mbox{ and }\, \langle\lambda,\check\alpha_i\rangle>0 \,\mbox{ for }\, i\notin\Gamma_M\}.
\end{equation}
In other words,
$\Lambda^{++,\BQ}_{G,P}$ is the set of those elements of $\Lambda^{+,\BQ}_{G,P}$ that are regular
(i.e., lie off the walls of $\Lambda^{+,\BQ}_{G,P}$). Clearly

\begin{equation} \label{e:parameterization}
\Lambda_G^{+,\BQ}=\underset{P}\bigsqcup\,  \Lambda^{++,\BQ}_{G,P},
\end{equation}
where the union is taken over the conjugacy classes of parabolics.

\sssec{} \label{sss:theprojector}

Note also that the inclusion $\Lambda^\BQ_{G,P}\hookrightarrow \Lambda^\BQ_G$ canonically splits as a direct summand:
the corresponding projector $\on{pr}_P :\Lambda^\BQ_G\to\Lambda^\BQ_{G,P}$ is defined so that
$$\on{ker}\,(\on{pr}_P)=\bigoplus\limits_{i\in\Gamma_M}\BQ\cdot \alpha_i.$$

\medskip

We can also view the map $\Lambda^\BQ_G\to \Lambda^\BQ_{G,P}$  as follows: it comes from the map
$$\Lambda_G\simeq \Lambda_M\to \Lambda_{M/[M,M]}$$ 
and the isomorphism
$$\Lambda^\BQ_{Z_0(M)}\iso \Lambda^\BQ_{M/[M,M]}$$ 
induced by the isogeny $Z_0(M)\to M/[M,M]$.

\sssec{}

We introduce the partial order on $\Lambda_G^\BQ$ by
$$\lambda_1\leqG\lambda_2\, \Leftrightarrow \, \lambda_2-\lambda_1\in\Lambda_G^{pos,\BQ}.$$

The following useful observation is due to S.~Schieder:

\begin{lem}  \label{l:proj preserves order}
For a parabolic $P$, the projection $\on{pr}_P$ is order-preserving.
\end{lem}

For a proof, see \cite[Proposition 3.1.2(a)]{Sch}.

\ssec{The degree of a bundle}
Fix a connected smooth complete curve $X$. For any algebraic group $H$ let $\Bun_H$ denote
the stack of $H$-bundles on $X$. 

\sssec{} One has a canonical isomorphism $\deg:\pi_0(\Bun_{\BG_m})\iso\BZ$. Accordingly, for
any torus $T$ one has a canonical isomorphism $\deg_T:\pi_0(\Bun_T)\iso\Lambda_T$.

\sssec{} \label{sss:deg1}
Let $\wt{G}$ be any connected affine algebraic group and let $\wt{G}_{\on{tor}}$ be its maximal quotient torus.
The composition $$\pi_0(\Bun_{\wt{G}})\to \pi_0(\Bun_{\wt{G}_{\on{tor}}})\buildrel{\deg_{\wt{G}_{\on{tor}}}}\over{\longrightarrow}\Lambda_{\wt{G}_{\on{tor}}}$$ 
will be denoted
by $\deg_{\wt{G}}$. 

\medskip

If $\wt{G}=G$ is reductive then $G_{\on{tor}}=G/[G,G]$, and the map $Z_0(G)\to G_{\on{tor}}$ is an isogeny, so
$\Lambda^{\BQ}_{G_{\on{tor}}}\simeq\Lambda^{\BQ}_{Z_0(G)}$. Therefore
one has a locally constant map $\deg_G:\Bun_G\to\Lambda^{\BQ}_{Z_0(G)}\,$. Its fibers are not necessarily 
connected but have finitely many connected components; this follows from Remark \ref{sss:components} below.

\sssec{} \label{sss:deg2}
Let now $P$ be a parabolic subgroup of a reductive group $G$, and let $M$ be the Levi quotient of $P$.

\medskip

Then by Sects. \ref{sss:P} and \ref{sss:deg1}, one has the locally constant maps $\deg_M:\Bun_M\to \Lambda^\BQ_{G,P}$ and therefore
$\deg_P:\Bun_P\to \Lambda^\BQ_{G,P}\,$. 

\medskip

The preimage of $\lambda\in\Lambda^\BQ_{G,P}$ in
$\Bun_M$ (resp. $\Bun_P)$ is denoted by $\Bun_M^\lambda$ (resp. $\Bun_P^\lambda$). 

\medskip

It is easy to see that
$\Bun_M^\lambda$ and $\Bun_P^\lambda$ are empty unless $\lambda$ belongs to a certain
finitely generated subgroup $A_{G,P}\subset\Lambda^\BQ_{G,P}$ such that 
$A_{G,P}\otimes\BQ=\Lambda^\BQ_{G,P}$; namely, $A_{G,P}=\on{pr}_P(\Lambda_G)$,
where $\on{pr}_P :\Lambda^\BQ_G\to\Lambda^\BQ_{G,P}$ is as in Sect.~\ref{sss:theprojector}.

\sssec{Remark} \label{sss:components}
Let $\wt{G}$ be any connected affine algebraic group and $\wt{G}_{\red}$ its maximal reductive quotient.
Define $\pi_1(\wt{G})$ to be the quotient of $\Lambda_{\wt{G}_{\red}}$ by the subgroup generated by coroots.
It is well known that there is a unique bijection $\pi_0(\Bun_{\wt{G}})\iso\pi_1(\wt{G})$ such that the diagram
\[
\xymatrix{
\pi_0(\Bun_{\wt{B}})\ar[d]_{}\ar[r]_{}&\Lambda_{\wt{T}}\simeq \Lambda_{\wt{G}_{\red}}\ar[d]_{}\\
 \pi_0(\Bun_{\wt{G}})\ar[r]_{}&\pi_1(\wt{G})
    }
\]
commutes. Here $\wt{B}$ is a Borel subgroup of $\wt{G}$ and $\wt{T}$ is the maximal quotient torus of $\wt{B}$.

\ssec{Semistability}

\sssec{}

Let $G_{ad}$ denote the quotient of $G$ by its center and
%For the discussion of semi-stability, we will often have to use 
$$\Upsilon_G:\Lambda^\BQ_G\to \Lambda^\BQ_{G_{ad}},$$
the projection.
%where $G_{ad}$ is the quotient of $G$ by its center. 

\medskip

%\sssec{}        
Let $\sfp_P:\Bun_P\to\Bun_G$ be the natural morphism.
Recall that a $G$-bundle $\CP_G\in \Bun_G$ is called \emph{semi-stable} if for every parabolic $P$ such that
$\CP_G=\sfp_P(\CP_P)$ with $\CP_P\in \Bun_P^{\mu}$ we have 
$$\Upsilon_G(\mu)\leqGad 0.$$

In fact, semi-stability can be tested just using reductions to the Borel:

\begin{lem}  \label{l:Bore enough}
A $G$-bundle $\CP_G$ is semi-stable if and only if 
for every reduction $\CP_B$ of $\CP_G$ to the Borel $B$ with $\CP_B\in \Bun_B^\mu$, we have
$\Upsilon_G(\mu)\leqGad 0$.
\end{lem}

\begin{proof}
This follows from \lemref{l:proj preserves order} and the fact that 
every $M$-bundle admits a reduction to the Borel of $M$.
\end{proof}

%\sssec{}     

It is known that semi-stable bundles form an open substack 
$\Bun^{ss}_G\subset\Bun_G\,$, whose intersection with each connected component of
$\Bun_G$ is quasi-compact. 

\sssec{}     \label{sss:leqtheta}

More generally, for $\theta\in \Lambda^{+,\BQ}_G$ and a $G$-bundle $\CP_G$, we say that
$\CP_G$ has \emph{Harder-Narasimhan coweight} $\leqG \theta$ if for every parabolic $P$  such that
$\CP_G=\sfp_P(\CP_P)$ with $\CP_P\in \Bun_P^{\mu}$ we have 
$$\mu\leqG \theta.$$

\medskip

As in \lemref{l:Bore enough}, it suffices to check this condition for $P=B$. 

\sssec{}   \label{sss:opens}

One shows that $G$-bundles having Harder-Narasimhan coweight 
$\leqG \theta$ form an open substack of $\Bun_G\,$. The argument repeats the proof of the fact that
$\Bun_G^{ss}$ is open, given in \cite[Proposition 6.1.6]{Sch}.

\medskip

We denote the above open substack by $\Bun_G^{(\leq \theta)}$ and sometimes by 
$\Bun_G^{(\leqG \theta)}$. It lies in the (finite) union of connected components 
of $\Bun_G$ corresponding to the image of $\theta$ under 
$$\Lambda^\BQ_G\to \Lambda^\BQ_{G,G}\simeq \Lambda^\BQ_{Z_0(G)}.$$

\medskip

Furthermore,
$$\theta_1\leqG \theta_2\, \Rightarrow  \Bun_G^{(\le \theta_1)}\subset \Bun_G^{(\le \theta_2)},$$
and 
$$\underset{\theta\in \Lambda^{+,\BQ}_{G}}\bigcup\, \Bun_G^{(\le \theta)}=\Bun_G.$$

Finally, we have:

\begin{prop} \label{p:theta qc}
The open substack $\Bun_G^{(\leq \theta)}$ is quasi-compact.
\end{prop}

We will give two proofs:

\begin{proof}[Proof 1]
With no loss of generality, we can assume that $G$ is of adjoint type.
We will use the relative compactification $\fpb_B:\BunBb\to \Bun_G$ of the map 
$\sfp_B:\Bun_B\to \Bun_G$, see \secref{sss:proof of Shatz v}.

\medskip

For each connected component $'\!\Bun_G\subset \Bun_G$ choose a coweight $\lambda\in -\Lambda_G^+$ such that 
the map $\sfp_B:\Bun_B^\lambda\to \Bun_G$ lands in $'\!\Bun_G$ and is smooth
(for smoothness, it is enough to take $\lambda$ so that $\langle \lambda,\check\alpha_i\rangle<-(2g-2)$
for each simple root $\check\alpha_i$). Then the map $\fpb_B:\BunBb^\lambda\to {}'\!\Bun_G$ is surjective. 

\medskip 

It is a basic property of $\BunBb$ (see \cite[Sect. 6.1.4]{Sch}) that 
$$\fpb_B(\BunBb^\lambda)=\underset{\mu\in \Lambda_G^{pos}}\bigcup\, \sfp_B(\Bun_B^{\lambda+\mu}).$$
Therefore
$$'\!\Bun_G\cap \Bun_G^{(\leq \theta)}=
\underset{\mu\in \Lambda_G^{pos},\, \lambda+\mu\leqG \theta}\bigcup\, \sfp_B(\Bun_B^{\lambda+\mu}).$$

However, the set 
$$\{\mu\in \Lambda_G^{pos}\,|\, \lambda+\mu\leqG \theta\}$$
is finite. Hence, $'\!\Bun_G\cap \Bun_G^{(\leq \theta)}$ is contained in the image of finitely many quasi-compact stacks
$\Bun_B^{\lambda+\mu}$, and hence is itself quasi-compact.

\end{proof}

The second proof will be given after \corref{c:open as a union}. 

\sssec{}

By definition, for $\lambda\in \Lambda^\BQ_{G,G}=\Lambda^\BQ_{Z_0(G)}$
$$\Bun_G^{ss}\cap \Bun_G^\lambda=\Bun_G^{(\le \lambda)}$$
and 
$$\Bun_G^{ss}=\underset{\lambda\in \Lambda^\BQ_{G,G}}\bigcup\, \Bun_G^{(\leq \lambda)}.$$

\sssec{}    \label{sss:lambdass}

For each parabolic $P\subset G$ with Levi quotient $M$
we have the corresponding open substack $\Bun^{ss}_M\subset\Bun_M\,$; let 
$\Bun_P^{ss}$ denote the pre-image of $\Bun_M^{ss}$ in $\Bun_P$. 

\medskip

For $\lambda\in \Lambda_{G,P}^\BQ$ we let 
$$\Bun_M^{\lambda,ss}:=\Bun_M^{ss}\cap \Bun_M^\lambda=\Bun_M^{(\leqM \lambda)},\,\, 
\Bun_P^{\lambda,ss}:=\Bun_P^{ss}\cap \Bun_P^\lambda.$$

\ssec{The Harder-Narasimhan-Shatz stratification of $\Bun_G$}
This stratification was defined in \cite{HN,Sh,Sh2} in the case $G=GL(n)$. For any reductive $G$
it was defined in \cite{R1,R2,R3} and \cite{Beh,Beh1}.

\sssec{}

We give the following definition:

\begin{defn}  \label{d:almost-iso}
A schematic morphism of algebraic stacks $f:\CX_1\to\CX_2$ is an \emph{almost-isomorphism} 
if $f$ is finite and each geometric fiber of $f$ has a single point. 
\end{defn}

\medskip

The next theorem is a basic result of reduction theory. 

\begin{thm} \label{t:reduction theory}

\hfill

\smallskip

\noindent{\em(1)}
Let $\lambda\in \Lambda_G^{+,\BQ}$ and let $P\subset G$ be the unique parabolic
such that $\lambda$ belongs to the set $\Lambda^{++,\BQ}_{G,P}$ defined by \eqref{e:dom_reg}. 
Then $\sfp_P:\Bun_P\to\Bun_G$ induces an almost-isomorphism between $\Bun_P^{\lambda,ss}$
and a quasi-compact locally closed reduced substack $\Bun_G^{(\lambda)}\subset\Bun_G$. 

\smallskip

\noindent{\em(1$'$)} If $k$ has characteristic $0$ then the morphism $\Bun_P^{\lambda,ss}\to \Bun_G^{(\lambda)}$
is an isomorphism.

\smallskip

\noindent{\em(2)} The substacks $\Bun_G^{(\lambda )}$, $\lambda\in \Lambda_G^{+,\BQ}$,
are pairwise non-intersecting, and %their union equals $\Bun_G$ (i.e., 
every geometric point of $\Bun_G$ belongs to exactly one $\Bun_G^{(\lambda )}$.

\smallskip

\noindent{\em(3)}
Let $P'\subset G$ be a parabolic and let $\lambda'$ be any (not necessarily dominant) element of
$\Lambda^{\BQ}_{G,P'}\,$. If $\sfp_{P'}(\Bun_{P'}^{\lambda'})\cap\Bun_G^{(\lambda )}\ne\emptyset$ then
$\lambda'\leqG\lambda$.
\end{thm}

Statements (1), (1$'$), (2), and a slightly weaker version of (3) are due
to K.~Behrend \cite{Beh,Beh1}. A complete proof of the theorem was given
by S.~Schieder, see \cite[Theorem 2.3.3]{Sch}. In \secref{ss:proof of
Shatz} we give a sketch of the proof from \cite{Sch}.
%We will sketch a proof of this theorem in \secref{ss:proof of Shatz}. The full proof is given in \cite[Theorem 2.3.3]{Sch}.

\sssec{}

We apply \thmref{t:reduction theory} to obtain the following more explicit description
of the open substacks $\Bun_G^{(\leq \theta)}$:

\begin{cor} \label{c:open as a union}  For $\theta\in \Lambda^{+,\BQ}_{G}$ we have:
\begin{equation} \label{e:open as a union}
\Bun_G^{(\leq \theta)}=\underset{\lambda,\,\lambda\leqG \theta}\bigcup\, 
\Bun_G^{(\lambda)},
\end{equation}
and the set
\begin{equation} \label{e:lambda's that occur}
\{\lambda\in \Lambda^{+,\BQ}_{G} \,|\,\lambda\leqG \theta \text{ and } \Bun_G^{(\lambda)}\neq \emptyset\}
\end{equation} 
is finite.
\end{cor}

\begin{proof}

The fact that
$$\Bun_G^{(\lambda)}\cap \Bun_G^{(\leq \theta)}\neq \emptyset\, \Rightarrow \lambda\leqG\theta$$
follows from the definition of $\Bun_G^{(\leq \theta)}$. 

\medskip

The inclusion
$$\Bun_G^{(\lambda)}\subset \Bun_G^{(\leq \theta)}$$
for $\lambda\leqG\theta$ follows from \thmref{t:reduction theory}(3). 

\medskip

This proves \eqref{e:open as a union} in view of \thmref{t:reduction theory}(2). 
The finiteness of the set \eqref{e:lambda's that occur} follows from the fact that
$$\Bun_G^{(\lambda)}\neq \emptyset \, \Rightarrow \lambda\in \underset{P}\bigcup\,\, A_{G,P}\; ,$$
see the end of \secref{sss:deg2}.

\end{proof}

As a corollary, we obtain a 2nd proof of \propref{p:theta qc}:

\begin{proof}[Proof 2 (of \propref{p:theta qc})]
Follows from \corref{c:open as a union} and the fact that each $\Bun_G^{(\lambda)}$ is quasi-compact.
\footnote{The quasi-compactness of $\Bun_G^{(\lambda)}$ relied on the fact that the open substack
$\Bun_M^{\lambda,ss}$ is quasi-compact, which in itself is a particular case of \propref{p:theta qc}.}
\end{proof}

As another corollary of \corref{c:open as a union} we obtain:

\begin{cor} \label{c:incl of opens}  %\hfill

\smallskip
%\Drin{Denis probably wanted to say that
%$\Bun_G^{(\leq \theta_1)}\subset \Bun_G^{(\leq \theta_2)} \Rightarrow
%\theta_1\leqG\theta_2$. This is false because some strata are empty for stupid reasons.}
%
%\noindent{\em(a)} The implication
%$$\theta_1\leqG\theta_2\, \Rightarrow\, \Bun_G^{(\leq \theta_1)}\subset \Bun_G^{(\leq \theta_2)}$$
%is ``if and only if". 
%\smallskip
%\noindent{\em(b)} 
We have:
\begin{equation} \label{e:loc cl via open}
\Bun_G^{(\theta)}=\Bun_G^{(\leq \theta)}-\underset{\theta',\theta\neq \theta'\leqG \theta}\bigcup\,
\Bun_G^{(\leq \theta')}.
\end{equation}
\end{cor}

\begin{rem}
We could \emph{a priori} define the locally closed substacks $\Bun_G^{(\theta)}$ by  formula
\eqref{e:loc cl via open}. However, without the interpretation of $\Bun_G^{(\theta)}$ via
\thmref{t:reduction theory}, it would not be clear that these locally closed substacks are
pairwise non-intersecting.
\end{rem}

%\begin{rem}
%By \secref{sss:deg2}, the set 
%$\{ \lambda\in \Lambda_G^{+,\BQ} \, |\, \Bun_G^{(\lambda )}\ne\emptyset \}$
%is discrete in $\Lambda_G^{\BR}:=\Lambda_G\otimes\BR$.
%\end{rem}

%such that $\HN^{-1}(\lambda )=|\Bun_G^{(\lambda)}(k)|$ for any $\lambda\in\Lambda_G^{+,\BQ}$.

\sssec{The Harder-Narasimhan map}  \label{sss:summary}
Let $|\Bun_G (k)|$ denote the set of isomorphism classes of $G$-bundles on $X$ (or equivalently, of objects of the groupoid $\Bun_G (k)$). 
We equip $|\Bun_G (k)|$ with the Zariski topology. 

By Theorem~\ref{t:reduction theory}(2), for every $\CF\in\Bun_G (k)$ there exists a unique 
$\lambda\in\Lambda_G^{+,\BQ}$ such that $\CF\in\Bun_G^{(\lambda)}(k)$. This $\lambda$ is called the
\emph{Harder-Narasimhan coweight
\footnote{By \corref{c:incl of opens}, this agrees with the usage of the words ``Harder-Narasimhan coweight" in Sect.~\ref{sss:opens}.}} 
of $\CF$ and denoted by $\HN (\CF )$. Thus 
we have a map 
\begin{equation}   \label{e:HN}
\HN: |\Bun_G (k)|\to\Lambda_G^{+,\BQ}.
\end{equation}

\begin{lem}    \label{l:summary}
The map \eqref{e:HN} has the following properties.
\begin{enumerate}
\item[(i)] It is upper-semicontinuous, i.e., for each $\lambda_0\in\Lambda_G^{+,\BQ }$ the preimage of the subset 
\begin{equation}   \label{e:left_segment}
\{\lambda\in \Lambda_G^{+,\BQ}\,|\, \lambda\le \lambda_0\}
\end{equation}
is open.% for each $\lambda_0\in\Lambda_G^{+,\BQ }$.
\item[(ii)] The image of the map \eqref{e:HN} is discrete in the real vector space 
$\Lambda_G^{\BR}:=\Lambda_G\otimes\BR$.
\item[(iii)] A subset $S\subset |\Bun_G (k)|$ is quasi-compact if and only if $\HN (S)$ is bounded in
$\Lambda_G^{\BR}$.
\end{enumerate}
\end{lem}

\begin{proof}
Follows from \corref{c:open as a union} and the fact that the substacks  $\Bun_G^{(\leq \theta)}$ are open and quasi-compact.
\end{proof}

\sssec{}  \label{sss:who_is_who}

Let us equip the set $\Lambda_G^{+,\BQ}$ with the \emph{order topology}, i.e., the one whose base is formed by subsets of the form~\eqref{e:left_segment}. Then statement (i) of Lemma~\ref{l:summary} can be reformulated as follows: the map \eqref{e:HN}  is \emph{continuous}. 

\medskip

Now it is clear 
that if a subset of $\Lambda_G^{+,\BQ}$ is locally closed 
then so is its preimage in $\Bun_G$. 
Note that for a subset of $\Lambda_G^{+,\BQ}$ it is
easy to understand whether it is open, closed, or locally closed, see Lemma~\ref{l:who_is_who} from Appendix~\ref{s:preordered}.
Thus we obtain:

\begin{cor} \label{c:order topology}
Let $S\subset \Lambda^{+,\BQ}_G$ be a subset. Consider the corresponding subset
$$\Bun_G^{(S)}:=\underset{\lambda\in S}\bigcup\, \Bun_G^{(\lambda)}\subset \Bun_G.$$ 

\smallskip

\noindent{\em(a)} If $S$ has the property that $\lambda_1\in S,\,\lambda_1\leqG \lambda \, \Rightarrow\,  \lambda\in S$, then 
$\Bun_G^{(S)}$ is closed in $\Bun_G$.

\smallskip

\noindent{\em(b)} If $S$ has the property that $\lambda_1\in S,\,\lambda\leqG \lambda_1\, \Rightarrow\,  \lambda\in S$,
then $\Bun_G^{(S)}$ is open in $\Bun_G$.

\smallskip

\noindent{\em(c)} If $S$ has the property that $\lambda_1,\lambda_2\in S,\,\lambda_1\leqG\lambda\leqG \lambda_2 \, \Rightarrow\,  \lambda\in S$,
then $\Bun_G^{(S)}$ is locally closed in $\Bun_G$.
\end{cor}

In cases (a) and (c) of the lemma we will regard $\Bun_G^{(S)}$ as a substack of $\Bun_G$ with the reduced structure.

\ssec{On the proof of Theorem~\ref{t:reduction theory}}  \label{ss:proof of Shatz}

Let us make some remarks regarding the proof of \thmref{t:reduction theory}. 
For a full proof along these lines see \cite{Sch}.
%A full proof along these lines will appear in \cite{Sch}.

\sssec{} \label{sss:proof point 1}
For a $G$-bundle $\CP_G$  let $\lambda$ be \emph{a}\footnote{The ``a" is italicized because we do not yet know that
such a maximal element is unique, although we will eventually show that it is.} maximal 
element in $\Lambda^{\BQ}_G$, with respect to the $\leqG$ order relation,
such that there exists a parabolic $P$ and $\CP_P\in \Bun_P^{\lambda}$ such that 
$\CP_G=\sfp_P(\CP_P)$. 
One shows using \lemref{l:proj preserves order} that the maximality 
assumption on $\lambda$ implies that $\lambda\in  \Lambda^{+,\BQ}_G$ and that 
$\CP_P\in \Bun_P^{\lambda,ss}$. For details see \cite[Sect. 6.2]{Sch}. 
%One uses the Langlands retraction $\fL_G$ (see the Appendix,  \secref{ss:Langlands defn})
%to show that $\lambda\in  \Lambda^{+,\BQ}_G$ and that $\CP_P\in \Bun_P^{\lambda,ss}$. 
%See \cite[????]{Sch} for more details.

\sssec{} \label{sss:proof point 2}
Using Bruhat decomposition, one shows (see \cite[Theorem 4.5.1]{Sch}) that if $P'$ is another parabolic and
$\CP_{P'}\in \Bun_{P'}^{\lambda'}$ such that $\CP_G=\sfp_{P'}(\CP_{P'})$, then 
$\lambda'\leqG\lambda$, and the equality takes place if and only if $P'\subset P$ and 
$\CP_P$ is induced from $\CP_{P'}$ via the above inclusion.

\sssec{} \label{sss:proof point 3}
We obtain that the set of \emph{maximal elements} $\lambda$ as in \secref{sss:proof point 1} contains
a single element. Moreover, the set of parabolics as in \secref{sss:proof point 1} also contains a unique
maximal element $P$; namely, one for which $\lambda\in  \Lambda^{++,\BQ}_{G,P}$. 

\sssec{}
This establishes points (2) and (3) of the theorem, modulo the fact that
$\Bun_G^{(\lambda)}$ is locally closed, and not just constructible.  

\sssec{}  \label{sss:proof of Shatz v}
Let $\lambda$ and $P$ be as in \secref{sss:proof point 3}. To prove point (1), one uses the relative 
compactification $$\fpb_P:\BunPb\to \Bun_G$$ of
the map $\Bun_P\to \Bun_G$ defined in \cite[Sect. 1.3.2]{BG} under the assumption that $[G,G]$ is simply connected and in  
\cite[Sect. 7]{Sch} for an arbitrary reductive $G$. 
Since $\fpb_P$ is proper, the images of $\overline{\Bun}_P^{\lambda}$ and
$\overline{\Bun}_P^\lambda-\Bun_P^{\lambda,ss}$ in $\Bun_G$ are both closed. Using \secref{sss:proof point 2}, one shows 
that the latter does not intersect $\Bun_G^{(\lambda)}$. This implies that $\sfp$ defines a
finite map from $\Bun_P^{\lambda,ss}$ to a locally closed substack of $\Bun_G$. It
is bijective at the level of $k$-points by \secref{sss:proof point 2}. See \cite[Sect. 6.2.2]{Sch} for details. 
\footnote{The latter part of the argument will actually be carried out in a slightly 
more general situation in the proof of \propref{p:parabolic 1-1}.}

\sssec{}
Once (1) is proved, statement (1$'$) is equivalent to the fact that the map $\Bun_P^{\lambda,ss}\to \Bun_G$ 
is a monomorphism (on $S$-points for any scheme $S$). This is proved (see 
\cite[Prop.~5.2.1]{Sch}) using the fact that in characteristic $0$, a homomorphism
of reductive groups $G_1\to G_2$ that sends $Z_0(G_1)$ to $Z_0(G_2)$ sends $\Bun_{G_1}^{ss}$ to 
$\Bun_{G_2}^{ss}$, see \cite[Prop.~5.2.1]{Sch} for details. (We will use a similar argument in the proof of 
\propref{p:key}(a) given in Sect.~\ref{s:estimates}).

\section{Complements to reduction theory: $P$-admissible sets}  \label{s:compl red}
In this section we fix a parabolic $P\subset G$. Let $M$ be the corresponding Levi.

\medskip

Our goal is to prove \propref{p:parabolic 1-1}, which allows us to produce locally closed
substacks of $\Bun_G$ from locally closed substacks of $\Bun_M$.

\ssec{Some elementary geometry}  \label{ss:comb}

%We begin with the following two lemma: \footnote{Instead of reading the proofs of the combinatorial-geometric 
%Lemmas~\ref{l:2comb1} and \ref{l:2comb2} below,
%the reader may prefer to check the statements in the rank 2 case by drawing the picture, and believe that the
%statements are true in general.}

Instead of reading the proofs of %the combinatorial-geometric 
Lemmas~\ref{l:2comb1} and \ref{l:2comb2} below,
the reader may prefer to check the statements in the rank 2 case by drawing the picture, and believe that the
statements are true in general.

\sssec{} 
Recall that according to the definitions from Sect.~\ref{sss:G}, we have 
$\Lambda_G^{\BQ}=\Lambda_M^{\BQ}$ and $\Lambda_G^{+,\BQ}\subset\Lambda_M^{+,\BQ}$.

\begin{lem}   \label{l:2comb1}
Let $\lambda ,\lambda'\in\Lambda_G^{\BQ}=\Lambda_M^{\BQ}$ and $\lambda'\leqM\lambda$. Then

\smallskip

\noindent{\em(a)} $\langle\lambda'\, ,\check\alpha_i\rangle\geq \langle\lambda\, ,\check\alpha_i\rangle$ for 
$i\not\in\Gamma_M$;

\smallskip

\noindent{\em(b)} If $\lambda\in\Lambda_G^{+,\BQ}$ and $\lambda'\in\Lambda_M^{+,\BQ}$
then $\lambda'\in\Lambda_G^{+,\BQ}$.
\end{lem}

\begin{proof}
Statement (a) follows from the inequality $\langle\alpha_j\, ,\check\alpha_i\rangle\le 0$ for $i\ne j$.

\medskip

To prove (b), we have to show that $\langle\lambda'\, ,\check\alpha_i\rangle\ge 0$ for all $i\in\Gamma_G$.
If $i\in\Gamma_M$ this follows from the assumption  that $\lambda'\in\Lambda_M^{+,\BQ}\,$.
If $i\not\in\Gamma_M$ this follows from (a) and the assumption  that $\lambda\in\Lambda_G^{+,\BQ}\,$.
\end{proof}

\sssec{}  

In Sect.~\ref{sss:P}-\ref{sss:theprojector} we defined the subspace 
$\Lambda^\BQ_{G,P}\subset\Lambda^\BQ_G$ and 
the projector $$\on{pr}_P :\Lambda^\BQ_G\to\Lambda^\BQ_{G,P}.$$

\begin{lem}        \label{l:2comb2}
If $\lambda\in\Lambda_G^{+,\BQ}$ then
\begin{equation}   \label{e:2ineq1}
\on{pr}_P (\lambda )\leqM \lambda,
\end{equation}
\begin{equation}    \label{e:2ineq2}
\langle  \on{pr}_P (\lambda )\, ,\check\alpha_i\rangle\ge \langle \lambda\, ,\check\alpha_i\rangle 
\mbox{ for } i\not\in\Gamma_M\, .
\end{equation}
\end{lem}

\begin{proof}
On the one hand, for $i\in\Gamma_M$, one has 
$\langle \lambda-\on{pr}_P (\lambda )\, ,\check\alpha_i\rangle=\langle \lambda\, ,\check\alpha_i\rangle\ge 0$.
On the other hand, $\lambda-\on{pr}_P (\lambda )$ belongs to the subspace generated by 
the coroots of $M$. Thus $\lambda-\on{pr}_P (\lambda )$ is in the dominant cone of the root system of $M$. 
The latter is contained in $\Lambda_M^{pos,\BQ}$, so we get \eqref{e:2ineq1}. 

\medskip

The inequality \eqref{e:2ineq2} follows from \eqref{e:2ineq1} by Lemma~\ref{l:2comb1}(a).
\end{proof}

\ssec{$P$-admissible subsets of $\Lambda_G^{+,\BQ}$}

\sssec{}

Let $S$ be a subset of $\Lambda_G^{+,\BQ}$, and let $P$ be a parabolic.

\begin{defn} \label{d:P-adm}
We say that $S$ is $P$-\emph{admissible} if the following three properties hold: 

\smallskip

\begin{equation}     \label{e:S1} 
\text{There exists }\mu \in\Lambda_{G,P}^{\BQ} \text{ such that }
S\subset \on{pr}_P^{-1}(\mu)\cap \Lambda_G^{+,\BQ}.
\end{equation} 

\smallskip

%\noindent{\em(ii)} 
\begin{equation}     \label{e:S2}
\mbox{ If }\lambda_1\in  S \text{ and }\lambda_2\in \Lambda_G^{+,\BQ},\, \lambda_2\leqM \lambda_1\,\mbox{ then }\, \lambda_2\in S.
\end{equation}

\smallskip

%\noindent{\em(iii)} We have: 
\begin{equation} \label{e:condition nz}
\forall\, \lambda\in S,\,\,\forall i\in  \Gamma_G-\Gamma_M \text{ we have }
\langle \lambda\, ,\check\alpha_i\rangle >0.
\end{equation}

\end{defn}

\begin{rem}
%Note that it follows  from \eqref{e:2ineq2} that if 
If $S\neq\emptyset$ is $P$-admissible and $\on{pr}_P(S)=\mu\in \Lambda_{G,P}^{\BQ}$
then $$\mu\in \Lambda_{G,P}^{++,\BQ}\subset \Lambda_G^{+,\BQ},$$
where, as before,
\[
\Lambda^{++,\BQ}_{G,P}:=\{ \lambda\in \Lambda^\BQ_G\: | \:\langle\lambda,\check\alpha_i\rangle=0
\,\mbox{ for }\, i\in\Gamma_M \,\mbox{ and }\, \langle\lambda,\check\alpha_i\rangle>0 \,\mbox{ for }\, i\notin\Gamma_M\}.
\]
This follows from \eqref{e:2ineq2} and \eqref{e:condition nz}.
\end{rem}

\sssec{Examples} \label{sss:P-adm examples}

\hfill

\smallskip

\noindent(i) The subset of $\on{pr}_P^{-1}(\mu)\cap \Lambda_G^{+,\BQ}$ consisting of elements satisfying \eqref{e:condition nz},
is $P$-admissible.

\smallskip

\noindent(ii) If $\lambda\in \Lambda_G^{+,\BQ}$ is such that 
$\langle \lambda\, ,\check\alpha_i\rangle >0$ for all $i\not\in\Gamma_M$
then  the set 
$$S=\{\lambda'\in  \Lambda_M^{+,\BQ}\,|\, \lambda'\leqM \lambda\}$$
is $P$-admissible by \lemref{l:2comb1}. 

%The old formulation below is incorrect: it doesn't take condition \eqref{e:loc cl via open} in account.
%
%For $\lambda\in \Lambda_G^{+,\BQ}$, the set 
%$$S=\{\lambda'\in  \Lambda_M^{+,\BQ},\,\, \lambda'\leqM \lambda\}$$
%is $P$-admissible. This follows from \lemref{l:2comb1}(b).

\smallskip

\noindent(iii) If $\mu\in \Lambda_{G,P}^{++,\BQ}$ then the one-element set $\{\mu\}$ is $P$-admissible and moreover, it is the smallest non-empty $P$-admissible subset of  
$\on{pr}_P^{-1}(\mu)\cap \Lambda_G^{+,\BQ}$. This follows from \eqref{e:2ineq1}.

\sssec{}

Let $S\subset \Lambda_G^{+,\BQ}$ be $P$-admissible subset. Note that we can also regard $S$  as a subset of~$\Lambda_M^{+,\BQ}$.

\begin{lem}  \label{l:adm open}
The subset $S\subset \Lambda_M^{+,\BQ}$
is open.
\end{lem}

\begin{proof}
Follows from \lemref{l:2comb1}(b).
\end{proof}

\ssec{Reduction theory and $P$-admissible subsets}  

\sssec{}   \label{sss:Bungs}

Let $S\subset \Lambda_G^{+,\BQ}$ be a $P$-admissible subset.

\medskip

By \corref{c:order topology}, the subset
\begin{equation}   \label{e:Bungs}
\Bun_G^{(S)}:=\underset{\lambda\in S}\bigcup\, \Bun_G^{(\lambda)}\subset \Bun_G
\end{equation}
is locally closed (and thus we can regard it as a substack with the reduced structure).

\medskip

By \lemref{l:adm open} and \corref{c:order topology}, the subset 
$$\Bun_M^{(S)}:=\underset{\lambda\in S}\bigcup\, \Bun_M^{(\lambda)}\subset \Bun_M$$
is open.
Set $$\Bun_P^{(S)}:=\Bun_P\underset{\Bun_M}\times \Bun_M^{(S)}.$$

\sssec{}

The next proposition is a generalization of Theorem~\ref{t:reduction theory}(1); the latter corresponds to the case 
where the $P$-admissible subset $S$ has one element, see Example~\ref{sss:P-adm examples}(iii).

\begin{prop}  \label{p:parabolic 1-1}
Let $S\subset \Lambda_G^{+,\BQ}$ be a $P$-admissible subset.
Then the restriction of $\sfp_P:\Bun_P\to \Bun_G$ to
$\Bun_P^{(S)}$ defines an almost-isomorphism
\begin{equation}    \label{p:almost_S}
\Bun_P^{(S)}\to \Bun_G^{(S)}.
\end{equation}
\end{prop}

%\begin{rem}
%If $S$ is as in Example~\ref{sss:P-adm examples}(iii) then \propref{p:parabolic 1-1} follows from
%Theorem~\ref{?}(1).
%
%%Note that the assertion of \propref{p:parabolic 1-1} in the case $S=\{\mu\}$ contains the one from 
%%\secref{sss:proof of Shatz v}.
%\end{rem}

\begin{rem}
Later we will show that if $k$ has characteristic 0
then the map \eqref{p:almost_S} is, in fact, an \emph{isomorphism} (see
\lemref{l:a for c''_i} and Remark~\ref{r:automatic} below).
\end{rem}

The rest of this section is devoted to the proof of \propref{p:parabolic 1-1}.

\sssec{}

First, let us prove that the map \eqref{p:almost_S} maps $\Bun_P^{(S)}$ to $\Bun_G^{(S)}$ and 
is bijective at the level of $k$-points. To this end, it suffices to show that if $\lambda$ is an element of 
$$S\subset \Lambda^{+,\BQ}_G\subset \Lambda^{+,\BQ}_M$$
then the map $\sfp_P$ sends
$$\Bun_P\underset{\Bun_M}\times \Bun_M^{(\lambda)}$$
bijectively to $\Bun_G^{(\lambda)}$. 

\medskip

Let $M'$ be the Levi of $G$ such that
\[
\Gamma_{M'}=\{i\in \Gamma_G\,|\, \langle \lambda,\check\alpha_i\rangle=0\}.
\]
%Let $\Gamma_{M'}$ be the subset of $\Gamma_G$, defined by
%$$i\in \Gamma_{M'} \,\,\text{   if   } \,\, \langle \lambda,\check\alpha_i\rangle=0.$$
By condition \eqref{e:condition nz}, we have $\Gamma_{M'}\subset \Gamma_M$.
Let $P'\subset P$ be the corresponding parabolic.  Set $\PP:=P'/U(P)$; this is a
parabolic in $M$ whose Levi quotient identifies with $M'$. We have
$$\lambda\in \Lambda^{++,\BQ}_{G,P'}\subset \Lambda^{++,\BQ}_{M,\PP}\, .$$

\medskip

By the definition of $\Bun_M^{(\lambda)}$, we have a bijection
$$\Bun_{\PP}\underset{\Bun_{M'}}\times \Bun_{M'}^{\lambda,ss}\to \Bun_M^{(\lambda)}.$$

\medskip

Hence, it is enough to show that the map
$$\Bun_P\underset{\Bun_M}\times \Bun_{\PP}\underset{\Bun_{M'}}\times \Bun_{M'}^{\lambda,ss}\to \Bun_G$$
defines a bijection onto $\Bun_G^{(\lambda)}$. 

\medskip

However,
$$\Bun_P\underset{\Bun_M}\times \Bun_{\PP}\underset{\Bun_{M'}}\times \Bun_{M'}^{\lambda,ss}
\simeq \Bun_{P'}\underset{\Bun_{M'}}\times \Bun_{M'}^{\lambda,ss},$$
and the required assertion follows from the definition of $\Bun_G^{(\lambda)}$.

\sssec{}

To finish the proof of the proposition, we have to show that the map \eqref{p:almost_S} is finite. 
We already know that it is bijective, so it suffices to show the map \eqref{p:almost_S} is proper.
This will be done by generalizing the argument in \secref{sss:proof of Shatz v} using the stack $\BunPb$.
%, the latter defined, e.g., in \cite[Sect. 1.3.2]{BG}. \footnote{In \cite{BG} one defines $\BunPb$ assuming
%that the derived group of $G$ is simply-connected. The generalization for an arbitrary $G$ is 
%%%%not very difficult and will be 
%carried out in \cite[Sect.7]{Sch}.}

\medskip

Let $\mu \in\Lambda_{G,P}^{\BQ}$ be  such that 
$S\subset \on{pr}_P^{-1}(\mu)\cap \Lambda_G^{+,\BQ}$.
Consider the map 
$$\bsfp:\overline{\Bun}_P^{\mu}\to \Bun_G.$$
This map is proper. So to prove properness of \eqref{p:almost_S},
it is enough to show that 
\begin{equation} \label{e:prove emptyness}
\bsfp(\overline{\Bun}_P^{\mu}-\Bun_P^{(S)})\cap \Bun_G^{(S)}=\emptyset.
\end{equation}

We have
$$\overline{\Bun}_P^{\mu}-\Bun_P^{(S)}=(\overline{\Bun}_P^{\mu}-\Bun_P^\mu)\cup (\Bun_P^\mu-\Bun_P^{(S)}),$$
so to prove \eqref{e:prove emptyness}, it suffices to show that
\begin{equation} \label{e:1st int}
\bsfp(\overline{\Bun}_P^{\mu}-\Bun_P^\mu)\cap  \Bun_G^{(\lambda )}=\emptyset 
\quad \mbox{ for all } \lambda\in S
\end{equation}
and
\begin{equation} \label{e:2nd int}
\sfp_P(\Bun_P^\mu-\Bun_P^{(S)}) \cap  \Bun_G^{(\lambda )}=\emptyset
\quad \mbox{ for all } \lambda\in S.
\end{equation}

\sssec{}   \label{sss:1st int}

To prove \eqref{e:1st int}, we use the equality
$$\bsfp(\overline{\Bun}_P^{\mu}-\Bun_P^\mu)=\underset{\mu'\in \Lambda_{G,P}^{\BQ}\, \mu'-\mu\in \Lambda^{\on{pos}}_G-0}
\bigcup\, \sfp_P(\Bun_P^{\mu'}),$$
which follows from \cite[Sect. 6.1.4]{Sch}. This equality shows that
 if \eqref{e:1st int} were false we would have
\begin{equation} \label{e:non-empty}
\sfp_P(\Bun_P^{\mu'})\cap \Bun_G^{(\lambda )}\ne\emptyset
\end{equation}
for some $\mu'\in \Lambda_{G,P}^{\BQ}$ such that
\begin{equation}  \label{e:<}
\mu\leqG \mu'\neq \mu .
\end{equation}

However, \eqref{e:non-empty} implies,
by \thmref{t:reduction theory}(3), that $\mu'\leqG  \lambda$. So by \lemref{l:proj preserves order},
$$\mu'\leqG\on{pr}_P(\lambda)=\mu,$$ which contradicts \eqref{e:<}.

\sssec{}

To prove \eqref{e:2nd int}, we have to show that if $\lambda'\in \Lambda^{+,\BQ}_M$ is such that
\begin{equation} \label{e:2nd int'}
\sfp_P(\Bun_P^\mu\underset{\Bun_M}\times \Bun_M^{(\lambda')})\cap \Bun_G^{(\lambda)}\neq \emptyset
\end{equation}
then $\lambda'\in S$. 

\medskip

If \eqref{e:2nd int'} holds then $\lambda'\leqG\lambda$ by \thmref{t:reduction theory}(3).
Since $\on{pr}_P(\lambda)=\mu=\on{pr}_P(\lambda')$ this implies that
$\lambda'\leqM\lambda$.

Since $\lambda'\in \Lambda^{+,\BQ}_M$ and $\lambda'\leqM\lambda$ we get 
$\lambda'\in \Lambda^{+,\BQ}_G$ by \lemref{l:2comb1}(b). Since $\lambda\in S$, 
$\lambda'\in \Lambda^{+,\BQ}_G$, and $\lambda'\leqM\lambda$ we get  $\lambda'\in S$
by the admissibility of $S$.
\qed

\section{Proof of the main theorem}  \label{s:delo}

%This section has the following structure. In \secref{ss:intro_to_section} we formulate the
%main result of this section (Theorem~\ref{t:2cotruncative}), which implies Theorem~\ref{t:main}
%(which is the same as \thmref{t:main truncatable}). 
%In \secref{ss:proof_modulo} we prove the theorem modulo Propositions~\ref{p:2c'_i} and
%\ref{p:2point-wise_contr}. These propositions are proved in \secref{ss:numerical} and \ref{ss:final_proof}.
%Sects. \ref{ss:P&M}-\ref{ss:A1action} contain material used in \secref{ss:final_proof}.
%Sects. \ref{ss:P&M}-\ref{ss:final_proof} are independent of 
%Sect.~\ref{ss:numerical}.  

\ssec{The main result of this section}  \label{ss:outline of proof}

We wish to prove Theorem~\ref{t:main truncatable} (=\thmref{t:main}), which says that $\Bun_G$ is truncatable.

\sssec{}

For each $\theta\in\Lambda_{G}^{+,\BQ}\,$, we have the quasi-compact open substack 
$\Bun_G^{(\leq \theta)}\subset\Bun_G$, see Sect.~\ref{sss:opens}  and formula \eqref{e:open as a union}. 
The substacks $\Bun_G^{(\leq \theta)}$ cover $\Bun_G\,$. 

\medskip

So Theorem~\ref{t:main truncatable} is a consequence of the following fact: 

\begin{thm}   \label{t:2cotruncative}
The substack $\Bun_G^{(\leq \theta)}$ is co-truncative if for every simple root $\check\alpha_i$ one has
\begin{equation}  \label{e:3what_is_deep}
\langle \theta \, ,\check\alpha_i\rangle \geq 2g-2,
\end{equation}
where $g$ is the genus of $X$.
\end{thm}

In this section we will prove \thmref{t:2cotruncative} modulo a key geometric assertion,
\propref{p:key}.

\begin{rem}
In Theorem~\ref{t:2cotruncative} we assume that the ground field $k$ has characteristic 0
(because the notion of truncativeness is defined in terms of D-modules). However, \propref{p:key}
(of which \thmref{t:2cotruncative} is an easy consequence) is valid over any $k$. 
\footnote{We have in mind future applications to the $\ell$-adic derived category on $\Bun_G$, and this category 
makes sense in any characteristic.}
\end{rem}

Below follow some remarks on the proof of Theorem~\ref{t:2cotruncative}.

\sssec{The main difficulty}    \label{sss:main difficulty}
In Sect.~\ref{ss:SL 2} we already proved Theorem~\ref{t:2cotruncative} for $G=SL_2$.
The proof in the general case is more or less similar. 

\medskip

However, one has to keep in mind the following.
If $G=SL_2$ we saw that all but finitely many Harder-Narasimhan-Shatz strata
$\Bun_G^{(\lambda)}$ are truncative. This is false already for $G=SL_2\times SL_2$.
Indeed, the stratum of the form $\Bun^{(n)}_{SL_2}\times \Bun^{(m)}_{SL_2}$,
with $n$ small relative to the genus of $X$, is \emph{not} truncative in 
$\Bun_{SL_2}\times \Bun_{SL_2}=\Bun_{SL_2\times SL_2}$ because $\Bun^{(n)}_{SL_2}$
is not truncative in $\Bun_{SL_2}$, see \secref{sss:not}. 

\medskip

For any $G$, it turns out that $\Bun_G^{(\lambda)}$ is truncative if
$\lambda$ is ``deep inside" the interior of some face of the cone $\Lambda^{+,\BQ}_G$; the problem arises if
$\lambda$ is close to the boundary of the open face of  $\Lambda^{+,\BQ}_G$ containing $\lambda$.

\sssec{Resolving the difficulty}    \label{sss:resolving}
We prove that certain \emph{unions} of the strata $\Bun_G^{(\mu)}$ are truncative 
(see \corref{c:truncative}). In particular, we show that if $\lambda\in\Lambda^{+,\BQ}_G$ and
\begin{equation}   \label{e:Slambda}
S_{\lambda}:=\{ \mu\in\Lambda^{+,\BQ}_G\: | \;\lambda -\mu\in\sum_{i\in \Gamma_{G,\lambda}}\BQ^{\geq 0}\cdot\alpha_i\},
\end{equation}
where $\Gamma_{G,\lambda}:=\{ i\in\Gamma_G\: | \;\langle \lambda \, ,\check\alpha_i\rangle \leq 2g-2\}$ then 
$$\bigcup\limits_{\mu\in S_{\lambda}}\Bun_G^{(\mu)}$$ 
is a truncative locally closed substack of $\Bun_G\,$.
To finish the proof of Theorem~\ref{t:2cotruncative}, we show that if $\theta$ satisfies \eqref{e:3what_is_deep}
then the set 
$$\{\lambda\in\Lambda^{+,\BQ}_G\;| \lambda \not\leqG \theta\}$$
can be represented as a union of subsets of the form \eqref{e:Slambda}.

\sssec{The rank 2 case is representative enough}
All the difficulties of the proof of Theorem~\ref{t:2cotruncative} appear already if $G$ is a semi-simple group of rank 2. 
On the other hand, in this case various combinatorial-geometric statements (e.g., the above statement at the end of \secref{sss:resolving}) %or the ones from \secref{ss:comb}) 
become obvious once you draw a picture.

\sssec{}
In Appendix~\ref{s:Langlands} we give a variant of the proof of Theorem~\ref{t:2cotruncative}, which has some advantages compared with the one from \secref{ss:proof_modulo}. 
The relation between the two proofs is explained in Sects.~\ref{ss:a variant} and \ref{ss:Relation}.

%In a previous version of this article \cite{old_version} we gave a proof of Theorem~\ref{t:2cotruncative} which 
%was essentially equivalent to the one given below. But the book-keeping was organized slightly differently:  
%as the main combinatorial tool, we used in \cite[Sect.~6]{old_version} a certain coarsening of the 
%Harder-Narasimhan-Shatz stratification, which was defined using the \emph{Langlands retraction} $\fL_G$ 
%(see \secref{ss:coarse strat} of Appendix~\ref{s:Langlands},  where it is reviewed).  The approach from  
%\cite{old_version} has some advantages.

\ssec{A key proposition}

\sssec{}

We will deduce \thmref{t:2cotruncative} from the following assertion:

\begin{prop} \label{p:key} There exists an assignment
$$i\in \Gamma_G \rightsquigarrow c_i\in \BQ^{\geq 0}$$ such that 
%the following holds:
%
%\smallskip
%
%For a parabolic $P$ and a $P$-admissible subset $S\subset \Lambda_G^{+,\BQ}$ such that
for any parabolic $P$ and any $P$-admissible subset $S\subset \Lambda_G^{+,\BQ}$ satisfying the condition
\begin{equation} \label{e:key}
\forall\, \lambda\in S,\,\,\forall i\in  \Gamma_G-\Gamma_M \;\, %\text{ we have }
\langle \lambda\, ,\check\alpha_i\rangle >c_i %\, ,
\end{equation}
(where as usual, $M$ is the Levi quotient of $P$)
%we have:
the following properties hold:

\smallskip

\noindent{\em(a)}
The morphism $\Bun_P^{(S)}\to \Bun_G^{(S)}$ induced by $\sfp_P:\Bun_P\to \Bun_G$ is an isomorphism;
%the restriction of $\sfp_P:\Bun_P\to \Bun_G$ to $\Bun_P^{(S)}$ defines an isomorphism
%$$\Bun_P^{(S)}\to \Bun_G^{(S)}\, ;$$

\smallskip

\noindent{\em(b)}
The locally closed substack $\Bun_G^{(S)}\subset \Bun_G$ is contractive 
in the sense of  Sect.~\ref{sss:contractive}. 

\medskip

When $\mbox{char\,} k=0$ we can take $c_i=\on{max}(0,2g-2)$. 

\end{prop}

\sssec{}

\propref{p:key} will be proved in Sects. \ref{s:estimates} and \ref{s:constr contr}. Namely,
in \secref{s:estimates} we will produce the numbers $c_i$ and 
prove property (a) for these numbers (and, in fact, for smaller ones). In \secref{s:constr contr} we will prove property (b).

\begin{rem}
It is only property (b) that will be needed for the proof of \thmref{t:2cotruncative}.
Property (a) will be used for the proof of property (b).
\end{rem}

\begin{rem}
Note that the assertion of point (a) of \propref{p:key} differs from that of \propref{p:parabolic 1-1}
only slightly: the former asserts ``isomorphism", while the latter ``almost-isomorphism".
\end{rem}

\sssec{}

We now specialize to the case of $\mbox{char\,} k=0$, in which case we have the theory 
of D-modules and of truncative substacks (see Definition \ref{d:trunc2}).

\medskip

We have:

\begin{cor}         \label{c:truncative}
Let $S\subset \Lambda_G^{+,\BQ}$ be a $P$-admissible subset such that 
\begin{equation}    \label{e:condition_g}
\forall\, \lambda\in S,\,\,\forall i\in  \Gamma_G-\Gamma_M \text{ we have }
 \langle \lambda\, ,\check\alpha_i\rangle>2g-2. 
\end{equation}
Then the locally closed substack $\Bun_G^{(S)}\subset\Bun_G$ is truncative.
\end{cor}

\begin{proof}
%We claim that we can assume that $g\geq 1$. Indeed, in the case of $g=0$
%\emph{any}  locally closed substack of $\Bun_G$ is truncative, see Sect.~\ref{sss:finitely many points}.
Since $S$ is admissible, the condition $\langle \lambda\, ,\check\alpha_i\rangle>2g-2$ from
\eqref{e:condition_g} is equivalent to the condition $\langle \lambda\, ,\check\alpha_i\rangle>\on{max}(0,2g-2)$. 

\medskip

Now apply \propref{p:key} with $c_i=\on{max}(0,2g-2)$, and the assertion follows from the fact that
contractiveness %$\Rightarrow$ 
implies truncativeness, see \corref{c:contraction principle}.
\end{proof}

\ssec{Proof of Theorem ~\ref{t:2cotruncative}}   \label{ss:proof_modulo}

\sssec{}

We have to show that if 
\begin{equation}   \label{e:theta_condition}
\forall i\in\Gamma_G  \text{ we have } \langle \theta\, ,\check\alpha_i\rangle\ge 2g-2
\end{equation}
and $\theta'\underset{G}\geq\theta$ then the substack $\Bun_G^{(\leq \theta')}-\Bun_G^{(\leq \theta)}\subset\Bun_G$ 
is truncative. 

\medskip

If $g=0$ then \emph{any} locally closed substack of 
$\Bun_G$ is truncative (see Sect.~\ref{sss:finitely many points}). So we can and will assume that $g\ge 1$.

%Equivalently, it suffices to construct for each $\lambda\in\Lambda_G^{+,\BQ}$ such that 
%$\lambda\not\leqG\theta$ a subset $S_{\lambda}\subset\Lambda_G^{+,\BQ}$

\sssec{}

By  \propref{p:trunc and strat}, it suffices to cover  $\Bun_G^{(\leq \theta')}-\Bun_G^{(\leq \theta)}$ by finitely many truncative substacks. 

\medskip

Let $\lambda\in\Lambda_G^{+,\BQ}$ be such that $\lambda\not\leqG\theta$, i.e., 
\begin{equation}   \label{e:theta -lambda}
\theta -\lambda\not\in \Lambda_G^{pos,\BQ}:=\sum\limits_{i\in\Gamma_G}\BQ^{\geq 0}\cdot\alpha_i\, .
\end{equation}
It suffices to construct for each such  $\lambda$ a subset $S_{\lambda}\subset\Lambda_G^{+,\BQ}$ containing 
$\lambda$ such that the substack $\Bun_G^{(S_{\lambda})}\subset\Bun_G$ is truncative and
\begin{equation}   \label{e:empty intersec}
\Bun_G^{(\leq \theta)}\cap\Bun_G^{(S_{\lambda})}=\emptyset .
\end{equation}

\medskip

Here is the construction. Let $P$ be the parabolic whose Levi quotient, $M$, corresponds to the following subset of $\Gamma_G\,$:
\begin{equation} \label{e:parabolic_choice}
\Gamma_M=\{i\in\Gamma_G\,|\,\langle \lambda\, ,\check\alpha_i\rangle\le 2g-2\}.
\end{equation} 
Now define 
\begin{equation} \label{e:S_lambda}
S_{\lambda}:=\{\lambda'\in\Lambda_G^{+,\BQ}\,|\,\lambda'\leqM\lambda\}.
\end{equation}

Note that by \eqref{e:parabolic_choice}, for each $i\in \Gamma_G-\Gamma_M$ we have 
$\langle \lambda\, ,\check\alpha_i\rangle > 2g-2$, which implies that 
$\langle \lambda\, ,\check\alpha_i\rangle > 0$ (because we are assuming that $g\ge 1$).
So by \lemref{l:2comb1}(a), $S_\lambda$ is $P$-admissible and satisfies the condition of \corref{c:truncative}.
Hence, the substack $\Bun_G^{(S_{\lambda})}\subset\Bun_G$ is truncative. 

\medskip

Therefore, to prove \thmref{t:2cotruncative} it remains to check \eqref{e:empty intersec}.

%\sssec{}
%\
%\Clearly, $S_\lambda$ is $P$-admissible, and by \lemref{l:2comb1}(a), it 
%\satisfies the condition of \corref{c:truncative}. Hence,
%\$$\Bun_G^{(S_{\lambda})}\subset\Bun_G$$
%\is truncative. 
%
%\medskip
%
%Therefore, to prove \thmref{t:2cotruncative} it remains to show the following:
%
%\begin{lem} \label{l:S lambda int}
%$\Bun_G^{(\leq \theta)}\cap\Bun_G^{(S_{\lambda})}=\emptyset$.
%\end{lem}

%\qed(\thmref{t:2cotruncative})

%\sssec{Proof of  \lemref{l:S lambda int}}
\sssec{Proof of equality \eqref{e:empty intersec}}

We need the following lemma.

\begin{lem}     \label{l:opyat}
Let $\nu =\sum\limits_{i\in\Gamma_G} a_i\cdot \alpha_i\,$, $a_i\in\BQ$. 
Assume that $a_i\ge 0$ for $i\not\in\Gamma_M$ and $\langle \nu\, ,\check\alpha_i\rangle\ge 0$ for 
$i\in\Gamma_M\,$. Then $a_i\ge 0$ for all $i\in\Gamma_G\,$. \end{lem}

\begin{proof}
Set $\nu_M: =\sum\limits_{i\in\Gamma_M} a_i\cdot \alpha_i\,$. We have to show that 
$\nu_M\in\Lambda_M^{pos,\BQ}$.
The assumptions of the lemma and
the inequality  $\langle \alpha_j\, ,\check\alpha_i\rangle\le 0$ for $i\ne j$ imply that
$\langle \nu_M\, ,\check\alpha_i\rangle\ge 0$ for $i\in\Gamma_M\,$. Thus $\nu_M$ belongs to the dominant cone of the root system of $M$ and therefore to $\Lambda_M^{pos,\BQ}$. 
\end{proof}

We are now ready to prove \eqref{e:empty intersec}. %\lemref{l:S lambda int}:
It suffices to prove the following
\begin{lem}
There is no $\lambda'\in\Lambda_G^{+,\BQ}$ such that $\lambda'\leqG \theta$ and 
$\lambda'\leqM \lambda$. 
\end{lem}

\begin{proof}
Suppose that such  $\lambda'$ exists.
%$\Bun_G^{(\le \theta)}\cap\Bun_G^{(S_{\lambda})}\ne\emptyset$.
%Then there exists $\lambda'\in\Lambda_G^{+,\BQ}$ such that $\lambda'\leqG \theta$ and 
%$\lambda'\leqM \lambda$. So 
Then $\theta -\lambda$ has the form $\sum\limits_{i\in\Gamma_G} c_i\cdot \alpha_i\,$, 
where $c_i\ge 0$ for $i\not\in\Gamma_M\,$. By \eqref{e:theta_condition} and \eqref{e:parabolic_choice}, 
$\langle \theta -\lambda\, ,\check\alpha_i\rangle\ge 0$
for $i\in\Gamma_M\,$. Hence, by Lemma~\ref{l:opyat}, $\theta -\lambda\in\Lambda_G^{pos,\BQ}$, 
contrary to the assumption \eqref{e:theta -lambda}. 
\end{proof}

Thus we have proved \thmref{t:2cotruncative}.
%\qed(\lemref{l:S lambda int})

%\begin{rem}
%In \secref{ss:coarse strat} of Appendix~\ref{s:Langlands} we will define another system of subsets of 
%$\Lambda_G^{+,\BQ}$ that
%that can be used instead of the sets \eqref{e:S_lambda} %$S_\lambda$'s 
%to prove \thmref{t:2cotruncative}. 
%\end{rem}

\ssec{A variant of the proof}   \label{ss:a variant}
In the above proof of \thmref{t:2cotruncative} we used substacks of the form 
$\Bun_G^{(S_{\lambda})}$, $\lambda\in \Lambda_G^{+,\BQ}$, where $S_\lambda$ is defined by
\eqref{e:parabolic_choice} -\eqref{e:S_lambda}. Instead of considering \emph{all} substacks of this form, one could consider only \emph{maximal} ones among them; one can show that they form a \emph{stratification} of 
$\Bun_G$ all of whose strata are truncative. This somewhat ``cleaner" picture is explained in 
Appendix~\ref{s:Langlands}.

\section{The estimates}  \label{s:estimates}
In this section we produce the numbers $c_i$ mentioned in \propref{p:key} and 
prove  \propref{p:key}(a) for these numbers (and, in fact, for smaller ones).

\ssec{The vanishing of $H^0$ and $H^1$}   \label{ss:vanishing}

\sssec{} \label{sss:conventions_Sect6}

For what follows, we fix a maximal torus $T\subset B$. This allows us to view the Levi quotient
$M$ of a (standard) parabolic $P$ as a subgroup $M\subset P$ (the unique splitting that contains 
$T$).

% Denis, I removed the footnote because it is somewhat sloppy, because it is trivial to the extent it is not sloppy,
% and because we don't have to apologize. (The nontrivial fact is that in \propref{p:2c'_i} the vector bundles
%  can be defined without choosing a splitting. This fact is important but not for this article.)
%\footnote{This choice is made only in order to simplify the exposition; it is easy to see that the estimates in \propref{p:2c'_i} do not depend on this choice.}

\medskip

Recall that for a parabolic $P$ we denote by $U(P)$ its unipotent radical. We will use the following notation for Lie algebras:
$\fg:=\Lie\, (G)$, $\fp:=\Lie\, (P)$, $\fn (P):=\Lie\, (U(P))$.

\medskip

For an algebraic group $H$, a principal $H$-bundle $\CF_H$ and an $H$-representation $V$, we 
denote by $V_{\CF_H}$ the associated vector bundle. 

\sssec{}  \label{sss:The_main}

The main result of this section is: 
%We will deduce \propref{p:key}(a) from the following one.

\begin{prop}   \label{p:2c'_i}
There exists a collection of  numbers 
%For every $i\in \Gamma_G$ there exist numbers 
$$c'_i,\in\BQ, c_i''\in\BQ^{\geq 0},\quad i\in \Gamma_G$$
such that for any %parabolic $P$
quadruple $$(P,M,\lambda,\CF_M),$$ where $P$ is a parabolic, $M$ the corresponding Levi, 
$\lambda\in\Lambda_G^{+,\BQ}$ and $\CF_M\in\Bun_{M}^{(\lambda)}$,  the following statements hold:
\begin{equation}    \label{e:2c'_i}
\mbox{if }\;  \forall i\in\Gamma_G-\Gamma_M \text{ we have }
\langle\lambda\, , \check\alpha_i\rangle >c'_i \mbox{ then }  H^1(X, \fn (P)_{\CF_M})=0;
\end{equation}

\begin{equation}     \label{e:2c''_i}
\mbox{if }\; \forall i\in\Gamma_G-\Gamma_M \text{ we have }
\langle\lambda\, , \check\alpha_i\rangle >c''_i \mbox{ then } 
H^0(X, (\fg/\fp)_{\CF_M})=0.
\end{equation}

\medskip

If $\mbox{char\,} k=0$ then one can take $c'_i=2g-2$, $c''_i=0$.
\end{prop}

\medskip

\propref{p:2c'_i} will be proved in Sect.~\ref{sss:proof_2c'_i}. For a discussion of the case $\mbox{char\,} k>0$, see Sect.~\ref{ss:p>0}.

\ssec{The numbers $c_i\,$: proof of \propref{p:key}(a)}   \label{ss:c_i}
In this subsection we will assume \propref{p:2c'_i}.% and deduce  \propref{p:key}(a). 

%\sssec{}  \label{sss:c_i}

\medskip

Let $c'_i$ and $c''_i$ be as in \propref{p:2c'_i}. Set
\begin{equation} \label{e:c_i max}
c_i:=\on{max}(c'_i,c''_i).
\end{equation}
Eventually we will show that  \propref{p:key} holds for the numbers $c_i$ defined by \eqref{e:c_i max}.
For  \propref{p:key}(a) this follows from the next lemma (which is slightly sharper than \propref{p:key}(a) because the numbers $c_i$ are replaced by $c''_i\le c_i$).

\begin{lem}  \label{l:a for c''_i}
Let $c_i''$ be as in  \propref{p:2c'_i}. % (so in the case $\mbox{char\,} k=0$ one can take $c_i''=0$).
Let $S\subset  \Lambda^{+,\BQ}_G$ be a $P$-admissible subset such that 
\begin{equation}    \label{e:condition_a'}
\forall\, \lambda\in S,\,\,\forall i\in  \Gamma_G-\Gamma_M\quad
\langle \lambda\, ,\check\alpha_i\rangle >c''_i  \, .
\end{equation} 
Then   the morphism $\Bun_P^{(S)}\to \Bun_G^{(S)}$ induced by $\sfp_P:\Bun_P\to \Bun_G$ is an isomorphism.
\end{lem}

\begin{rem}   \label{r:automatic}
The lemma implies that
if $\mbox{char\,} k=0$ then the morphism $\Bun_P^{(S)}\to \Bun_G^{(S)}$ is an isomorphism for \emph{any}
$P$-admissible $S\subset  \Lambda^{+,\BQ}_G$. Indeed, if $\mbox{char\,} k=0$ one can take $c_i''=0$
(see the last sentence of \propref{p:2c'_i}); on the other hand, for $c_i''=0$ the inequality \eqref{e:condition_a'} holds by the definition of $P$-admissibility, see Definition~\ref{d:P-adm}.

\end{rem}

%\sssec{}

%Recall that point (a) of \propref{p:key} says that if $S\subset  \Lambda^{+,\BQ}_G$ is a $P$-admissible subset such that 
%\begin{equation}    \label{e:condition a}
%\forall\, \lambda\in S,\,\,\forall i\in  \Gamma_G-\Gamma_M\quad
%\langle \lambda\, ,\check\alpha_i\rangle >c_i  \, ,
%\end{equation}
%then the map \eqref{p:almost_S} is an isomorphism.
%
%\medskip
%
%We will show this for a slightly sharper estimate, namely if 
%\begin{equation}    \label{e:condition_a'}
%\forall\, \lambda\in S,\,\,\forall i\in  \Gamma_G-\Gamma_M\quad
%\langle \lambda\, ,\check\alpha_i\rangle >c''_i  \, ,
%\end{equation}
%where the numbers $c''_i$ are as in Proposition~\ref{p:2c'_i}.

%\sssec{Proof}
\begin{proof}[Proof of \lemref{l:a for c''_i}]
By \propref{p:parabolic 1-1}, it suffices to show that for any $y\in \Bun_P(k)$,
the tangent space at $y$ to the fiber of 
$\sfp_P:\Bun_P\to \Bun_G$ over $\sfp_P(y)$ is zero. 

\medskip

The tangent space in question identifies with
$H^0(X,(\fg/\fp)_{\CF_P})$, where $\CF_P$ is the $P$-bundle corresponding to $y$.

\medskip

Note that the vector bundle $(\fg/\fp )_{\CF_M}$ can be identified with the associated graded
of $(\fg/\fp)_{\CF_P}$ with respect to a (canonically defined) filtration on the latter. Hence, 
\eqref{e:2c''_i} implies that $H^0(X,(\fg/\fp )_{\CF_P})=0$.
\end{proof}

\ssec{The notion of strangeness}   \label{ss:strangeness}

\sssec{}

Let $\wt{G}$ be a reductive group and $V$ a finite-dimensional $\wt{G}$-module on which $Z_0(\wt{G})$ acts by a character $\check\mu$. 

\begin{lem} \label{l:2strangeness} \hfill

\smallskip

\noindent{\em(i)} There exists a number $c\in \BQ$  such that for every 
$\CF_{\wt{G}}\in \Bun_{\wt{G}}^{ss}$ the degree of any line sub-bundle of $V_{\CF_{\wt{G}}}$
is $\le \langle \deg_{\wt{G}}(\CF_{\wt{G}}),\check\mu\rangle+c$.

\smallskip

\noindent{\em(ii)}
If $\mbox{char\,} k=0$ one can take $c=0$.
\end{lem}

\begin{proof}
Statement (i) follows from the fact that the intersection of $\Bun_{\wt{G}}^{ss}$ with every connected component of 
$\Bun_{\wt{G}}$ is quasi-compact, and that under the action of 
$\Bun_{Z_0(\wt{G})}$ on $\Bun_{\wt{G}}$, the number of orbits of $\pi_0(\Bun_{Z_0(\wt{G})})$ on 
$\pi_0(\Bun_{\wt{G}})$ is finite. 

\medskip

Statement (ii) follows from the fact that if $\mbox{char\,} k=0$ then
for every $\CF_{\wt{G}} \in \Bun_{\wt{G}}^{ss}$ the vector bundle $V_{\CF_{\wt{G}}}$ is semistable of slope
$\langle \deg_{\wt{G}}(\CF_{\wt{G}}),\check\mu\rangle$ (a proof of this fact can be found in  \cite[Sect. 3]{RR};
for references to other proofs see the introduction to \cite{RR}).
\end{proof}

\sssec{}

We give the following definition:

\begin{defn} \label{d:strangeness}
The \emph{strangeness} $\on{strng}_X (\wt{G}, V)$ is the smallest number $c\in \BQ$ having the property from 
Lemma \ref{l:2strangeness}(i). %\Volod{I could also introduce $\on{strng}_X^* (\wt{G}, V):=\on{strng}_X^* (\wt{G}, V^*)$.}
\end{defn}

One always has $\on{strng}_X (\wt{G}, V)\ge 0$ (because the trivial $\wt{G}$-bundle is semi-stable).
If $\mbox{char\,} k=0$ then $\on{strng}_X (\wt{G}, V)= 0$.

\begin{rem} \label{r:2vanishing}
As before, let $\wt{G}$ be a reductive group, $V$ a finite-dimensional $\wt{G}$-module on which 
$Z_0(\wt{G})$ acts by a character $\check\mu$, and $\CF\in \Bun_{\wt{G}}^{ss}$. Then
%Let $\wt{G}$, $V$, $\check\mu$ be as in \lemref{l:2strangeness} and let $\CF\in \Bun_{\wt{G}}^{ss}$. Then
%%%%In the situation of Lemma~\ref{l:2strangeness} \Volod{(???)} one has
\[   
H^0(X,V_{\CF_{\wt{G}}})=0\,\, \mbox {   if    }\,\,
\langle \deg_{\wt{G}}(\CF_{\wt{G}}),\check\mu\rangle < -\on{strng}_X (\wt{G}, V),
\]
\[   
H^1(X,V_{\CF_{\wt{G}}})=0\,\, \mbox {   if    }\,\, 
\langle \deg_{\wt{G}}(\CF_{\wt{G}}),\check\mu\rangle >2g-2+\on{strng}_X (\wt{G}, V^*).
\]
In particular, if $\mbox{char\,} k=0$ then 
\begin{equation} \label{e:2vanishing3}
H^0(X,V_{\CF })=0 \,\,\mbox { if  }\,\,
\langle \deg_{\wt{G}}(\CF_{\wt{G}}),\check\mu\rangle <0, \quad H^1(X,V_{\CF })=0 \,\,\mbox { if  }\,\,
\langle \deg_{\wt{G}}(\CF_{\wt{G}}),\check\mu\rangle >2g-2.
\end{equation}
\end{rem}

\ssec{Proof of Proposition~\ref{p:2c'_i}}   \label{sss:proof_2c'_i}

\sssec{}

Let us introduce some notation. Let $P'\subset G$ be a parabolic and $M'\subset P'$ the corresponding
Levi (see our conventions in \secref{sss:conventions_Sect6}); in particular $M'\supset T$.

\medskip

Given a root $\check\alpha$ of $G$ which is not a root of $M'$, define
an $M'$-submodule $V_{M',\check\alpha}\subset\fg$ by
\begin{equation}   \label{e:2the_submodules}
%V_{M',\check\alpha}:=\bigoplus_{\check\beta{\underset{M'}\sim}\check\alpha}\fg_{\check\beta},
V_{M',\check\alpha}:=\bigoplus_{\check\beta,\check\beta-\check\alpha\in R(M')}\fg_{\check\beta},
\end{equation}
where $R(M')$ is the root lattice of $M'$.

\medskip

The coefficient of $\check\alpha_i$ in  a root $\check\alpha$ will be denoted by 
$\coef_i (\check\alpha )$. 

\sssec{}   \label{sss:inequlities for c}
We are going to formulate %Here is 
a slightly more precise version of Proposition~\ref{p:2c'_i}.

\medskip

Let 
$$i\in \Gamma_G \rightsquigarrow c'_i,c''_i\in\BQ$$  be numbers satisfying the following inequalities:

\smallskip

For every Levi subgroup $M'\subset G$, every $i\in\Gamma_G-\Gamma_{M'}$, and every root
$\check\alpha$ of $G$ such that $\coef_i(\check\alpha )>0$, we have:
\begin{equation}    \label{e:c'_i_bound}
\coef_i(\check\alpha )\cdot c'_i\geq  2g-2+\on{strng}_X (M', (V_{M',\check\alpha})^*).
\end{equation}

\begin{equation}      \label{e:c''_i_bound}
\coef_i(\check\alpha )\cdot c''_i\ge \on{strng}_X (M', V_{M',-\check\alpha}).
\end{equation}

\sssec{Remark}
%Note that in 
In the characteristic 0 case we can 
take $c'_i=2g-2$ and $c''_i=0$: indeed, in this case
the numbers $\on{strng}_X$ from formulas 
\eqref{e:c'_i_bound}-\eqref{e:c''_i_bound} are zero. %Hence, in this case 

\sssec{}
Here is the promised version of Proposition~\ref{p:2c'_i}.

\begin{prop}              \label{p:3c'_i}
Let $c'_i,c''_i$ be numbers satisfying the conditions from Sect.~\ref{sss:inequlities for c}.
Let $P\subset G$ be a parabolic, $M$ be the corresponding Levi, 
$\lambda\in\Lambda_G^{+,\BQ}\subset\Lambda_M^{+,\BQ}$ and 
$\CF_M\in\Bun_{M}^{(\lambda)}$. 

\medskip

\noindent{\em(a)} If $\langle\lambda\, , \check\alpha_i\rangle >c'_i$ for all $i\in\Gamma_G-\Gamma_M$
then $H^1(X, \fn (P)_{\CF_M})=0$.

\smallskip

\noindent{\em(b)} If $\langle\lambda\, , \check\alpha_i\rangle >c''_i$ for all $i\in\Gamma_G-\Gamma_M$
then $H^0(X, (\fg/\fp)_{\CF_M})=0$.

\end{prop}

\begin{proof}
Let $P_{\lambda}$ be the parabolic of $M$ corresponding to the subset
$
\{ i\in\Gamma_M\,|\,\langle\lambda\, , \check\alpha_i\rangle =0\}\subset\Gamma_M\, .
$
Let $M_{\lambda}$ be the corresponding Levi.
The fact that $\CF_M\in\Bun_{M}^{(\lambda)}$ means that
$\CF_M$ admits a reduction to $P_{\lambda}$ such that the corresponding $M_{\lambda}$-bundle
$\CF_{M_\lambda}$ is semi-stable of degree $\lambda$.

\medskip

Let us prove statement (a). Note that the vector bundle $\fn (P)_{\CF_M}$ has a canonical
filtration with the associated graded identified with $\fn (P)_{\CF_{M_\lambda}}$. Hence, 
it suffices to show that $H^1(X, \fn (P)_{\CF_{M_\lambda}})=0$.

\medskip

By Remark~\ref{r:2vanishing}, it suffices to prove that
$$\langle \lambda\, ,\check\alpha\rangle >2g-2+\on{strng}_X (M_{\lambda}, (V_{M_{\lambda},\check\alpha})^*)$$
for any positive root $\check\alpha$ of $G$ which is not a root of $M$. 

\medskip

Let 
$i\in\Gamma_G-\Gamma_M$ be such that $\coef_i(\check\alpha )>0$.
Since $\lambda$ is dominant for $G$ and $\langle\lambda\, , \check\alpha_i\rangle >c'_i$ we have 
$\langle\lambda\, , \check\alpha\rangle \ge 
\coef_i(\check\alpha )\cdot \langle\lambda\, , \check\alpha_i\rangle >
\coef_i(\check\alpha )\cdot c'_i\,$. Now use \eqref{e:c'_i_bound} for $M'=M_{\lambda}\,$
(this is possible because $i\not\in\Gamma_M$ and therefore $i\not\in\Gamma_{M_{\lambda}}$).

\medskip

The proof of statement (b) is similar.
\end{proof}

\begin{rem}
If $k$ has characteristic $p>0$ then \propref{p:3c'_i} would become really useful if combined 
with good\footnote{A good upper bound should have the form $c(p,G,\check\alpha)\cdot (g-1)$. The number 
$c(p,G,\check\alpha )$ should be independent of $X$ and small enough to explain the phenomenon in 
\secref{sss:zero_strangeness} below.} 
upper bounds for the strangeness of the relevant representations.
It would be interesting to obtain such bounds.
 \end{rem}

\ssec{Remarks on the positive characteristic case}  \label{ss:p>0}
%Let $k$ be a field of any characteristic and 

\sssec{}  \label{sss:zero_strangeness}
Let $V_1$ and $V_2$ be finite-dimensional vector spaces over a field of any characteristic.
Let $G_i$ denote the algebraic group $GL(V_i)$. Then the $(G_1\times G_2)$-modules
$\Hom (V_1,V_2)$ and $V_1\otimes V_2$ have strangeness $0$. This immediately follows from the definition of semi-stability.

\begin{cor} \label{c:zero_strangeness}
Suppose that $G_{ad}\simeq PGL(d_1)\times\ldots\times PGL(d_n)$. Then all inequalities
\eqref{e:c'_i_bound}-\eqref{e:c''_i_bound} hold for $c'_i=2g-2, c''_i=0$ (without any assumption on $\mbox{char }k$).
\end{cor}

\sssec{}   \label{sss:strangeness>0}
Suppose that $\mbox{char }k=2$. Let $V$ be a 2-dimensional vector space over $k$.
If $g>1$ then the representation of $GL(V)$ in the symmetric square $\on{Sym}^2(V)$ has strangeness
$g-1>0$. This follows from \secref{sss:zero_strangeness}, combined with the exact sequence
\[
0\to V^{(2)}\to \on{Sym}^2(V)\to V\otimes V% \Lambda^2(V)\to 0
\]
and the equality $\on{strng}_X (V^{(2)})=g-1$, which is proved, e.g., in  \cite[Sect.~4.5]{JRXY}.

\sssec{}

The assertion in \secref{sss:strangeness>0} implies that if $g>1$ , $\mbox{char }k=2$ and $G=Sp(2n)$, 
$n\ge 2$, then some of the inequalities  \eqref{e:c'_i_bound} \emph{do not} hold for $c'_i=2g-2$. 

%
%\medskip

\sssec{} The situation for the numbers $c''_i$ is as follows:
%In \cite{He}, J.~Heinloth showed that some of the inequalities \eqref{e:c''_i_bound} \emph{do not} hold 
%either for the group of type $G_2$ over a field of characteristic $2$. 

\medskip

J.~Heinloth \cite{He} proved that if $G$ is a classical group over a field of odd characteristic then
all inequalities \eqref{e:c''_i_bound} hold for $c''_i=0$. He also showed that if $\mbox{char }k=2$ this is still true if $G_{ad}$ 
is a product of groups of type $A$ and $C$. 

\medskip

On the other hand, according to
 \cite{He,P}, some of the inequalities  \eqref{e:c''_i_bound}  do not hold if $\mbox{char }k=2$ and $G_{ad}$ 
 has one of the following types: $G_2$, $B_n$ ($n\ge 3$), $D_n$ ($n\ge 4$).
 %\eqref{e:c''_i_bound} 

%On the other hand, J.~Heinloth \cite{He} proved that if $G$ is a classical group over a field of any characteristic then all inequalities 
%\eqref{e:c''_i_bound} hold for $c''_i=0$. He also showed in \cite{He} that this is not true for the group of type $G_2$ 
%over a field of characteristic $2$.  

\section{Constructing the contraction}  \label{s:constr contr}

The goal of this section is to prove point (b) of \propref{p:key} for the numbers $c_i$ defined
in formula \eqref{e:c_i max} from Sect.~\ref{ss:c_i}.

\ssec{Morphisms between $\Bun_P\,$, $\Bun_{P^-}\,$, $\Bun_M\,$, and $\Bun_G\,$}   \label{ss:P&M}

In this subsection and the next we recall some well known facts that
will be used in the proof of \propref{p:key}(b).

\sssec{}

From now on we fix a (standard) parabolic $P$. Let $P^-$ be the parabolic opposite to $P$ such that
$P^-\supset T$. Note that $P^-$ is \emph{not} a standard parabolic! Namely, $P^-$ is the unique
parabolic such that $P\cap P^-=M$, when the latter is viewed a subgroup of $P$, see \secref{sss:conventions_Sect6}.

\medskip

\begin{lem}  \label{l:2generic}
The morphism $\Bun_M\to\Bun_{P^-}\underset{\Bun_G}\times \Bun_P$
is an open embedding.
\end{lem}

\begin{proof}  
An $M$-bundle on $X$ is the same as a $G$-bundle $\CF_G$ equipped with an $M$-structure, i.e.,
a section of $(G/M)_{\CF_G}$. 

\medskip

The assertion follows from the fact that the morphism $G/M\to G/P^-\times G/P$ is an open embedding.
\end{proof}  

\sssec{}

Define open substacks $\CU_i\subset \Bun_M$ as follows:
\begin{equation}   \label{e:2CU_0}
\CU_0:=\{ \CF_M\in \Bun_M \,|\,H^0(X, (\fg/\fp)_{\CF_M})=0 \} ,
\end{equation}

\begin{equation}   \label{e:2CU_1}
\CU_1:=\{ \CF_M\in \Bun_M \,|\,H^1(X, \fn(P)_{\CF_M})=0 \} .
\end{equation}

\begin{prop}  \label{p:Pplusminus}  We have:

\smallskip

\noindent{\em(a)} The morphism $\iota_P:\Bun_M\to \Bun_P$ induces a 
smooth surjective morphism $$\CU_1\to\Bun_P\underset{\Bun_M}\times\CU_1.$$

\smallskip

\noindent{\em(b)} The morphism $\sfp_{P^-}:\Bun_{P^-}\to \Bun_G$
is smooth when restricted to the open substack $$\Bun_{P^-}\underset{\Bun_M}\times\CU_1\subset\Bun_{P^-}.$$

\smallskip

\noindent{\em(c)} The morphism $\sfq_{P^-}:\Bun_{P^-}\to \Bun_M$ is schematic, affine, and smooth over 
$\CU_0\subset \Bun_M$. 

\medskip

\noindent{\em(d)} The morphism $\iota_{P^-}:\Bun_M\to \Bun_{P^-}$ defines a closed embedding
$$\CU_0\to \Bun_{P^-}\underset{\Bun_M}\times\CU_0.$$

\end{prop}

\begin{proof}
Let $\CF_M\in\CU_1 (k)$, i.e., $\CF_M$ is an $M$-torsor on $X$ such that $H^1(X, \fn(P)_{\CF_M})=0$.
Using an appropriate filtration on $U(P)$ one deduces from this that $H^1(X, U(P)_{\CF_M})=0$, i.e., every
$U(P)_{\CF_M}$-torsor on $X$ is trivial. This implies the surjectivity part of statement (a).
 
\medskip

To prove the smoothness part of (a), it  suffices to check that the differential  of the morphism
 $\iota_P:\Bun_M\to \Bun_P$ at any point $\CF_M\in\CU_1 (k)\subset \Bun_M (k)$ is surjective. Its
cokernel equals $H^1(X, \fn(P)_{\CF_M})$, which is zero by the definition of $\CU_1$, see 
for\-mu\-la~\eqref{e:2CU_1}.

\medskip

Note also that the smoothness part of (a) will follow from~(b):
to see this, decompose the morphism $\iota_P:\Bun_M\to \Bun_P$ as
$$\Bun_M\to\Bun_{P^-}\underset{\Bun_G}\times \Bun_P\to\Bun_P$$
and use Lemma~\ref{l:2generic}. 

\medskip

To prove (b), we have to show that the differential of $\sfp_{P^-}:\Bun_{P^-}\to \Bun_G$ at any $k$-point $y$
of $\Bun_{P^-}\underset{\Bun_M}\times\CU_1$ is surjective.
Its cokernel equals $H^1(X, (\fg/\fp^-)_{\CF_{P^-}})$, where $\CF_{P^-}$ is 
the $P^-$-bundle corresponding to $y$. Let $\CF_M$ be the corresponding $M$-bundle. We have
$H^1(X, \fn (P)_{\CF_M})=0$ by the definition of $\CU_1$, see formula \eqref{e:2CU_1}. Now, the
associated graded of $(\fg/\fp^-)_{\CF_{P^-}}$ with respect to a (canonically defined) filtration
identifies with $(\fg/\fp^-)_{\CF_M}$, and the assertion follows from the fact that the composition
$$\fn(P)\to \fg\to \fg/\fp^-$$
is an isomorphism of $M$-modules.

\medskip

 To prove (c), consider the filtration 
 \[
 U(P^-)=U^1\supset U^2\supset\ldots ,
 \]
 where $U^m$ is the subgroup generated by the root subgroups corresponding to the roots 
 $\check\alpha$ of $G$ such that
 \[
 \sum\limits_{i\not\in\Gamma_M}\coef_i(\check\alpha )\le -m.
 \]
(Here $\coef_i (\check\alpha )$ denotes the coefficient of $\check\alpha_i$ in $\check\alpha$.)
Note that each quotient $U^m/U^{m+1}$ is a vector group (i.e., a product of finitely many copies of $\BG_a$). To prove (c), 
it suffices to check that for each $m$ the morphism 
\begin{equation}    \label{e:mth_morphism}
(\Bun_{P^-/U^m})\underset{\Bun_M}\times\CU_0\to(\Bun_{P^-/U^{m+1}}) \underset{\Bun_M}\times\CU_0
\end{equation}
is schematic, affine and smooth. In fact, it is a torsor over a certain vector bundle. 

\medskip

To see this, note that by 
\eqref{e:2CU_0}, for  each 
$\CF_M\in\CU_0$ we have $$H^0(X, (U^m/U^{m+1})_{\CF_M})=0,$$ so the stack of 
$(U^m/U^{m+1})_{\CF_M}$-torsors on $X$ is a scheme; 
namely, it is  the vector space $H^1(X, (U^m/U^{m+1})_{\CF_M})$. As $\CF_M$ varies, 
these vector spaces form a vector bundle on $\CU_0$. 
Let $\xi$ denote its pullback to
$(\Bun_{P^-/U^m})\underset{\Bun_M}\times\CU_0$, then the morphism \eqref{e:mth_morphism} is a $\xi$-torsor.

\medskip

Point (d) follows from point (c) since the map $\CU_0\to \Bun_{P^-}\underset{\Bun_M}\times\CU_0$
is a section of the map
$$\Bun_{P^-}\underset{\Bun_M}\times\CU_0\to \CU_0,$$
and the latter is schematic and separated.

\end{proof}

\ssec{The action of $\BA^1$ on $\Bun_{P^-}\,$}   \label{ss:A1action}

\sssec{}

Let $Z(M)$ denote the center of $M$. Choose a homomorphism $\mu :\BG_m\to Z(M)$ such that
$\langle\mu\, , \check\alpha_i\rangle>0$ for $i\not\in\Gamma_M$. Then the action of $\BG_m$
on $P^-$ defined by
\begin{equation}   \label{e:3A1-action_on_P}
%\BG_m\times P^-\to P^-, 
\rho_t (x):=\mu (t)^{-1}\cdot x\cdot \mu (t), \quad t\in \BG_m\, , x\in P^-  
\end{equation}
extends to an action of the multiplicative monoid $\BA^1$ on $P^-$ such that the endomorphism
$\rho_0\in\End(P)$ equals the composition $P^-\epi M\mono P^-$. 

\sssec{}

The above action of $\BA^1$ on $P^-$ induces an $\BA^1$-action on $\Bun_{P^-}\,$.
Equip $M$ and $\Bun_M$ with the trivial $\BA^1$-action. The
projection $P^-\to M$ is $\BA^1$-equivariant, so \emph{the corresponding morphism 
$\sfq_{P^-}:\Bun_{P^-}\to\Bun_M$ has a canonical $\BA^1$-equivariant structure.}

\begin{rem}   \label{r:A1action0}
The above description of $\rho_0$ implies that the morphism ${\bf 0}:\Bun_{P^-}\to\Bun_{P^-}$ corresponding to 
$0\in\BA^1$ equals the composition
%of $\sfq_{P^-}:\Bun_{P^-}\to \Bun_M$ and $\iota_{P^-}:\Bun_M\to \Bun_P\,$.
\begin{equation}   \label{e:3the_0_diagram}
\Bun_{P^-}\overset{\sfq_{P^-}}\longrightarrow  \Bun_M \overset{\iota_{P^-}}\longrightarrow\Bun_{P^-}\, .
\end{equation}

\end{rem}

\begin{rem}   \label{r:A1action1}
The action of $\BG_m$ on $\Bun_{P^-}$ is trivial: this follows from formula \eqref{e:3A1-action_on_P}, 
which says that the automorphisms $\rho_t\in\Aut(P^-)$, $t\in\BG_m$, are inner. Moreover, formula  
\eqref{e:3A1-action_on_P} provides a \emph{canonical} trivialization of this action.
\end{rem}

\begin{rem}   \label{r:A1action2}
Despite the previous remark, it is \emph{not true} that the action of $\BG_m$ on each \emph{fiber} 
of the morphism $\Bun_{P^-}\to\Bun_M$ is trivial. (Note that although $\BG_m$ acts on $\Bun_{P^-}$ 
by automorphisms over $\Bun_M\,$, the trivialization of the $\BG_m$-action on $\Bun_{P^-}$ provided by
\eqref{e:3A1-action_on_P} is \emph{not over $\Bun_M\,$}.)
\end{rem}

\begin{rem}  \label{r:A1action3}
It is not hard to show that the trivialization of the $\BG_m$-action on $\Bun_{P^-}$ defined in \remref{r:A1action1} yields an action of 
the monoidal stack\footnote{For any scheme $S$, the groupoid $(\BA^1/\BG_m) (S)$ is the groupoid of line bundles over 
$S$ equipped with a section, so $(\BA^1/\BG_m) (S)$ is a monoidal category with respect to $\otimes$. 
In this sense $\BA^1/\BG_m$ is a monoidal stack.}
 $\BA^1/\BG_m$ on $\Bun_{P^-}\,$. The proof is straightforward; it uses the formula
\[
\rho_t (\mu (s))=\mu (s), \quad\quad t\in\BA^1,\, s\in\BG_m\,   ,
\]
which follows from  \eqref{e:3A1-action_on_P}.
\end{rem}

\ssec{Proof of \propref{p:key}(b)}  

\sssec{}
Let the numbers $c_i$, $i\in\Gamma_G$, be as in formula \eqref{e:c_i max} from Sect.~\ref{ss:c_i}.
Let $S\subset  \Lambda^{+,\BQ}_G$ be a $P$-admissible subset, and assume that  $S$ satisfies \eqref{e:key}, i.e., 
$$\forall\, \lambda\in S,\,\,\forall i\in  \Gamma_G-\Gamma_M \text{ we have }
\langle \lambda\, ,\check\alpha_i\rangle >c_i\, .$$
We have to prove that the locally closed substack $\Bun_G^{(S)}\subset \Bun_G$ is contractive 
in the sense of  Sect.~\ref{sss:contractive}. 

\sssec{}

%Let $S\subset  \Lambda^{+,\BQ}_G$ be a $P$-admissible subset, and assume that  $S$ satisfies 
%\eqref{e:key}, i.e., 
%$$\forall\, \lambda\in S,\,\,\forall i\in  \Gamma_G-\Gamma_M \text{ we have }
%\langle \lambda\, ,\check\alpha_i\rangle >c_i,$$
%where the numbers $c_i$ are is as in formula \eqref{e:c_i max} from \secref{ss:c_i}.

\medskip

Recall that $c_i:=\on{max}(c'_i,c''_i)$, where $c'_i$ and $c''_i$ are as in \propref{p:2c'_i}. So
for all $\lambda\in S$ and $i\in  \Gamma_G-\Gamma_M$ we have 
\begin{equation}   \label{e:lambda:>c1}
\langle \lambda\, ,\check\alpha_i\rangle >c'_i\, ,
\end{equation}

\begin{equation}   \label{e:lambda:>c2}
\langle \lambda\, ,\check\alpha_i\rangle >c''_i\, .
\end{equation}

%By the defining property of the numbers $c'_i$ (see \propref{p:2c'_i}),
By \eqref{e:lambda:>c1} and the assumption on the numbers $c'_i$ (see \propref{p:2c'_i}),
we have

%The fact that 
%$$\forall\, \lambda\in S,\,\,\forall i\in  \Gamma_G-\Gamma_M \text{ we have }
%\langle \lambda\, ,\check\alpha_i\rangle >c'_i$$
%implies that

\begin{equation}      \label{e:H1vanishes}
\Bun_M^{(S)}\subset\CU_1:=\{ \CF_M\in \Bun_M \,|\,H^1(X, (\fn(P))_{\CF})=0 \} .
\end{equation}

\medskip

%The fact that 
%$$\forall\, \lambda\in S,\,\,\forall i\in  \Gamma_G-\Gamma_M \text{ we have }
%\langle \lambda\, ,\check\alpha_i\rangle >c''_i$$
Similarly, \eqref{e:lambda:>c2} implies that

\begin{equation}      \label{e:H0vanishes}
\Bun_M^{(S)}\subset\CU_0:=\{ \CF_M\in \Bun_M \,|\,H^0(X, (\fg/\fp)_{\CF_M})=0 \}. 
\end{equation}

\sssec{}

Let $\Bun_{P^-}^{(S)}\subset\Bun_{P^-}$ denote the preimage of the open substack $\Bun_M^{(S)}\subset\Bun_M$.
The embeddings $M\mono P\mono G$ and $M\mono P^-\mono G$ induce a commutative diagram 
\begin{equation}   \label{e:2the_square}
\xymatrix{
\Bun_M^{(S)}\ar[d]_{\ips}\ar[r]^{\ipms}&\Bun_{P^-}^{(S)}\ar[d]^{\sfppms}\\
\Bun_P^{(S)}\ar[r]^{\sfpps}&\Bun_G
    }
\end{equation}
%\Volod{Keep in mind: the image of the right vertical arrow is NOT contained in  $\Bun_G^{(S)}$!!!}
%%%One can show that this diagram is Cartesian. We will not use this fact; for our purposes statement~(i) of 
%%%the following lemma will suffice.

We summarize the properties of the maps in the above diagram in the following lemma:

\begin{lem}   \label{l:3the_square} \hfill

\smallskip

\noindent{\em(i)} The morphism $\Bun_M^{(S)}\to\Bun_{P}^{(S)}\underset{\Bun_G}\times \Bun_{P^-}^{(S)}$ 
defined by diagram \eqref{e:2the_square} is an open embedding.

\smallskip

\noindent{\em(ii)}  The morphism $\ips :\Bun_{M}^{(S)}\to \Bun_{P}^{(S)}$
is surjective and smooth. 

\smallskip

\noindent{\em(iii)} The morphism $\sfppms:\Bun_{P^-}^{(S)}\to\Bun_G$
is smooth. 

\smallskip

\noindent{\em(iv)} The morphism $\sfpps$ 
induces an isomorphism $\Bun_P^{(S)}\to\Bun_G^{(S)}$.

\smallskip

\noindent{\em(v)} 
The morphism $\ipms :\Bun_{M}^{(S)}\to \Bun_{P^-}^{(S)}$ is a closed embedding. 

\end{lem}

\begin{proof}
Statement (i) follows from Lemma~\ref{l:2generic}. 

\medskip

By \eqref{e:H1vanishes}, statements (ii) and (iii) follow from Proposition~\ref{p:Pplusminus} points (a) and (b),
respectively. 

\medskip

Statement (iv) holds by \propref{p:key}(a). By \eqref{e:H0vanishes}, statement (v) follows 
from \propref{p:Pplusminus}(d).
\end{proof}

\begin{rem}
One can show, using \cite[Proposition 4.4.4]{Sch}, that the map in point (i) of \lemref{l:3the_square}
is an isomorphism for \emph{any} $P$-admissible set $S$ (i.e., $S$ does not even have to satisfy \eqref{e:key}.)
\end{rem}

\sssec{}
Our goal is to prove that the locally closed substack  $\Bun_G^{(S)}\subset \Bun_G$ is contractive.
By the definition of contractiveness (see \secref{sss:contractive}), this follows
from Lemma~\ref{l:3the_square} and the next statement:

\begin{prop}  \label{p:Z_is_contractive}
Let $S$ be as in \propref{p:key}. Then the substack $\on{Im}(\ipms)\subset\Bun_{P^-}^{(S)}$ from 
\lemref{l:3the_square}(v) is contractive.
\end{prop}

\begin{proof}

Equip $\Bun_{P^-}$ with the $\BA^1$-action from \secref{ss:A1action} corresponding to some $\mu:\BG_m\to Z(M)$.
The open substack $\Bun_{P^-}^{(S)}\subset\Bun_{P^-}$ is $\BA^1$-stable, so we obtain an
$\BA^1$-action on $\Bun_{P^-}^{(S)}\,$.  

\medskip

We apply \lemref{l:stacky_contraction} to 
the canonical morphism $\sfq_{P^-}^{(S)}: \Bun_{P^-}^{(S)}\to\Bun_M^{(S)}$ and the above $\BA^1$-action on $\Bun_{P^-}^{(S)}\,$. 
We only have to check that the conditions of the lemma hold. 

\medskip

By \eqref{e:H0vanishes} and Proposition~\ref{p:Pplusminus}(c), the 
morphism $\sfq_{P^-}^{(S)}: \Bun_{P^-}^{(S)}\to\Bun_M^{(S)}$ is schematic and affine. 
Conditions (i)-(ii) from \lemref{l:stacky_contraction} hold by
Remarks~\ref{r:A1action0}-\ref{r:A1action1}.
\end{proof}

\section{Counterexamples}  \label{s:counterexamples}

The goal of this section is to show that the property of being truncatable is a purely ``stacky"
phenomenon, i.e., that it ``typically" fails for non quasi-compact schemes.  

\ssec{Formulation of the theorem}  \label{ss:counterexamples}

\begin{thm}  \label{t:counterexample}
Let $Y$ be an irreducible smooth scheme of dimension $n$, such that for some (or, equivalently, any)
non-empty quasi-compact open $U\subset Y$ the set
\begin{equation}   \label{e:codim1}
\{ y\in Y-U\; | \,\dim_y(Y-U)=\dim Y-1\}
\end{equation}
is not quasi-compact. Then $\Dmod(Y)$ is not compactly generated.
\end{thm}

The theorem will be proved in Sect.~\ref{ss:counterexample-proof} below.
Here are two examples of schemes $Y$ satisfying the condition of Theorem~\ref{t:counterexample}.

\begin{example} Let $I$ be an infinite set and let $Y$ be the non-separated 
curve that one obtains from $\BA^{1}\times I$ by gluing together the open
subschemes $(\BA^{1}-\{ 0\})\times \{ i\}$, $i\in I$ (in other words, $Y$ is the affine line
with the point $0$ repeated $I$ times). 
%Then $Y$ satisfies the condition of Theorem~\ref{t:counterexample}.
\end{example}

\begin{example} Let $X_0$ be a smooth surface and $x_0\in X_0$ a point. Set $U_0=X-\{x_0\}$. Let 
$X_1$ be the blow-up of $X_0$ at $x_0$. Let $x_1\in X_1$ be a point on the exceptional divisor. 
We have an open embedding 
$$U_0=X-\{x_0\}\hookrightarrow X_1-\{x_1\}=U_1$$
such that $U_1-U_0$ is a divisor.
We can now apply the same process for $(X_1,x_1)$ instead of $(X_0,x_0)$. Thus we obtain
a sequence of schemes
$$U_0\hookrightarrow U_1\hookrightarrow U_2\hookrightarrow ...$$
Then $Y:=\bigcup\limits_i\, U_i$ satisfies the condition of Theorem~\ref{t:counterexample}.  Note that $Y$ is separated if $X_0$ is.
\end{example}

\ssec{Proof of \thmref{t:counterexample}}   \label{ss:counterexample-proof}
We will use facts from \secref{sss:compcoh} about the relation between compactness and coherence
(in the easier case of smooth schemes).

\sssec{}

Let $Y$ be a smooth scheme, $Z\subset Y$ a non-empty divisor, and 
$Y-Z=U\overset{j}\hookrightarrow Y$ be the complementary open embedding.  
   
\begin{lem}  \label{l:counterexample1}
Suppose that $\CN\in \Dmod(Y)$ is coherent and $j_*\circ j^*(\CN)=\CN$. Then the singular support $SS(\CN)\subset T^*(Y)$ is 
\emph{not} equal to $T^*(Y)$.
\end{lem}

\begin{proof}
We can assume that $Y$ is affine and $Z$ is smooth. Since $j_*$ is t-exact we can also assume that $\CN$ is in
$\Dmod(\CN)^\heartsuit$. Suppose that $SS(\CN)=T^*(Y)$. Then there exists an injective map $\CalD_Y\hookrightarrow \CN$,
where $\CalD_Y$ is the D-module of differential operators on $Y$. We obtain an injective map 
$j_*\circ j^*(\CalD_Y)\hookrightarrow j_*\circ j^*(\CN)=\CN$. But $\CN$ is coherent while
$j_*\circ j^*(\CalD_Y)$ is not.
\end{proof}

\sssec{}

Let $Y$ be as in Theorem~\ref{t:counterexample} and $\CM\in \Dmod(Y)$ a compact object. Note that by 
Remark \ref{r:c coh non-qc}, $\CM$ is automatically coherent. 
 
\medskip

We claim: 
 
\begin{lem}  \label{l:counterexample2}
$SS(\CM)\ne T^*(Y)$.
\end{lem}

\begin{proof}
By \propref{p:tautological compact}, there exists a quasi-compact open 
$U\overset{j}\hookrightarrow Y$ such that $\CM=j_!(j^*(\CM))$ or equivalently, 
$$\BD^{\on{Verdier}}_Y(\CM)=j_*\circ j^*(\BD^{\on{Verdier}}_Y(\CM)).$$

We can assume that $U\ne\emptyset$ (otherwise $\CM=0$ and $SS(\CM)=\emptyset$). Then the set \eqref{e:codim1}
is non-empty, so after shrinking $Y$ we can assume that the set $Z:=Y-U$ is
a non-empty divisor. 

\medskip

Applying Lemma~\ref{l:counterexample1} to $\CN=\BD^{\on{Verdier}}_Y(\CM)$ we get $SS (\BD^{\on{Verdier}}_Y(\CM))\ne T^*(Y)$.
Finally, $SS(\CM)=SS (\BD^{\on{Verdier}}_Y(\CM))$.
\end{proof}

\sssec{}

Recall that the full subcategory of compact objects in a DG category $\bC$ 
is denoted by  $\bC^c$.

\begin{lem}  \label{l:counterexample3}
Let $\A\subset \Dmod(Y)$ be the DG subcategory generated by $\Dmod(Y)^c$.
If $\CM\in\A$ is coherent then $SS(\CM)\ne T^*(Y)$.
\end{lem}

\begin{proof}
Let $U\overset{j}\hookrightarrow Y$ be a non-empty quasi-compact open subset.

\medskip

Let $\bC\subset \Dmod(U)$ be the full DG subcategory of $\Dmod(U)$ generated by 
$j^*(\Dmod(Y)^c)$. Since $j^*(\Dmod(Y)^c)\subset \Dmod(U)^c$, we have
$$\bC^c=\bC\cap \Dmod(U)^c,$$
and by \corref{c:Karoubi}, the latter is Karoubi-generated by $j^*(\Dmod(Y)^c)$. 

\medskip

This  observation, combined with  Lemma~\ref{l:counterexample2} and the fact that $T^*(U)$ 
is dense in $T^*(Y)$, implies that for any $\CN\in \bC^c$,
$$SS(\CN)\neq T^*(U).$$

\medskip

Now, $j^*(\CM)\in \bC\cap \Dmod_{\on{coh}}(U)$, and since $U$ is quasi-compact, we have 
$\Dmod_{\on{coh}}(U)=\Dmod(U)^c$. Hence, $j^*(\CM)\in \bC^c$, implying the assertion of the lemma.

\end{proof}

\begin{cor}   \label{c:counterexample}
The DG category $\A$ from Lemma~\ref{l:counterexample3} does not contain $\CalD_Y$.
\end{cor}

Theorem~\ref{t:counterexample} clearly follows from Corollary~\ref{c:counterexample}.

\appendix

\section{Preordered sets as topological spaces}  \label{s:preordered}
The material in this section is standard.

\ssec{Definition of the topology}

Given a preordered set $X$ we equip it with the following topology: a subset $U\subset X$ 
is said to be \emph{open} if for every $x\in U$ one has $\{y\in X|y\le x\}\subset U$.

\begin{lem} \label{l:who_is_who}\hfill
\begin{enumerate}
\item[(i)] A subset $F\subset X$ is closed if and only if for every $x\in F$ one has $\{y\in X|y\ge x\}\subset F$.
\item[(ii)] A subset $Z\subset X$ is locally closed if and only if 
\begin{equation}   \label{e:loc_closed}
\forall x_1,x_2\in Z \quad
\{y\in X|x_1\le y\le x_2\}\subset Z.
\end{equation}
\item[(iii)] For every subset $Y\subset X$ the topology on $Y$ corresponding to the induced preordering on $Y$ is induced by the topology on $X$.
\end{enumerate}
\end{lem}

\begin{proof}
We will only prove (ii). Condition \eqref{e:loc_closed} holds for locally closed subsets because it holds for open and closed ones. Conversely, if $Z$ satisfies \eqref{e:loc_closed} then $Z$ has the following representation as $F\cap U$ with $F$ closed and $U$ open:
 \[
 F:=\bar Z=\{x\in X\, |\,\exists z\in Z :\: z\le x\}, \quad U:=\{x\in X\, |\,\exists z\in Z :\: z\ge x\}.
 \]
\end{proof}

\ssec{Continuous maps}

The following is also easy to see: 

\begin{lem} \hfill
Let $X,X'$ be preordered sets equipped with the above topology. 
Then a map $f:X\to X'$ is continuous if and only if it is monotone, i.e., $x_1\le x_2\Rightarrow f(x_1)\le f(x_2)$. \qed
\end{lem}

\section{The Langlands retraction and coarsenings of the Harder-Narsimhan-Shatz stratification}   \label{s:Langlands}

In \secref{ss:recalling} we recall the definition of the Langlands retraction $\fL:\Lambda_G^\BQ\to \Lambda^{+,\BQ}_G$.

\medskip

Using this retraction, we define in \secref{ss:eta-str} a coarsening of the usual Harder-Narasimhan-Shatz stratification of $\Bun_G$ depending on the choice of $\eta\in \Lambda_G^{+,\BQ}$ (the usual stratification itself corresponds to $\eta=0$). 

\medskip

In  \secref{ss:deep inside} we show that if $\eta$ is  ``deep inside" 
$\Lambda_G^{+,\BQ}$ then all the strata of the corresponding stratification are contractive (and therefore truncative if $\mbox{char }k=0$). 
Combined with \propref{p:key} this immediately implies 
Theorem~\ref{t:2cotruncative} (see \secref{sss:proof_of_main} below).

\medskip

In \secref{ss:Relation} we explain the relation between the two proofs of Theorem~\ref{t:2cotruncative}.

\ssec{Recollections on the Langlands retraction}   \label{ss:recalling}
%Let us recall some facts from \cite{Dr}. 
Equip $\Lambda_G^\BQ$ with the $\underset{G}\le$ ordering.
The following notion goes back to \cite[Sect.~4]{La}.

%\sssec{}

\begin{defn} \label{d:retraction}
The  \emph{Langlands retraction} $\fL:\Lambda_G^\BQ\to \Lambda^{+,\BQ}_G$ is defined as follows:
for $\lambda\in \Lambda_G^\BQ\,$, let $\fL (\lambda )$ be the least element of the set 
$\{\mu\in\Lambda^{+,\BQ}_G\;|\; \mu\underset{G}\ge\lambda \}$ in the sense of the $\underset{G}\le$ ordering.
\end{defn}

%The \emph{Langlands retraction} $\fL:\Lambda_G^\BQ\to \Lambda^{+,\BQ}_G$ is defined as follows:
%
%\medskip
%
%For $\lambda\in \Lambda_G^\BQ$, we let $\fL (\lambda )$ be the smallest element of the set 
%$\{\mu\in\Lambda^{+,\BQ}_G\;|\; \mu\underset{G}\ge\lambda \}$. 

\sssec{}

The existence of the least element is not obvious; it was proved by R.P.~Langlands in \cite[Sect.~4]{La}.
The material from  \cite[Sect.~4]{La} is known under the name of ``Langlands' geometric lemmas".
We give a short review of it in [Dr]. In particular, we give there two proofs of the existence of the least element: J.~Carmona's ``metric" proof 
(see  \cite[Sections~2-3]{Dr}) and another one 
(see Sect. 4 of  \cite{Dr}, including Example 4.3).

\sssec{}

It is clear that the map $\fL:\Lambda_G^\BQ\to \Lambda^{+,\BQ}_G$ is an order-preserving retraction. 
The following description of the fibers of $\fL$ is given in \cite[Sect.~4]{La}; see also \cite[Cor.~5.3(iii)]{Dr}.

%We will need the following description of its fibers, which is a reformulation of \cite[Cor.~5.3(iii)]{Dr}.

\begin{lem}   \label{e:fiber of L}
For any $\lambda\in  \Lambda^{+,\BQ}_G$ one has
\[
\fL^{-1}(\lambda)=\lambda+ \sum_{i\in I_{\lambda}}\BQ^{\le 0}\cdot\alpha_i\, , \quad \mbox{where} \quad
I_{\lambda}:=\{i\in\Gamma_G\,|\,\langle \lambda,\alpha_i\rangle=0\}.
\]
\end{lem}

\ssec{The $\eta$-stratification}   \label{ss:eta-str}

\sssec{The $\eta$-shifted Langlands retraction}
Let $\eta\in \Lambda_G^{+,\BQ}$. The map
\begin{equation} \label{e:eta-shifted1}
\fL^+_{\eta}:\Lambda_G^{+,\BQ}\to (\eta+\Lambda_G^{+,\BQ}), 
\quad \fL^+_{\eta}(\lambda):= \fL(\lambda-\eta )+\eta
\end{equation}
is an order-preserving retraction (this follows from a similar property of $\fL$). By definition, 
\begin{equation} \label{e:eta-shifted2}
\forall\lambda'\in \Lambda_G^{+,\BQ}, \forall\lambda\in(\eta+\Lambda_G^{+,\BQ}) \quad\mbox{we have} \quad
\fL^+_{\eta}(\lambda')\leqG\lambda\Leftrightarrow \lambda'\leqG\lambda .
\end{equation}

\sssec{The $\eta$-stratification of $\Bun_G$\,}
In \secref{sss:summary} we defined the Harder-Narasimhan map $\HN: |\Bun_G (k)|\to\Lambda_G^{+,\BQ}$ and formulated three properties of it,
see \lemref{l:summary} (i-iii). Since the map $\fL^+_{\eta}:\Lambda_G^{+,\BQ}\to (\eta+\Lambda_G^{+,\BQ})$ is
order-preserving, the map
\begin{equation}   \label{e:HN_eta}
\HN_{\eta}:  |\Bun_G (k)|\to (\eta+\Lambda_G^{+,\BQ}), \quad \HN_{\eta}:=\fL^+_{\eta}\circ\HN
\end{equation}
has the same three properties. So the fibers of the map \eqref{e:HN_eta} %, i.e., the substacks
form a stratification of $\Bun_G$ with quasi-compact strata. We call it the \emph{$\eta$-stratification} of 
$\Bun_G\,$; the corresponding strata are 
  \begin{equation}   \label{e:eta-stratum}
 \Bun_G^{(\lambda)_\eta}:=\underset{\lambda'\in( \fL^+_{\eta})^{-1}(\lambda )
 }\bigcup\, 
\Bun_G^{(\lambda')},\quad \lambda\in (\eta+\Lambda^{+,\BQ}_{G}).
\end{equation}
 It is clear that the $\eta$-stratification is coarser than the Harder-Narsimhan-Shatz stratification (the word "coarser" 
 is understood in the non-strict sense).

%\sssec{Remark} 
\sssec{Open substacks associated to the $\eta$-stratification} 
 Recall that for each $\lambda \in \Lambda_G^{+,\BQ}$ the open substack 
 $\Bun_G^{(\leq \lambda)}\subset\Bun_G$  is the union of the strata $$\Bun_G^{(\lambda')}, \quad%$ for all 
 \lambda'\leqG\lambda.$$ If one considers similar unions of the strata of the $\eta$-stratification then one gets ``essentially" the same class of open substacks of $\Bun_G\,$; more precisely, we claim that for each 
 $\lambda \in (\eta+\Lambda_G^{+,\BQ})$ one has
% By \eqref{e:eta-shifted2} and \eqref{e:eta-stratum}, for each $\lambda \in (\eta+\Lambda_G^{+,\BQ})$ one has
$$\underset{\lambda'\in (\eta+\Lambda_G^{+,\BQ}),\,\lambda'\leqG \lambda}\bigcup\, 
\Bun_G^{(\lambda')_{\eta}}=\Bun_G^{(\leq \lambda)}.$$
%(So the notation $\Bun_G^{(\leq \lambda)_{\eta}}$ is unnecessary because $\Bun_G^{(\leq \lambda)_{\eta}}
%=\Bun_G^{(\leq \lambda)}$ for all $\lambda \in (\eta+\Lambda_G^{+,\BQ})$.)
This follows from \eqref{e:eta-shifted2} and \eqref{e:eta-stratum}.

\sssec{Changing $\eta$} If $\eta'\in (\eta+\Lambda_G^{+,\BQ})$ then 
$\fL^+_{\eta'}\circ \fL^+_{\eta}=\fL^+_{\eta'}\,$, so the $\eta'$-stratification is coarser than the $\eta$-stratification.
If $\eta'$ and $\eta$ have the same image in $\Lambda_{G_{ad}}^{+,\BQ}$ then $\fL^+_{\eta'}=\fL^+_{\eta}$, so the $\eta'$-stratification 
and the $\eta$-stratification are the same.

\sssec{$(\fL^+_{\eta})^{-1}(\lambda )$ as a $P$-admissible set.} \label{ss:T is P-adm}
Let $\lambda\in(\eta+\Lambda^{+,\BQ}_{G})$. Let $P$ be the parabolic whose Levi quotient, $M$, corresponds to the following subset of 
$\Gamma_G\,$:
\begin{equation} \label{e:2parabolic_choice}
\Gamma_M=\{i\in\Gamma_G\,|\,\langle \lambda-\eta\, ,\check\alpha_i\rangle =0\}.
\end{equation} 
Equivalently,
\begin{equation} \label{e:3parabolic_choice}
\Gamma_G-\Gamma_M=\{i\in\Gamma_G\,|\,\langle \lambda\, ,\check\alpha_i\rangle >
\langle \eta\, ,\check\alpha_i\rangle\}.
\end{equation}

By \lemref{e:fiber of L}, the subset $T_{\lambda}:=(\fL^+_{\eta})^{-1}(\lambda )\subset \Lambda^{+,\BQ}_{G}$ has 
the following description in terms of the $\leqM$ ordering:
\begin{equation}  \label{e:T is P-adm}
T_{\lambda}=\{\lambda'\in  \Lambda_G^{+,\BQ}\,|\, \lambda'\leqM \lambda\}.
\end{equation}
So by \eqref{e:3parabolic_choice} and \lemref{l:2comb1}(a), the set $T_{\lambda}$ %:=(\fL^+_{\eta})^{-1}(\lambda )$ 
is $P$-admissible in the sense of Definition~\ref{d:P-adm}. Moreover, by \eqref{e:eta-stratum} and \eqref{e:T is P-adm}, the stratum $\Bun_G^{(\lambda)_\eta}$ is equal to the locally closed substack 
$\Bun_G^{(T_{\lambda})}$ defined in \secref{sss:Bungs} by formula~\eqref{e:Bungs}.

\ssec{The case where $\eta$ is ``deep inside" $\Lambda_G^{+,\BQ}$}  \label{ss:deep inside}
\sssec{Contractiveness of the strata}
Suppose now that 
\begin{equation}  \label{e;what is deep}
\langle\eta\, ,\check\alpha_i\rangle \ge c_i \quad\mbox{for all} \quad i\in\Gamma_G\, ,
\end{equation}
where the numbers $c_i\in\BQ^{\ge 0}$ are as in \propref{p:key}.
\begin{prop}
Under these conditions, all strata of the $\eta$-stratification are contractive.
\end{prop}

\begin{proof}
Let $\lambda\in(\eta+\Lambda^{+,\BQ}_{G})$. By \secref{ss:T is P-adm}, 
$\Bun_G^{(\lambda)_\eta}=\Bun_G^{(T_{\lambda})}$, where $T_{\lambda}\subset \Lambda^{+,\BQ}_{G}$ is the $P$-admissible set defined 
by \eqref{e:T is P-adm}. So by \propref{p:key}(b), it suffices to check that for all 
$\lambda'\in T_{\lambda}$ and $i\in\Gamma_G-\Gamma_M$ one has $\langle\lambda'\, ,\check\alpha_i\rangle> c_i$.
If $\lambda'=\lambda$ this  is clear from  \eqref{e:3parabolic_choice} and \eqref{e;what is deep}. The general case follows by \lemref{l:2comb1}(a).
%Indeed, since $\lambda\in(\eta+\Lambda^{+,\BQ}_{G})$ and $i$ is not in the set \eqref{e:2parabolic_choice} 
%one has $\langle\lambda\, ,\check\alpha_i\rangle> \langle\eta\, ,\check\alpha_i\rangle \ge c_i\,$.
\end{proof}
 
\sssec{The characteristic 0 case}
Now assume that $\mbox{char }k=0$. Then by \propref{p:key}, one can take $c_i=\on{max}(0,2g-2)$, where $g$ is the genus of $X$. In this situation condition \eqref{e;what is deep} can be rewritten as
%$\eta\in (\eta_0+ \Lambda^{+,\BQ}_{G})$, where
\begin{equation}   \label{e:eta_0}
\eta\in (\eta_0+ \Lambda^{+,\BQ}_{G}), \quad \mbox{where}\quad\eta_0:=\on{max}(0,2g-2)\cdot \rho 
\end{equation}
(as usual, $\rho$ denotes the half-sum of positive coroots). E.g., one can take $\eta=\eta_0\,$.

\medskip

In characteristic 0 we have the notion of truncativeness and the fact that
contractiveness %$\Rightarrow$ 
implies truncativeness, see \corref{c:contraction principle}. Thus we get part (i) of the following

\begin{cor}  \label{c:deep_char0}
Suppose that $\mbox{char }k=0$ and %$\eta\ge\on{max}(0,2g-2)$.
\begin{equation}  \label{e:5what is deep}
\langle\eta\, ,\check\alpha_i\rangle \ge \on{max}(0,2g-2) \quad\mbox{for all} \quad i\in\Gamma_G\,.
\end{equation}
Then:

\smallskip  

\noindent{\em(i)} all strata of the $\eta$-stratification are truncative;

\smallskip

\noindent{\em(ii)}  the open strata of the $\eta$-stratification are co-truncative;
\end{cor}

\begin{proof}
We have already proved (i). The complement of an open stratum is a union of strata, so statement (ii) follows from \propref{p:trunc and strat}.
\end{proof}

\sssec{Proof of Theorem~\ref{t:2cotruncative}}   \label{sss:proof_of_main}
We have to show that the substack $\Bun_G^{(\leq \eta)}\subset\Bun_G$
is co-truncative if $\langle \eta \, ,\check\alpha_i\rangle \geq 2g-2$ for all $i\in\Gamma_G\,$.
If $g=0$ then any open substack of $\Bun_G$ is co-truncative by Sect.~\ref{sss:finitely many points}. So we can assume that \eqref{e:5what is deep} holds. Then the statement follows from \corref{c:deep_char0}(ii) because
$\Bun_G^{(\leq \eta)}$ is an open stratum of the $\eta$-stratification; namely, it is
the stratum corresponding to $\eta\in (\eta+ \Lambda^{+,\BQ}_{G})$.
\qed

\ssec{Relation between the two proofs of Theorem~\ref{t:2cotruncative}}   \label{ss:Relation}
Suppose that $\mbox{char }k=0$ and $g\ge 1$.  
%In the proof from \secref{sss:proof_of_main} we used the strata of the $\eta$-stratification. On the other hand, 

In the proof of Theorem~\ref{t:2cotruncative} given in 
\secref{ss:proof_modulo} we used substacks 
$\Bun_G^{(S_{\lambda})}\subset\Bun_G$, $\lambda\in \Lambda_G^{+,\BQ}$, where
\begin{equation}   \label{e:S and I}
S_{\lambda}:=\{\lambda'\in\Lambda_G^{+,\BQ}\,|\,\lambda'-\lambda\in\sum_{i\in I}\BQ^{\le 0}\cdot\alpha_i\},
\quad I:=\{i\in\Gamma_G\,|\,\langle \lambda,\alpha_i\rangle\le 2g-2\}.
\end{equation}
%and $\Gamma_M:=\{i\in\Gamma_G\,|\,\langle \lambda\, ,\check\alpha_i\rangle\le 2g-2\}$.
These substacks are related to the strata of the $\eta_0$-stratification, where $\eta_0$ is as in \eqref{e:eta_0}.
The relation is as follows. The stratum of the $\eta_0$-stratification corresponding to 
$\lambda\in (\eta_0+\Lambda_G^{+,\BQ})$ equals $\Bun_G^{(S_{\lambda})}$. On the other hand, for any
 $\lambda\in \Lambda_G^{+,\BQ}$ the stack $\Bun_G^{(S_{\lambda})}$ is a locally closed substack of the
 stratum of the $\eta_0$-stratification corresponding to $\fL^+_{\eta_0} (\lambda)$. (The proof of these facts is left to the reader.)

\section{A stacky contraction principle}   \label{s:stacky_contraction}
The main goal of this appendix is to prove Theorem~\ref{t:adjunctions} and 
Corollaries~\ref{c:adjunctions}-\ref{c:embedding}. %, which imply \propref{p:stacky_contraction}.

Corollary~\ref{c:embedding} is a ``contraction principle", which is slightly more general than \propref{p:2contraction principle}. Theorem~\ref{t:adjunctions} and Corollary~\ref{c:adjunctions} are generalizations of the classical adjunction from \propref{p:adjointness}. 

\medskip

Convention: throughout this appendix algebras are always associative but not necessarily unital; coalgebras are coassociative but not necessarily counital.

\ssec{Idempotent algebras in monoidal categories}   \label{sss:idemp_algebras}
The notions of algebra and coalgebra make sense in any monoidal category. 
%There is also a dual notion of a coalgebra in $\CM$.

\begin{defn}     \label{d:idemp_algebra}
An algebra $A$ in a monoidal category is said to be \emph{idempotent} if the multiplication morphism $A\otimes A\to A$ is an isomorphism.
\end{defn}

\begin{rem}      \label{r:idemp_coalgebra}
The dual notion of idempotent coalgebra is, in fact, equivalent to that of idempotent algebra:
an isomorphism $m:A\otimes A\to A$ is an algebra structure if and only if 
$m^{-1}:A\to A\otimes A$ is a coalgebra structure.
\end{rem}

In any monoidal category $\CM$ the unit object ${\mathbf 1}_{\CM}$ has a canonical structure of idempotent algebra.

Here is another example. Any monoid $M$ can be considered as a monoidal category
(with $M$ as the set of objects and no morphisms other than the identities).
In particular, this applies to $\{ 0,1\}$ as a monoid with respect to multiplication. Clearly $0$ is an idempotent algebra in the monoidal category $\{ 0,1\}$.

\begin{rem}    \label{r:0}
The category of idempotent algebras in a monoidal category $\CM$ is equivalent to the
category of monoidal functors $F:\{ 0,1\}\to\CM$; namely, the idempotent algebra corresponding to $F$ is $F(0)$.
\end{rem}

%\begin{rem}    \label{r:1}
%Let $A$ be an idempotent algebra in a monoidal category $\CM$. If $A$ is isomorphic to 
%${\mathbf 1}_{\CM}$ as an object of $\CM$ then there exists a unique \emph{algebra} isomorphism
%$f:{\mathbf 1}_{\CM}\iso A$ (namely, $f$ is the unit of $A$).
%\end{rem}

Let $\CC$ be a category. By an \emph{idempotent functor} $\CC\to\CC$ we mean an idempotent algebra in the 
monoidal category of functors $\CC\to\CC$. One can think of idempotent functors
in terms of the two mutually inverse constructions below.

\medskip

\noindent \emph{Construction 1.} Suppose we have categories $\CA$ and $\CC$, functors
$\CA\overset{i}\longrightarrow\CC\overset{\pi}\longrightarrow\CA$, and an isomorphism
$f:\pi\circ i\iso\Id_{\CA}$. Set ${\mathbf 0}:=i\circ\pi$. Then 
${\mathbf 0}:\CC\to\CC$ is an idempotent functor: the isomorphism 
${\mathbf 0}\circ{\mathbf 0}\iso{\mathbf 0}$ is the composition  
\[
{\mathbf 0}\circ{\mathbf 0}=(i\circ\pi)\circ(i\circ\pi)=i\circ(\pi \circ i)\circ\pi\iso i\circ\Id_{\CA}\circ\pi
=i\circ\pi ={\mathbf 0}.
\]

\medskip

\noindent \emph{Construction 2.} Let $\CC$ be a category equipped with an 
idempotent functor ${\mathbf 0}:\CC\to\CC$. Equivalently, $\CC$ carries an action of the monoid
$\{ 0,1\}$ (see Remark~\ref{r:0}). Let $\CC^0$ be the category of $\{ 0,1\}$-equivariant functors
$\{ 0\}\to\CC$. Then the $\{ 0,1\}$-equivariant maps $\{ 0\}\mono\{ 0,1\}\to\{ 0\}$ induce functors
$\CC^0\overset{\pi}\longleftarrow\CC\overset{i}\longleftarrow\CC^0$ with $\pi\circ i=\Id_{\CC^0}\,$.
Equivalently, one can think of $\CC^0$ as the category of 
${\mathbf 0}$-modules\footnote{This notion makes sense because ${\mathbf 0}$ is an algebra in the monoidal category of functors 
$\CC\to\CC$. } $c$ in $\CC$ such that the morphism ${\mathbf 0}\cdot c\to c$ is an isomorphism;
then $i:\CC^0\to\CC$ is the forgetful functor and $\pi:\CC\to\CC^0$ is the ``free module" functor.

\medskip

It is easy to check that the above Constructions 1 and 2 are mutually inverse.\footnote{This is a ``baby case" of
the theory of retracts and idempotents in $\infty$-categories from \cite[Sect.~4.4.5]{Lu1}.}

\begin{rem}    \label{r:monad-comonad}
In the situation of Construction 1, the algebra ${\mathbf 0}$ is unital (or equivalently, is a monad)
if and only if $(\pi ,i)$ is an adjoint pair of functors with $f:\pi\circ i\iso\Id_{\CA}$ being one of the
adjunctions (in this case the unit of ${\mathbf 0}$ is the other adjunction). Similarly, ${\mathbf 0}$ is a counital 
coalgebra (or equivalently, a comonad) if and only if $(i,\pi )$ is an adjoint pair.
\end{rem}

\ssec{The monodromic subcategory}      \label{sss:monodromic}
Let $\CY$ be a QCA stack equipped with a $\BG_m$-action. Then one has the quotient stack $\CY/\BG_m$ and the canonical morphism $p:\CY\to\CY/\BG_m$.

\begin{defn}         \label{d:monodromic}
The \emph{monodromic subcategory} $\Dmod (\CY )_{\mmu}\subset\Dmod (\CY )$ is the subcategory 
generated by the essential image of $p^!:\Dmod (\CY/\BG_m)\to \Dmod (\CY)$ (or equivalently, by the essential image of $p^*$).
\end{defn}

(A more precise name for $\Dmod (\CY )_{\mmu}$ would be ``unipotently monodromic subcategory.")  

%\begin{defn}         \label{d:weakly trivial}
%A $\BG_m$-action on $\CY$ is said to be \emph{weakly trivial} if for each $y\in \CY (k)$ and each 
%$\lambda\in k^{\times}=\BG_m (k)$ the point  $\lambda\cdot y\in \CY (k)$ is isomorphic to $y$. 
%\end{defn}

\begin{lem}        \label{l:monodromic}
If the $\BG_m$-action on $\CY$ is trivial then $\Dmod (\CY )_{\mmu}=\Dmod (\CY )$.
\end{lem}

\begin{proof}
A trivialization of the $\BG_m$-action on $\CY$ identifies $\CY/\BG_m$ with $\CY\times(\on{pt}/\BG_m)$
and the morphism $p:\CY\to\CY/\BG_m$ with the canonical morphism $\CY=\CY\times\on{pt}\to\CY\times (\on{pt}/\BG_m)$.
\end{proof}

\ssec{Recollections on the renormalized direct image} 
Let $\pi:\CY_1\to \CY_2$ be a morphism of QCA stacks.
The \emph{renormalized direct image functor} 
$$\pi_{\blacktriangle}:\Dmod(\CY_1)\to \Dmod(\CY_2)$$
is defined in \cite[Sect.~9.3]{finiteness} to be the functor dual to $\pi^!:\Dmod(\CY_2)\to \Dmod(\CY_1)$ 
(dual in the sense of Sects.~\ref{sss:dual functor} and \ref{sss:2Verdier} of this article). By definition, 
$\pi_{\blacktriangle}$ is continuous.
One also has a not necessarily continuous de Rham direct image functor 
$\pi_{\dr,*}:\Dmod(\CY_1)\to \Dmod(\CY_2)$, see \cite[Sect.~7.4]{finiteness}. If %the latter 
$\pi_{\dr,*}$ is continuous then one has a canonical
isomorphism $\pi_{\blacktriangle}\iso\pi_{\dr,*}\,$, see \cite[Corollary 9.3.8]{finiteness}. 
For instance, this happens if %$\pi$ is schematic or more generally, if 
the fibers of $\pi$ are algebraic spaces, see \cite[Corollary 10.2.5]{finiteness}.

\ssec{Formulation of the theorem}   \label{sss:unital_theorem}
Let $\CY$ be a QCA stack equipped with an action of the multiplicative monoid
$\BA^1$. Let ${\mathbf 0}\in\Mor (\CY,\CY)$ denote the endomorphism of $\CY$ corresponding to
$0\in\BA^1$. One has continuous functors 
${\mathbf 0}^!,{\mathbf 0}_{\rD}:\Dmod (\CY )\to\Dmod (\CY )$ (the functor ${\mathbf 0}_{\dr,*}$
is continuous only if it equals ${\mathbf 0}_{\rD}$).
By Remark~\ref{r:0}, ${\mathbf 0}$ is an idempotent algebra in the monoidal category 	$\Mor (\CY,\CY)$.
So the functors ${\mathbf 0}_{\rD}$ and ${\mathbf 0}^!$ are idempotent algebras in the monoidal category $\on{Funct}_{\on{cont}}(\Dmod (\CY ),\Dmod (\CY ))$ and also in
the monoidal category $\on{Funct}_{\on{cont}}(\Dmod (\CY )_{\mmu},\Dmod (\CY )_{\mmu})$
(here $\Dmod (\CY )_{\mmu}\subset\Dmod (\CY )$ is the monodromic subcategory, see 
Sect.~\ref{sss:monodromic}). By Remark~\ref{r:idemp_coalgebra}, one can also consider 
${\mathbf 0}_{\rD}$ and ${\mathbf 0}^!$ as idempotent coalgebras.

%Let $\CY$ be an algebraic stack equipped with an action of the multiplicative monoid
%$\BA^1$ such that the action morphism $\on{act}:\BA^1\times\CY\to\CY$ is QCA in the sense of \ref{?}.
%(Example: $\CY=W/\BG_m$, where $W$ is as in \secref{sss:formulation}.)
%Then the morphism ${\mathbf 0}:\CY\to\CY$ corresponding to $0\in\BA^1$ is QCA. Conversely, the 
%commutative diagram
%\[
%\CD
%\BA^1\times\CY\  @>{\on{act}}>>  \CY \\
%@V{\on{pr}}VV    @VV{\bf 0}V   \\
%\CY  @>{\bf 0}>> \CY
%\endCD
%\]
%shows that if ${\mathbf 0}:\CY\to\CY$ is QCA then so is $\on{act}:\BA^1\times\CY\to\CY$.
%
%One has continuous functors 
%${\mathbf 0}^!,{\mathbf 0}_{\rD}:\Dmod (\CY )\to\Dmod (\CY )$ 
%(the functor ${\mathbf 0}_{\rD}$ is defined because the morphism ${\mathbf 0}$ is QCA).
%By Remark~\ref{r:0}, ${\mathbf 0}$ is an idempotent algebra in the monoidal category 	$\Mor (\CY,\CY)$.
%So the functors ${\mathbf 0}_{\rD}$ and ${\mathbf 0}^!$ are idempotent algebras in the monoidal category $
%\on{Funct}_{\on{cont}}(\Dmod (\CY ),\Dmod (\CY ))$ and also in
%the monoidal category $\on{Funct}_{\on{cont}}(\Dmod (\CY )_{\mmu},\Dmod (\CY )_{\mmu})$
%(here $\Dmod (\CY )_{\mmu}\subset\Dmod (\CY )$ is the monodromic subcategory, see 
%Sect.~\ref{sss:monodromic}). By Remark~\ref{r:idemp_coalgebra}, one can also consider 
%${\mathbf 0}_{\rD}$ and ${\mathbf 0}^!$ as idempotent coalgebras.

\begin{thm}      \label{t:unital}
The algebra ${\mathbf 0}_{\rD}\in\on{Funct}_{\on{cont}}(\Dmod (\CY )_{\mmu},
\Dmod (\CY )_{\mmu})$ is unital.
The coalgebra ${\mathbf 0}^!\in\on{Funct}_{\on{cont}}(\Dmod (\CY )_{\mmu},\Dmod (\CY )_{\mmu})$ is counital.
\end{thm}

A proof will be given in Sect.~\ref{sss:key_lemma}-\ref{sss:monodromic_proof}. A slightly different proof will be sketched in Sect.~\ref{sss:2monodromic_proof}.  

\begin{cor}   \label{c:unital}
If the $\BG_m$-action on $\CY$ is trivial
%weakly trivial in the sense of Definition~\ref{d:weakly trivial}
then the algebra $${\mathbf 0}_{\rD}\in\on{Funct}_{\on{cont}}(\Dmod (\CY ),\Dmod (\CY ))$$ is unital and
the coalgebra $${\mathbf 0}^!\in\on{Funct}_{\on{cont}}(\Dmod (\CY ),\Dmod (\CY ))$$ is counital.
\end{cor}

\begin{proof}
Use Theorem~\ref{t:unital} and Lemma~\ref{l:monodromic}.
\end{proof}

\ssec{Reformulation in terms of adjunctions}    \label{sss:reformulation_adjunctions}
Let $\CY$ 
%\Volod{and the action of $\BA^1$ on $\CY$} 
be as in Sect.~\ref{sss:unital_theorem}. In particular, the submonoid $\{ 0, 1\}\subset\BA^1$ acts on $\CY$. Define $\CY^0$ to be the stack of $\{ 0, 1\}$-equivariant maps $\{ 0 \}\to\CY$. Equivalently, for any test scheme $S$, the groupoid $\CY^0 (S)$ is
obtained from $\CY (S)$ using Construction 2 from Sect.~\ref{sss:idemp_algebras}.  It is clear that the stack $\CY^0$ is QCA.
%\Volod{(Note that if $y^0\in\CY^0 (k)$ and $y$ is its image in $\CY (k)$ then 
%$\Aut (y^0)\subset \Aut (y)$.)}

The $\{ 0,1\}$-equivariant maps $\{ 0,1\}\to\{ 0\}\mono\{ 0,1\}$ induce morphisms
\begin{equation}    \label{e:the_0_diagram}
\CY\overset{i}\longleftarrow\CY^0\overset{\pi}\longleftarrow\CY, \quad\quad
\pi\circ i=\Id_{\CY^0},\quad i\circ\pi ={\mathbf 0}.
\end{equation}
The $\BA^1$-action on $\CY$ induces an $\BA^1$-action on the diagram 
\eqref{e:the_0_diagram}. The $\BA^1$-action on $\CY^0$ is canonically trivial (this follows from the identity 
$\lambda\cdot 0=0$  in $\BA^1$). So $\Dmod (\CY^0)_{\mmu}=\Dmod (\CY^0)$.

\begin{example}   \label{ex:real_life}
Let $\CY$ be the stack $\Bun_{P^-}$ equipped with the $\BA^1$-action from Sect~\ref{ss:A1action}.
Then $\CY^0=\Bun_M$ and diagram \eqref{e:the_0_diagram} identifies with diagram \eqref{e:3the_0_diagram}.
\end{example}

\begin{rem}     \label{r:representability_of_i}
Since $\pi\circ i=\Id_{\CY^0}$ the morphism $i$ is representable\footnote{Example~\ref{ex:real_life} shows that $i$ is not necessarily a monomorphism. Simpler example: let $G$ be an affine algebraic group equipped with an  $\BA^1$-action and set
$\CY:=\on{pt}/G$, then $\CY^0=\on{pt}/G^0$, where $G^0\subset G$ is the subgroup of $\BA^1$-fixed points.} (i.e., its fibers are algebraic spaces rather than stacks). So the renormalized direct image functor $i_{\rD}$ equals the 
``usual" direct image $i_{\dr,*}\,$.
\end{rem}

By Remarks~\ref{r:monad-comonad} and \ref{r:representability_of_i}, one can reformulate Theorem~\ref{t:unital} and Corollary~\ref{c:unital} as follows.

\begin{thm}      \label{t:adjunctions}
%The algebra ${\mathbf 0}_{\rD}\in\on{Funct}_{\on{cont}}(\Dmod (\CY )_{\mmu},\Dmod (\CY )_{\mmu})$ is unital.
%The coalgebra ${\mathbf 0}^!\in\on{Funct}_{\on{cont}}(\Dmod (\CY )_{\mmu},\Dmod (\CY )_{\mmu})$ is counital.

%Under the above assumptions, the 
The functors 
\begin{equation}
\pi_{\rD}:\Dmod (\CY )_{\mmu}\rightleftarrows \Dmod (\CY^0 ):i_{\dr,*}\, ,\quad\quad 
i^!:\Dmod (\CY )_{\mmu} \rightleftarrows \Dmod (\CY^0 ):\pi^!
\end{equation}
form adjoint pairs with the adjunctions  $\pi_{\rD}\circ i_{\dr,*}\iso\Id_{\Dmod (\CY^0 )}$ and 
$\,\Id_{\Dmod (\CY^0 )}\iso  i^!\circ \pi^!$ coming from the isomorphism $\pi\circ i\iso\Id_{\CY^0}\,$. 
\end{thm}

\begin{cor}   \label{c:adjunctions}
If the $\BG_m$-action on $\CY$ is trivial 
%weakly trivial in the sense of Definition~\ref{d:weakly trivial}
then the functors 
\begin{equation}
\pi_{\rD}:\Dmod (\CY )\rightleftarrows \Dmod (\CY^0 ):i_{\dr,*}\, ,\quad\quad 
i^!:\Dmod (\CY )\rightleftarrows\Dmod (\CY^0 ):\pi^!
\end{equation}
form adjoint pairs with the adjunctions  $\pi_{\rD}\circ i_{\dr,*}\iso\Id_{\Dmod (\CY^0 )}$ and 
$\,\Id_{\Dmod (\CY^0 )}\iso  i^!\circ \pi^!$ coming from the isomorphism $\pi\circ i\iso\Id_{\CY^0}$.
\end{cor}

\begin{cor}  \label{c:embedding}
Suppose that the $\BG_m$-action on $\CY$ is 
trivial and the morphism $i:\CY^0\to\CY$ is a composition of an almost-isomorphism\footnote{See 
Definition~\ref{d:almost-iso}.} 
$\CY^0\to\CZ$ and a locally closed embedding $\CZ\mono\CY$. Then the substack $\CZ\subset\CY$ is truncative.
%Moreover, the nonstandard functors $i^*:\Dmod (\CY )\to\Dmod (\CY^0)$ and 
%$i_?:\Dmod (\CY^0 )\to\Dmod (\CY)$ (see Sect.~\ref{sss:voprosial}) identify with the following
%standard functors:
%\begin{equation}    \label{e:nonstandard_as_standard}
%i^*=\pi_{\rD}\; , \quad i_?= \pi^!\,.
%\end{equation}
\end{cor}

%%\begin{rem}
%%Of course, one can reformulate \eqref{e:nonstandard_as_standard} as follows: the functors $
%%\pi^*, \pi_!$ are everywhere defined and $\pi_! =i^!$, $\pi^*= i_{\dr,*}\,$.
%% \Volod{Do we write $\pi_! $ or $\pi_{\dr,!}\,$?}
%%\end{rem}

\begin{rem}    \label{r:embedding}
The assumption of \corref{c:embedding} is equivalent to the following one: 
%${\mathbf 0} (\CY)\subset\CZ$ and ${\mathbf 0}|_{\CZ}\simeq\Id_{\CZ}$ 
%(in this case $\CY^0=\CZ$).
${\mathbf 0} (\CY)\subset\CZ$ and ${\mathbf 0}|_{\CZ}:\CZ\to \CZ$ is an almost-isomorphism.
\end{rem}

\begin{rem}     \label{r:without_QCA}
%The truncativeness statement from 
\corref{c:embedding} holds even if $\CY$ is \emph{locally } QCA (but not necessarily quasi-compact).
To show this, we can assume that $\CY^0$ is quasi-compact
(otherwise replace  $\CY$ by $\CY\times_{\CY^0}S$, where $S$ is any 
quasi-compact scheme equipped with a smooth morphism to $\CY^0$). Then $\CY^0$ is contained in a
quasi-compact open substack $U\subset\CY$. Since the $\BG_m$-action on $\CY$ is trivial $U$ is
$\BA^1$-stable. Applying \corref{c:embedding} to $U$ we see that $\CY^0$ is truncative in $U$ and therefore in $\CY$.
\end{rem}

\ssec{The key lemma}     \label{sss:key_lemma}
Similarly to the notion of monoidal groupoid, there is a notion of monoidal stack.
Of course, any algebraic group or the multiplicative monoid $\BA^1$ are examples of monoidal stacks. In the proof of 
Theorem~\ref{t:unital} we will use the monoidal stack $\BA^1/\BG_m$.
(If you wish, $S$-points of $\BA^1/\BG_m$ can be interpreted as line bundles over $S$ equipped with a section; 
this is a monoidal category with respect to $\otimes$).

Let $G$ be a monoidal QCA stack over $k$. Then $\Dmod (G)$ is a monoidal category with respect to the convolution
\begin{equation}
M*N:=m_{\rD}(M\boxtimes N), \quad M,N\in\Dmod (G),
\end{equation}
where $m:G\times G\to G$ is the multiplication map.

For any $g\in G(k)$ define $\underline g\in\Dmod (G)$ to be the direct image of
$k\in \Dmod (\on{pt})$ under the map $g:\on{pt}\to G$ (this is a kind of ``delta-function" at $g$).
The assignment $g\mapsto\underline g$ is a monoidal functor $ G(k)\to\Dmod (G)$. In particular, 
$\underline 1\in\Dmod (G)$ is the unit object.

If $f:G_1\to G_2$ is a morphism of monoidal stacks then $f_{\rD}:\Dmod (G_1)\to\Dmod (G_2)$
is a monoidal functor. If $f$ is only a morphism of semigroups then $f_{\rD}$ is a semigroupal
\footnote{There exists a precedent of the usage of ``semigroupal" in the literature; this word means
``monoidal, but without asking that the unit map to the unit."}
functor, so $f_{\rD}(\underline 1)\in\Dmod (G_2)$ is an idempotent algebra.
% (but not necessarily the unit object).

%\Den{Does the word ``semi-groupal" exist? Maybe better homomorphism of semi-groups?}

Applying this to $0:\on{pt}\to\BA^1/\BG_m$ we see that $\underline 0\in\Dmod (\BA^1/\BG_m)$ is an idempotent algebra.

\begin{lem}        \label{l:key_lemma}
The algebra $\underline 0\in\Dmod (\BA^1/\BG_m)$ is unital.
\end{lem}

\begin{proof}       
Consider the morphisms 
$\{ 0\}/\BG_m\overset{i}\mono\BA^1/\BG_m\overset{\pi}\longrightarrow\{ 0\}/\BG_m$
induced by the morphisms $\{ 0\}\mono\BA^1\to\{ 0\}$. Set $\CC:=\Dmod (\{ 0\}/\BG_m)$; this is
a monoidal category because $\{ 0\}/\BG_m$ is a monoidal stack. We have a monoidal functor
$\pi_{\dr,*}:\Dmod (\BA^1/\BG_m)\to\CC$ and a semigroupal functor 
$i_{\dr,*}:\CC\to\Dmod (\BA^1/\BG_m)$ with $\pi_{\dr,*}\circ i_{\dr,*}=\Id_{\CC}\,$. By definition, 
$\underline 0=i_{\dr,*}({\mathbf 1}_{\CC})$, where ${\mathbf 1}_{\CC}$ is the unit object of $\CC$.

Let us now construct the unit $e:\underline 1\to\underline 0$ of the algebra $\underline 0$.
By Sect.~\ref{sss:baby_Springer},  $(\pi_{\dr,*},i_{\dr,*})$ is an adjoint pair of functors (this is the ``baby case" of Theorem~\ref{t:adjunctions}). So
\[
\CMaps(\underline 1,\underline 0)=\CMaps(\underline 1,i_{\dr,*}({\mathbf 1}_{\CC}))=
\CMaps(\pi_{\dr,*}(\underline 1), {\mathbf 1}_{\CC})=\CMaps( {\mathbf 1}_{\CC}, {\mathbf 1}_{\CC});
\]
more precisely, the map 
$\pi_{\dr,*}:\CMaps (\underline 1,\underline 0)\to\CMaps( {\mathbf 1}_{\CC}, {\mathbf 1}_{\CC})$
is an isomorphism. Define $e:\underline 1\to\underline 0$ to be the morphism such that
$\pi_{\dr,*}(e)$ equals $\id :{\mathbf 1}_{\CC}\iso{\mathbf 1}_{\CC}$.

Let us show that $e$ is indeed a unit. 
Let $f:\underline 0\to\underline 0$ denote the composition of the morphism 
$e*\id_{\underline 0}:\underline 0=\underline 1*\underline 0\to \underline 0*\underline 0$
with the multiplication map $\underline 0*\underline 0\to \underline 0$. We have to prove that
$f=\id_{\underline 0}\,$. 
%The restriction of the functor $\pi_{\dr,*}:\Dmod (\BA^1/\BG_m)\to\CC$ to the essential image of
%induces an equivalence $i_{\dr,*}(\CC)\iso\CC$, 
%\Den{I don't understand what this formula means.}
To do this,  it suffices to show that $\pi_{\dr,*} (f)$ equals the identity.
This is clear because $\pi_{\dr,*}$ is a monoidal functor and
$\pi_{\dr,*}(e)$ equals $\id :{\mathbf 1}_{\CC}\iso{\mathbf 1}_{\CC}$.
\end{proof}

\ssec{Proof of a particular case of Theorem~\ref{t:unital}}    \label{sss:particular}
The following statement is a particular case of Theorem~\ref{t:unital} and of Corollary~\ref{c:unital}.
%Suppose that the stack $\CY$ carries an action of $\BA^1/G_m$. In this case 
%Theorem~\ref{t:unital} and Corollary~\ref{c:unital} are immediate consequences of 
%Lemma~\ref{l:key_lemma} and the following general considerations.

\begin{lem}    \label{l:particular}
Let $\CY$ be a QCA stack equipped with an action of the monoidal stack $\BA^1/\BG_m$.
Then the algebra ${\mathbf 0}_{\rD}\in\on{Funct}_{\on{cont}}(\Dmod (\CY ),\Dmod (\CY ))$ is unital and the 
coalgebra ${\mathbf 0}^!\in\on{Funct}_{\on{cont}}(\Dmod (\CY ),\Dmod (\CY ))$ is counital.
\end{lem}

This lemma is an immediate consequence of Lemma~\ref{l:key_lemma} and the following general considerations.

Suppose that a monoidal QCA stack $G$ acts on a QCA stack $\CY$. Then the monoidal category $G(k)$ 
acts on $\Dmod (\CY )$ on the left by $g\mapsto g_{\rD}$, $g\in G(k)$. One also has the right action
\footnote{Usually it does not commute with the left action. E.g., if $g\in G$ is invertible then $g^!=(g^{-1})_{\rD}$ 
does not have to commute with $g'_{\rD}$, $g'\in G$.} 
$g\mapsto g^!$. Each of these two actions extend to an action of $\Dmod (G )$. Namely, the left action is defined by
\begin{equation}
M*N:=a_{\rD}(M\boxtimes N),\quad\quad M\in\Dmod (G ), N\in\Dmod (\CY ),
\end{equation}
where $a:G\times\CY\to\CY$ is the action map.  One can get the right action of $\Dmod (G )$ on 
$\Dmod (\CY )$ from the left one using the equivalence $\Dmod(\CY)^\vee\simeq \Dmod(\CY)$ that comes from 
Verdier duality, see \eqref{e:Verdier}. One can also define the right action explicitly by
\begin{equation}
N*M:=(p_{\CY})_{\rD}\left(p_G^!(M)\sotimes a^!(N)\right),\quad\quad M\in\Dmod (G ), N\in\Dmod (\CY ),
\end{equation}
where $p_G:G\times\CY\to G$ and $p_\CY:G\times\CY\to\CY$ are the projections.

Now Lemma~\ref{l:particular} is clear. It immediately implies the following statement.
%, which is a particular case of Corollary~\ref{c:adjunctions}.

\begin{cor}   \label{c:particular}
Let $\CY$ be a QCA stack equipped with an action of the monoidal stack $\BA^1/\BG_m$.
Then the functors 
\begin{equation}
\pi_{\rD}:\Dmod (\CY )\rightleftarrows \Dmod (\CY^0 ):i_{\dr,*}\, ,\quad\quad 
i^!:\Dmod (\CY )\rightleftarrows\Dmod (\CY^0 ):\pi^!
\end{equation}
form adjoint pairs with the adjunctions  $\pi_{\rD}\circ i_{\dr,*}\iso\Id_{\Dmod (\CY^0 )}$ and 
$\,\Id_{\Dmod (\CY^0 )}\iso  i^!\circ \pi^!$ coming from the isomorphism $\pi\circ i\iso\Id_{\CY^0}$.
\end{cor}

\ssec{Proof of Theorems~\ref{t:unital} and \ref{t:adjunctions}}   \label{sss:monodromic_proof}
We will deduce them from Corollary~\ref{c:particular}. First, let us make some general remarks.

If $\CZ$ is an algebraic stack equipped with a morphism $\psi :\CZ\to B\BG_m$ then  $\Dmod (\CZ )$ is 
equipped with the following action of the tensor category $(\Dmod (B\BG_m),\overset{!}\otimes )$:
\[
M\otimes\CF:=\psi^!(M)\overset{!}\otimes\CF ,\quad\quad M\in\Dmod (B\BG_m),\quad \CF\in \Dmod (\CZ).
\]
If $f:\CZ_1\to\CZ_2$ is a morphism of QCA stacks over $B\BG_m$ then the functors $f_{\rD}$ and $f^!$ are 
compatible with the above action of $\Dmod (\BG_m)$.

\begin{lem}    \label{l:paradigm}
Suppose we have a Cartesian diagram of QCA stacks 
\[
\CD
\wt\CZ  @>>>  \on{pt} \\
@V{p}VV    @VV{\varphi}V   \\
\CZ  @>{\psi}>> B\BG_m
\endCD
\]
Then one has a canonical isomorphism
\begin{equation}   \label{e:nonequiv Hom}
\CMaps (p^!(\CF_1),p^!(\CF_2))=\CMaps (\CF_1,A\otimes\CF_2), \quad\quad  \CF_1,\CF_2\in\Dmod (\CZ ),
\end{equation}
where $A:=\varphi_{\dr,*} (k)[-2]$ and $A\otimes\CF_2:=\psi^!(A)\overset{!}\otimes\CF_2$.
\end{lem}

\begin{proof}
$p^!(\CF_1)=p^*_{\dr}(\CF_1)[2]$, so $$\CMaps (p^!(\CF_1),p^!(\CF_2))=\CMaps (\CF_1,p_{\dr,*}\circ p^!(\CF_2)[-2]) =\CMaps(\CF_1,A\otimes\CF_2).$$
\end{proof}

Now let us prove the assertion of  Theorem~\ref{t:adjunctions} concerning the pair $(\pi_{\rD}, i_{\dr,*})$.
The proof of the other assertion of Theorem~\ref{t:adjunctions} is similar, and Theorem~\ref{t:unital} follows 
from Theorem~\ref{t:adjunctions}.

%Now let us prove Theorem~\ref{t:adjunctions} (which implies Theorem~\ref{t:unital}).

%Without loss of generality, we can assume that $\CY$ is QCA (otherwise replace $\CY$ by
%$\CY\times_{\CY^0}S$, where $S$ is a quasi-compact scheme equipped with a smooth morphism to 
%$\CY^0$).

%Let us prove the assertion of  Theorem~\ref{t:adjunctions} concerning the pair $(\pi_{\rD}, i_{\dr,*})$
%(the proof of the other assertion is similar).

We have to show that
for any $\wt\CF_1\in\Dmod(\CY)_{\mmu}$ and $\wt\CF_2\in \Dmod(\CY^0)=\Dmod(\CY^0)_{\mmu}$ the canonical map
\begin{equation} \label{e:2nonequiv Hom}
\CMaps(\wt\CF_1,i_{\dr,*}(\wt\CF_2))\to 
\CMaps(\pi_{\rD}(\wt\CF_1),\pi_{\rD}\circ i_{\dr,*}(\wt\CF_2))=\CMaps(\pi_{\rD}(\wt\CF_1),\wt\CF_2)
\end{equation}
is an isomorphism. 
%Since the category $\Dmod(\CY/\BG_m)$ is compactly generated and the functor
%$$p^!:\Dmod(\CY/\BG_m)\to \Dmod(\CY),\quad\quad p:\CY\to\CY/\BG_m$$ sends compact objects to compact ones, we 
By the definition of the monodromic subcategory, we can assume 
that $\wt\CF_1=p^!(\CF_1)$ for some %compact 
$\CF_1\in \Dmod(\CY/\BG_m)$. %Then we 
Since the action of $\BG_m$ on $\CY^0$ is trivial, we can also assume that 
$\wt\CF_2=(p^0)^!(\CF_2)$ (here $p^0:\CY^0\to\CY^0/\BG_m$).
Applying Lemma~\ref{l:paradigm} for $\CZ=\CY/\BG_m$, $\wt\CZ=\CY$ and for 
$\CZ=\CY^0/\BG_m$, $\wt\CZ=\CY^0$ we get
$$\CMaps(\wt\CF_1,i_{\dr,*}(\wt\CF_2))=\CMaps (\CF_1,A\otimes  i'_{\dr,*}(\CF_2)), \quad 
i':\CY^0/\BG_m\to\CY/\BG_m ,$$
$$\CMaps(\pi_{\rD}(\wt\CF_1),\wt\CF_2)=\CMaps (\pi'_{\rD}(\CF_1),A\otimes \CF_2),
\quad  \pi' :\CY/\BG_m\to\CY^0/\BG_m.$$
%where $\tilde i:\CY^0/\BG_m\to\CY/\BG_m$ and $\tilde \pi :\CY/\BG_m\to\CY^0/\BG_m$ are the natural maps.
The map \eqref{e:2nonequiv Hom} is a particular case of the canonical map
\begin{equation}   \label{e:map_in_question}
\CMaps (\CF_1,\CM\otimes i'_{\dr,*}(\CF_2))\to\CMaps (\pi'_{\rD}(\CF_1),M\otimes \CF_2)
\end{equation}
which is defined for any $\CM\in\Dmod(B\BG_m)$. Applying Corollary~\ref{c:particular} to the action of
$\BA^1/\BG_m$ on $\CY/\BG_m$ we see that the map \eqref{e:map_in_question} is an isomorphism if 
$\CM=\omega_{B\BG_m}$. This implies that \eqref{e:map_in_question} is an isomorphism for any
$\CM\in\Dmod_{\on{coh}}(B\BG_m)$ (because by connectedness of $\BG_m$, $\Dmod_{\on{coh}}(B\BG_m)$ 
is the smallest non-cocomplete triangulated subcategory of 
$\Dmod (B\BG_m)$ containing $\omega_{B\BG_m}$). In particular, \eqref{e:map_in_question} is an isomorphism for $\CM=A$, and we are done.

\ssec{Sketch of another approach to Theorems~\ref{t:unital} and \ref{t:adjunctions}}
\label{sss:2monodromic_proof} 
In Sect.~\ref{sss:monodromic_proof} we deduced Theorems~\ref{t:unital} and \ref{t:adjunctions}
from Corollary~\ref{c:particular}, which relies on the study of $\Dmod(\BA^1/\BG_m)$
(see Lemma~\ref{l:key_lemma}). Here we sketch a slightly different approach, which is based on the study of $\Dmod(\BA^1)_{\mmu}$ and does not rely on Corollary~\ref{c:particular}.

\begin{prop}
The subcategories $\Dmod(\BG_m)_{\mmu}\subset\Dmod(\BG_m)$ and 
$\Dmod(\BA^1)_{\mmu}\subset\Dmod(\BA^1)$ are closed under convolution. Moreover, they are monoidal categories. The unit object of $\Dmod(\BG_m)_{\mmu}$ equals $\underline 1_{\mmu}$,
where $\underline 1$ is the unit object of $\Dmod(\BG_m)$ and 
\[
\CM\mapsto \CM_{\mmu}
\]
is the ``monodromization" functor $\Dmod(\BG_m)\to\Dmod(\BG_m)_{\mmu}$, i.e., the functor right adjoint to the embedding $\Dmod(\BG_m)_{\mmu}\mono\Dmod(\BG_m)$. The unit object of
$\Dmod(\BA^1)_{\mmu}$ has a similar description and can also be described as
$j_*(\underline 1_{\mmu})$, where $j:\BG_m\mono\BA^1$ is the embedding.   \qed
\end{prop}

The adjunction mentioned in the proposition defines a canonical morphism 
$\varepsilon :\underline 1_{\mmu}\to\underline 1$.

\begin{rem}
Let $\Gamma_{\dr}(\BG_m ,-)$ denote the de Rham cohomology functor $\Dmod(\BG_m)\to\Vect$.
The pair $(\underline 1_{\mmu},\varepsilon )$ is uniquely characterized by the following properties: $\underline 1_{\mmu}\in\Dmod(\BG_m)_{\mmu}$ and the map
$\Gamma_{\dr} (\BG_m ,\underline 1_{\mmu})\to\Gamma_{\dr} (\BG_m ,\underline 1)=k$ induced
by $\varepsilon$ is an isomorphism. This implies that $\underline 1_{\mmu}$ is nothing but the
``infinite Jordan block" $I^{-\infty,0}$ from \cite[Sect.~1.3]{Be}. In particular, 
%$\underline 1_{\mmu}[1]$ is a D-module (rather than a complex of D-modules), and 
the image of $\underline 1_{\mmu}$ under the Riemann-Hilbert correspondence is a sheaf
(rather than a complex of sheaves).
\end{rem}

Similarly to Lemma~\ref{l:key_lemma}, one has the following statement 
(which implies Lemma~\ref{l:key_lemma}).

\begin{lem}        \label{l:monodr_key_lemma}
The idempotent algebra $\underline 0=\underline 0_{\mmu}\in\Dmod (\BA^1)_{\mmu}$ is unital in $\Dmod (\BA^1)_{\mmu}$. \qed
\end{lem}

One shows that if $\CY$ is a QCA stack equipped with a $\BG_m$-action then
$\Dmod(\BG_m)_{\mmu}$ acts on $\Dmod(\CY)_{\mmu}$ as a monoidal category, i.e.,
$\underline 1_{\mmu}$ acts as identity. Similarly, if $\CY$ is equipped with an $\BA^1$-action then
one has the left and right monoidal action of %the monoidal category 
$\Dmod(\BA^1)_{\mmu}$ on $\Dmod(\CY)_{\mmu}$. Now Theorem~\ref{t:unital} follows from Lemma~\ref{l:monodr_key_lemma}, and Theorem~\ref{t:adjunctions} follows from \ref{t:unital}.

\end{document}